\documentclass[reqno,11pt]{amsart}

\usepackage{xcolor}
\usepackage{graphicx}
\usepackage[hidelinks]{hyperref}
\usepackage{amscd}

\usepackage{amsmath}
\usepackage{amsthm}
\usepackage{amssymb}
\usepackage{latexsym,array}
\usepackage{amsfonts}
\usepackage{shadow}
\usepackage{amsbsy}
\usepackage{amssymb}
\usepackage{dsfont}
\usepackage{mathtools}
\usepackage{enumitem}
\usepackage{doi}

\usepackage[notref, notcite, final]{showkeys}
\usepackage{color}
\usepackage{graphicx}
\usepackage[hidelinks]{hyperref}
\usepackage{cleveref}
\usepackage{fullpage}
\usepackage{comment}
\usepackage{tikz-cd}
\usepackage{tikz}
\usepackage{float}

\usepackage{dsfont}

\usepackage{cleveref}

\numberwithin{equation}{section}

\newcommand{\bigD}{\mathcal{D}}

\newcommand{\qcs}{\mathcal Q_{c,s}(x)}

\newcommand{\PRes}{\mathcal{Q}}

\newtheorem{theorem}{Theorem}[section]
\newtheorem{lemma}[theorem]{Lemma}
\newtheorem{corollary}[theorem]{Corollary}
\newtheorem{proposition}[theorem]{Proposition}
\newtheorem{prob}[theorem]{Problem}

\theoremstyle{definition}
\newtheorem{remark}[theorem]{{\bf Remark}}
\newtheorem{definition}[theorem]{Definition}

\newcommand{\cc}{\mathbb{C}}

\newcommand{\pp}{\partial}
\newcommand{\rr}{\mathbb{R}}

\renewcommand{\Re}{\mathrm{Re}}

\newcommand{\dr}{\partial_r}
\newcommand{\dx}{\partial_{x_0}}

\crefname{enumi}{}{}
\crefname{enumii}{}{}

\title[The fine structure of the spectral theory on the $S$-spectrum in dimension five ]{The fine structure of the spectral theory\\ on the $S$-spectrum in dimension five }

\author[F. Colombo]{Fabrizio Colombo}
\address{(FC)
Politecnico di Milano\\Dipartimento di Matematica\\Via E. Bonardi, 9\\20133
Milano, Italy}
\email{fabrizio.colombo@polimi.it}

\author[A. De Martino]{Antonino De Martino}
\address{(ADM)
 Politecnico di Milano\\Dipartimento di Matematica\\Via E. Bonardi, 9\\20133
Milano, Italy
} \email{antonino.demartino@polimi.it}

\author[S. Pinton]{Stefano Pinton}
\address{(SP)
Politecnico di Milano\\Dipartimento di Matematica\\Via E. Bonardi, 9\\20133
Milano, Italy
} \email{stefano.pinton@polimi.it}

\author[I. Sabadini]{Irene Sabadini}
\address{(IS)
Politecnico di Milano\\Dipartimento di Matematica\\Via E. Bonardi, 9\\20133
Milano, Italy
} \email{irene.sabadini@polimi.it}

\date{}

\begin{document}

\maketitle
\begin{abstract}
Holomorphic functions play a crucial role in operator theory and the Cauchy formula is a very important tool to define functions of operators. The Fueter-Sce-Qian extension theorem is a two steps procedure to extend holomorphic functions to the hyperholomorphic setting.
The first step gives the class of slice hyperholomorphic functions;
their Cauchy formula allows to define the so-called $S$-functional calculus for noncommuting operators  based on the $S$-spectrum.
In the second step this extension procedure generates monogenic functions; the
related monogenic functional calculus, based on the monogenic spectrum, contains the Weyl functional calculus as a particular case.
In this paper we show that the extension operator from slice hyperholomorphic
functions to monogenic functions admits various possible factorizations
that induce different function spaces.
The integral representations in such spaces allows to define the associated functional calculi based on the $S$-spectrum.
The function spaces and the associated functional calculi define the so called
{\em fine structure of the spectral theories on the $S$-spectrum}.
Among the possible fine structures there are the harmonic and poly-harmonic functions and the associated harmonic and poly-harmonic functional calculi. The study of the fine structures depends on the dimension considered and in this paper we study in detail the case of dimension five, and we describe all of them.
The five-dimensional case is of crucial importance
because it allows to determine almost all the function spaces 
will also appear in dimension greater than five, but with different orders.

\end{abstract}

\medskip
\noindent AMS Classification 47A10, 47A60.


\tableofcontents
\section{Introduction}\label{INTRO}

The spectral theory on the $S$-spectrum for quaternionic operators has been widely developed in the last 15 years motivated by the precise formulation quaternionic quantum mechanics (see \cite{BF,JONAQS}), and other applications has been found more recently,
among which we mention the fractional powers of vectors operators that are useful in fractional diffusion problems, see
\cite{frac4,frac5,frac1}.
The main references on quaternionic spectral theory on the $S$-spectrum
 are the books \cite{6COFBook,6ACSBOOK,FJBOOK,CGKBOOK,6JONAME} and the references therein.
The spectral theory on the $S$-spectrum in the Clifford algebra setting, see \cite{6css}, started in parallel with the quaternionic one,
but in the last few years there have been new and unexpected developments.
 In fact, even if a precise version of the quaternionic
 spectral theorem on the $S$-spectrum was expected and proved in \cite{6SpecThm1}
 (for perturbation results see also \cite{6CCKS}), it was only in recent times that the spectral theorem
 for fully Clifford operators was proved, see
\cite{CLIFST}. Moreover,  the validity of the $S$-functional calculus was extended beyond the Clifford algebra setting, see \cite{UNIV},
and it was used to the define slice monogenic functions of a Clifford variable in \cite{ClifFUN}.

\medskip
This paper belongs a new research direction that is related to
the Fueter-Sce-Qian mapping theorem.
To illustrate our results we briefly recall this theorem.
Roughly speaking, it is a two steps procedure that extends
holomorphic functions of one complex variable to the hypercomplex setting.
In the first step it generates slice hyperholomorphic functions
and in the second step it generates monogenic functions, i.e., functions in the kernel of the Dirac operator.
More precisely,
let $\mathcal{O}(D)$ be the set of holomorphic functions
on $D\subseteq \mathbb{C}$ and let $\Omega_D\subseteq \mathbb{R}^{n+1}$ be the set induced by $D$ (see Section \ref{FSPAC}).
The first  Fueter-Sce-Qian map $T_{FS1}$ applied to $\mathcal{O}(D)$ generates the set ${\mathcal{SH}(\Omega_D)}$ of slice monogenic functions on $\Omega_D$
(which turn out to be intrinsic) and
the second  Fueter-Sce-Qian map $T_{FS2}$ applied to ${\mathcal{SH}(\Omega_D)}$ generates
 axially monogenic functions on $\Omega_D$. We denote this second class of functions
by $\mathcal{AM}(\Omega_D)$.
The extension procedure is illustrated in the diagram:
\begin{equation*}
\begin{CD}
\textcolor{black}{\mathcal{O}(D)}  @>T_{FS1}>> \textcolor{black}{\mathcal{SH}(\Omega_D)}  @>\ \   T_{FS2}=\Delta^{(n-1)/2} >>\textcolor{black}{\mathcal{AM}(\Omega_D)},
\end{CD}
\end{equation*}
where $T_{FS2}=\Delta^{(n-1)/2} $ and $\Delta$ is the Laplace operator in dimension $n+1$ (sometimes we shall write $\Delta_{\mathbb{R}^{n+1}}$ to specify the dimension).
The Fueter-Sce-Qian mapping theorem induces two spectral theories according to the two classes
of hyperholomorphic functions it generates: using the Cauchy formula of slice hyperholomorphic functions one defines the $S$-functional calculus, which is based on the $S$-spectrum, while using the Cauchy formula of monogenic functions one obtains the monogenic functional calculus, based on the monogenic spectrum.
This last calculus was introduced by  A. McIntosh and his collaborators,
see \cite{JM}), to define functions of
noncommuting operators on Banach spaces; it has several applications,
as discussed in the books \cite{J,TAOBOOK}.

\medskip
The Fueter-Sce mapping theorem (when $n$ is odd provides an
alternative way to define the monogenic functional calculus.
The main idea is to apply the Fueter-Sce operator $T_{FS2}$ to the slice hyperholomorphic Cauchy kernel and to write the Fueter-Sce mapping theorem in integral form. Using this integral formulation, we can define the so-called $F$-functional calculus, which
 is a monogenic functional calculus, but it is based on the $S$-spectrum.
In diagram form, we have:
\begin{equation*}
\begin{CD}
{\mathcal{SH}(\Omega_D)} @.  {\mathcal{AM}(\Omega_D)} \\   @V  VV
  @.
\\
{{\rm  Slice\ Cauchy \ Formula}}  @> T_{FS2}=\Delta^{(n-1)/2}>> {{\rm Fueter-Sce\ theorem \ in \  integral\  form}}
\\
@V VV    @V VV
\\
{S-{\rm  functional \ calculus}} @. F-{{\rm functional \ calculus}}
\end{CD}
\end{equation*}
 Observe that in the above diagram the arrow from the space of axially monogenic function $\mathcal{AM}(\Omega_D)$ is missing because the $F$-functional calculus is deduced from the slice hyperholomorphic Cauchy formula.

\medskip
We are now in the position to define the fine structure of the spectral theories on the $S$-spectrum
taking advantage of the following observation.
Let $h:=\frac{n-1}{2}$ be the so-called Sce exponent, and $\Delta$ be the Laplace operator in dimension $n+1$: the operator $T_{FS2}:= \Delta^{h}$  maps the slice hyperholomorphic function $f(x)$ to the monogenic function $\breve{f}(x)$ given by
$$
\breve{f}(x)= \Delta^hf(x),\ \ \ \ x\in \Omega_D.
$$
If we denote by $e_\ell$, $\ell=1,...,n$ the units of the Clifford algebra $\mathbb{R}_n$
the Dirac operator $\mathcal{D}$ and its conjugate $\overline{\mathcal{D}}$ are defined by
$$
\mathcal{D} :=\partial_{x_0} +\sum_{i=1}^n e_i
\partial_{x_i}, \ \ \ \overline{\mathcal{D}}=\partial_{x_0} -\sum_{i=1}^n e_i
\partial_{x_i}.
$$
The powers of the Laplace operator $\Delta^h$ can be factorized in terms of the
Dirac operator $ \mathcal{D}$ and its conjugate
 $\mathcal{\overline{D}}$ because
   $$
   \mathcal{D}\mathcal{\overline{D}}=\mathcal{\overline{D}}\mathcal{D}=\Delta.
   $$
   So it is possible to repeatedly apply to a slice hyperholomorphic function $f(x)$
   the Dirac operator and its conjugate, until we reach the maximum power of the Laplacian, i.e., the Sce exponent.
 This implies the possibility to build different sets of functions which lie between the set of slice hyperholomorphic functions and the set of axially monogenic functions.

\medskip
We will call {\em fine structure of the  spectral theory on the $S$-spectrum}
 the set of the functions spaces and the associated functional calculi
 induced by a factorization of the operator $T_{FS2}$ in the Fueter-Sce extension theorem.

\medskip
One of the most important factorizations leads to the so-called
 Dirac fine structure that corresponds to an alternating sequence
of $ \mathcal{D}$ and $ \mathcal{\overline{D}}$, $h$ times, for example
$$ T_{FS2}=\Delta_{\mathbb{R}^{n+1}}^{h}= \mathcal{D} \mathcal{\overline{D}}\ldots \mathcal{D}\mathcal{\overline{D}}.$$

\medskip
The fine structure is also defined for the Fueter-Sce-Qian extension theorem, namely when $n$ is even, but it involves the fractional powers of the
 Laplace operator and will be treated in a future publication.

\medskip
The fine structure of the spectral theory on the $S$-spectrum
generates, in a unified way, several classes of functions,
some of which have already been studied in the literature.
We can summarise them by defining  polyanalytic holomorphic Cliffordian functions of order $(k, \ell)$ which are defined as follows.

\medskip
Let $U$ be an open set in $\mathbb{R}^{n+1}$. A function $f:U \subset \mathbb{R}^{n+1} \to \mathbb{R}_n$ of class $\mathcal{C}^{2k+ \ell}(U)$ is said to be (left) polyanalytic holomorphic Cliffordian of order $(k, \ell)$ if
	$$ \Delta^k \mathcal{D}^{\ell}  f(x)=0\quad \forall x\in U,$$
	where $0 \leq k \leq \frac{n-1}{2}$ and $ \ell \geq 0$.

\medskip
Clearly,  polyharmonic functions of degree $k$ are a particular case of polyanalytic holomorphic Cliffordian functions.
In fact, a  function $f:U \subset \mathbb{R}^{n+1} \to \mathbb{R}_n$ of class $\mathcal{C}^{2k}(U)$ is called polyharmonic of degree $k$ in the open set $U \subset \mathbb{R}^{n+1}$ if
	$$ \Delta^k f(x)=0, \quad\forall x \in U.$$
Analogously,  polyanalytic functions of order $m$ are contained in the class of polyanalytic holomorphic Cliffordian functions, since they are those
functions $f$ defined on an open set $ U \subset \mathbb{R}^{n+1}$ with values in $\mathbb{R}_n$  of class $ \mathcal{C}^{m}(U)$ such that
	$$ \mathcal{D}^{m}f(x)=0, \qquad \forall x \in U.$$

\medskip

In \cite{CDPS} we studied the Dirac fine structure when $n=3$, i.e., the quaternionic case.
In particular we have studied the fine structure associated with the factorization:
\begin{equation}
\label{fine1}
\mathcal{O}(D) \overset{T_{FS1}}{\longrightarrow} \mathcal{SH}(\Omega_D)\overset{\mathcal{D}}{\longrightarrow} \mathcal{AH}(\Omega_D)\overset{\mathcal{\overline{D}}}{\longrightarrow}\mathcal{AM}(\Omega_D),
\end{equation}
where $\mathcal{AH}(\Omega_D)$ is the set of axially harmonic functions
and their integral representation give rise to the harmonic functional calculus on the $S$-spectrum.
This structure also allows to obtain a product formula for the $F$-functional calculus, see \cite[Thm. 9.3]{CDPS}.

However, since $\Delta=\mathcal{D} \mathcal{\overline{D}}= \mathcal{\overline{D}}\mathcal{D}$, we can interchange the order of the operators $ \mathcal{D}$ and $ \mathcal{\overline{D}}$ in \eqref{fine1}. This gives rise to the factorization:
\begin{equation}
\mathcal{O}(D) \overset{T_{FS1}}{\longrightarrow} \mathcal{SH}(\Omega_D)\overset{\mathcal{\overline{D}}}{\longrightarrow} \mathcal{AP}_2(\Omega_D)\overset{\mathcal{D}}{\longrightarrow}\mathcal{AM}(\Omega_D).
\end{equation}
where $\mathcal{AP}_2(\Omega_D)$ is a space of polyanalytic functions. This structure is investigated in \cite{CDPS1} together with its functional calculus.

\medskip
Clearly, as the dimension of the Clifford algebra increases
there are more possibilities and we denote by
 $\mathcal{FS}(\Omega_D)$, the set of function
spaces associated with the fine structures.
These
 functions spaces  lie between the set of slice hyperholomorphic functions and axially monogenic functions and in dimension five there are seven such spaces, precisely:
$ \mathcal{ABH}(\Omega_D)$ the axially bi-harmonic functions,
$ \mathcal{ACH}_1(\Omega)$ the axially Cliffordian holomorphic functions of order 1,
(which is a short cut for order $(1,1)$),
$ \mathcal{AH}(\Omega_D)$ the axially harmonic functions,
$ \mathcal{AP}_2(\Omega_{D})$ the axially polyanalytic of order $2$,
$ \mathcal{ACH}_1(\Omega_D)$ the axially anti Cliffordian of order $1$,
$ \mathcal{ACP}_{(1,2)}$ the axially polyanalytic Cliffordian of order $(1,2)$,
$ \mathcal{AP}_3(\Omega_{D})$ the axially polyanalytic of order $3$,
and they will be defined precisely in the sequel.

\medskip
It is very important to point out that these function spaces appear in different contexts
in the literature and they seem to be unrelated. In this paper we show that they all appear as fine structures in the Fueter-Sce
construction.
In dimension greater than five there will be one more function
space, that is not indicated in the list above, and with this addition, all the fines structures can be described using those function spaces of different orders.

\medskip
In dimension five there are different fine structures of Dirac type
and to indicate the type of fine structure we put
into an array the Dirac $\mathcal{D}$ and conjugate Dirac $\overline{\mathcal{D}}$ operator that define the specific fine structure.
For example, the  Dirac fine structure of the type $(\mathcal{D}, \mathcal{\overline{D}}, \mathcal{D}, \mathcal{\overline{D}} )$ is given by:
\begin{equation}
\mathcal{O}(D) \overset{T_{FS1}}{\longrightarrow} \mathcal{SH}(\Omega_D)\overset{\mathcal{D}}{\longrightarrow} \mathcal{ABH}(\Omega_D)\overset{\mathcal{\overline{D}}}{\longrightarrow}\mathcal{AHC}_1(\Omega_D) \overset{\mathcal{D}}{\longrightarrow} \mathcal{AH}(\Omega_D) \overset{\mathcal{\overline{D}}}{\longrightarrow} \mathcal{AM}(\Omega_D),
\end{equation}
where the spaces are precisely specified in the sequel.
The Dirac fine structure of type $(\mathcal{D}, \mathcal{\overline{D}}, \mathcal{D}, \mathcal{\overline{D}} )$ is of
crucial importance to prove the product rule for the $F$-functional calculus (see Theorem \ref{prodo}).
However, by rearranging the sequence of $ \mathcal{D}$ and $ \mathcal{\overline{D}}$ it is possible to obtain other structures,
in which other sets of functions are involved.

\medskip
Another interesting  fine structure is the harmonic one, in which only the harmonic $\mathcal{AH}(\Omega_D)$ and bi-harmonic $\mathcal{ABH}(\Omega_D)$ sets of functions  appear:
$$ \mathcal{SH}(\Omega_D)\overset{\mathcal{D}}{\longrightarrow}  \mathcal{ABH}(\Omega_D) \overset{\Delta}{\longrightarrow}  \mathcal{AH}(\Omega_D)\overset{\overline{\mathcal{D}}}{\longrightarrow} \mathcal{AM}(\Omega_D).$$

The integral  representation of the functions in $\mathcal{AH}(\Omega_D)$ and $\mathcal{ABH}(\Omega_D)$ is obtained
by applying the operators in the above sequence to the
Cauchy kernels of slice hyperholomorphic functions.
 From these integral formulas
  we define the related functional calculus on the $S$-spectrum that can be visualized by the following diagram:
\begin{equation*}
\begin{CD}
{\mathcal{SH}(U)}  @. \mathcal{ABH}(U)  @. \mathcal{AH}(U)  @.{\mathcal{AM}(U)} \\   @V  VV
  @.
\\
{{\rm  Cauchy\ Formula}}  @> \mathcal{D}>> {{\rm \mathcal{ABH}\ \mathrm{Int.\ Form}}} @> \Delta>> {{\rm
\mathcal{AH}\ \mathrm{Int.\ Form}}}@> \overline{\mathcal{D}}>> {{\rm \mathcal{AM}\ \mathrm{Int.\ Form}  }}
\\
@V VV    @V VV  @V VV    @V VV
\\
{S-{\rm  \mathrm{Func.\ Cal.} }} @. {{\rm \mathcal{ABH}-\mathrm{Func.\ Cal.} }}@. {{\rm \mathcal{AH} -\mathrm{Func.\ Cal.}}}@. {F-{\rm \mathrm{Func.\ Cal.}}}
\end{CD}
\end{equation*}
where the functional calculi of this fine structure are the
 bi-harmonic $\mathcal{ABH}$-functional calculus and the harmonic $\mathcal{AH}$-functional calculus,
both based on the $S$-spectrum.
Note that the integral representation of the axially monogenic functions
$\mathcal{AM}$ is already known from the Fueter-Sce mapping theorem
in integral form and its functional calculus is the $F$-functional calculus, see \cite{CSS10}.

\medskip
To be more precise, let us consider
 $\mathbb{S}$, the sphere of purely imaginary paravectors with modulus 1.
Observe that given an element $x=x_0+\underline{x}\in\rr^{n+1}$ we can put
$
J_x=\underline{x}/|\underline{x}|$ if $\underline{x}\not=0,
$
 and given an element $x\in\rr^{n+1}$, the set
$$
[x]:=\{y\in\rr^{n+1}\ :\ y=x_0+J |\underline{x}|, \ J\in \mathbb{S}\}
$$
is an $(n-1)$-dimensional sphere in $\mathbb{R}^{n+1}$.
In order to give a glimpse  on the integral representation of the harmonic and bi-harmonic functions
we recall that, for $s,\, x\in \mathbb R^{n+1}$ with $x\notin [s]$, the function
	$$
	\PRes_{c,s}(x)^{-1}:=(s^2-2\Re(x)s+|x|^2)^{-1},
	$$
	and the  left slice hyperholomorphic Cauchy kernel  $S^{-1}_L(s,x)$ defined by
		$$S^{-1}_L(s,x):=(s-\overline x) \PRes_{c,s}(x)^{-1}.$$

\medskip
Let $n=5$ and consider,
for  $s,x \in \mathbb{R}^6$ such that $x \notin [s]$ the functions
\begin{equation}
		S^{-1}_{\bigD ,L}(s,x):=-4 \mathcal{Q}_{c,s}(x)^{-1},
	\end{equation}
	 and
	\begin{equation}
S^{-1}_{\Delta\bigD ,L}(s,x):=\Delta \mathcal{D} \left(S^{-1}_L(s,x)\right)=16 \mathcal{Q}_{c,s}(x)^{-2}.
	\end{equation}
called the left slice $ \mathcal{D}$-kernel and the left slice  $ \Delta\mathcal{D}$-kernel, respectively.
Let $W \subset \mathbb{R}^6$ be an open set and $U$ be a slice Cauchy domain such that $\bar{U} \subset W$.
Then for $J \in \mathbb{S}$ and $ds_J=ds(-J)$, for  $f \in \mathcal{SH}_L(W)$,
 the function $ \tilde{f}_\ell(x):=\Delta^{1-\ell} \mathcal{D}f(x)$, $0 \leq \ell  \leq 1$, is $\ell+1$-harmonic and the following integral representation
		$$ \tilde{f}_\ell(x)= \frac{1}{2\pi} \int_{\partial(U \cap \mathbb{C}_J)} S^{-1}_{\Delta^{1-\ell}\mathcal D,L}(s,x) ds_J f(s),$$
holds, where
the integrals depend neither on U nor on the imaginary unit $J$.
If $f \in \mathcal{SH}_R(W)$
a similar integral representation holds.

\medskip
We now illustrate the functional calculi of the harmonic fine structure from the operator theory point of view.
We denote by
$\mathcal{B}(V)$ is the space of all bounded $\mathbb{R}$-linear operators
and $\mathcal{B}(V_n)$ is the space of all bounded $\mathbb{R}_n$-linear operators
on $V_n=\mathbb{R}_n\otimes V$.
The $S$-spectrum is defined as
$$
\sigma_S(T)=\{ s\in \mathbb{R}^{n+1}\ : \ T^2-2s_0T+|s|^2\mathcal{I}\ \ {\rm is \ not\ invertible\ in}\ \  \mathcal{B}(V_n)\}.
$$
This definition of the spectrum is used for operators in $\mathcal{B}(V_n)$ with noncommuting components.
For the fine structure of the spectral theories on the $S$-spectrum
in $\mathcal{B}(V_n)$ we will consider bounded paravector operators
$T=e_0T_0+e_1T_1+...+e_n T_n$, with commuting components $T_\ell\in\mathcal{B}(V)$ for $\ell=0,1,\ldots ,n$
and we denote this set by $\mathcal{BC}^{\small 0,1}(V_n)$.
In this case the most appropriate definition of the $S$-spectrum is its commutative version (also called $F$-spectrum), i.e.,
$$
\sigma_F(T)=\{ s\in \mathbb{R}^{n+1}\ \ :\ \ s^2\mathcal{I}-(T+\overline{T})s +T\overline{T}\ \ \
{\rm is\ not\  invertible\ in \ }\mathcal{B}(V_n)\}
$$
where the operator $\overline{T}$ is defined by
$$
\overline{T}=T_0-T_1e_1 - \dots  - T_n e_n.
$$
Note that it has been be proved that
$$
\sigma_F(T)=\sigma_S(T) \ \  {\rm for\ all} \ \  T\in \mathcal{BC}^{\small 0,1}(V_n).
$$

\medskip
Let $T\in \mathcal {BC}^{0,1}(V_n)$ and
recall that the $S$-resolvent set is defined as
$$
 \rho_S(T):=\mathbb{R}^{n+1}\setminus \sigma_S(T).
$$
Then, the commutative pseudo $SC$-resolvent operator is:
$$
\mathcal Q_{c,s}(T)^{-1}:=(s^2\mathcal I-s(T+\overline T)+T\overline T)^{-1} \ \ {\rm for}\ \ \  s\in \rho_S(T)
$$
while the left $SC$-resolvent operator is:
	$$ S^{-1}_{L}(s,T):=(s\mathcal I-\overline T)(s^2\mathcal I-s(T+\overline T)+T\overline T)^{-1} \ \ {\rm for}\ \ \  s\in \rho_S(T), $$
	and similarly we have the right $SC$-resolvent operator.

\medskip
In the case of dimension $n=5$ we
take $T \in \mathcal{BC}^{0,1}(V_5)$ and set $ds_J=ds(-J)$ for $J \in \mathbb{S}$.
Let $f$ be a function in $\mathcal{SH}_L(\sigma_S(T))$
(and similarly for $f\in\mathcal{SH}_R(\sigma_S(T))$).
Let $U$ be a bounded slice Cauchy domain with $\sigma_S(T)\subset U$ and $\overline U\subset\operatorname{dom}(f)$.
We have the following definitions of the $\ell+1$-harmonic functional calculus for $\ell=0,1$:
	  for every function $\tilde{f}_\ell=\Delta^{1-\ell}\mathcal{D}f$ with $f \in \mathcal{SH}_L(\sigma_S(T))$ we set
		\begin{equation}
			\tilde{f}_\ell(T):=\frac 1{2\pi} \int_{\partial(U \cap \mathbb{C}_J)} S^{-1}_{\Delta^{1-\ell}\bigD,L}(s,T) ds_J f(s),
		\end{equation}
where:
\\
(I)
the left $\mathcal{D}$-resolvent operator $S^{-1}_{\bigD ,L}(s,T)$ is defined as
	\begin{equation}
		S^{-1}_{\bigD ,L}(s,T):=-4 \mathcal{Q}_{c,s}(T)^{-1}  \ \ {\rm for}\ \ \  s\in \rho_S(T),
	\end{equation}
\\
(II)
 the
left $\Delta\bigD$-resolvent operator $S^{-1}_{\Delta \bigD,L}(s,T)$ is defined as
		\begin{equation}
		S^{-1}_{\Delta\bigD ,L}(s,T):=16 \mathcal{Q}_{c,s}(T)^{-2}, \ \ {\rm for}\ \ \  s\in \rho_S(T).
	\end{equation}
	We point out that similar formulas hold for the right case.

\medskip
{\it Plan of the paper.}
The paper consists of 10 sections, besides this introduction.

In Section \ref{FSPAC} we introduce some functions spaces
that naturally arise from the factorization of the second Fueter-Sce operator $T_{FS2}$.
These definitions are valid for any dimension $n$.

Section \ref{FACTRFIVE} contains the factorization of the Fueter-Sce operator  $T_{FS2}$ in dimension $n=5$
and, in particular, we define the function spaces of axial type.

In Section \ref{section4} we introduce the concept of  fine structure of the spectral theory on the $S$-spectrum,
that is associated with the functions spaces, and the related functional calculi.
For $n=5$ we can fully describe the fine structures and in the appendix
 we provide a description with diagrams.

In section \ref{VEKUA} we study the systems of differential equations for the fine structure spaces of axial type, i.e.,
in  analogy with the Vekua-type system of differential equations for axially monogenic functions
we explicitly we give all the systems of differential equations for the fine structure spaces in dimension $n=5$.

In section \ref{INTREP} we give the integral representation of the functions of the fine structure spaces.
Precisely, we recall the slice hyperholomorphic Cauchy formulas
and applying the operators associated with the fine structure to the slice hyperholomorphic Cauchy kernels we
deduce the integral representations.

In Section \ref{serieskernels} we prove some technical lemmas to write the explicit
series expansion of the kernels of the fine structures spaces.
These results are of crucial importance for operator theory because they allow to define the series expansions of the resolvent
operators of the fine structures.

Then, after some preliminary results on the $SC$-functional calculus and the $F$-functional calculus  collected in Section \ref{prelSCFFC}, in Section \ref{FCFSTRU} we define the functional calculi of the fine structure. To this end,
 we define the
series expansions of the resolvent operators of the  functional calculi using the results in Section \ref{serieskernels}.
Observe that all the resolvent operators associated with fine structures involve the $S$-spectrum.

We conclude the paper with an application to the $F$-functional calculus, indeed in Section \ref{PRDRULE} we prove
the product rule for the $F$-functional calculus using the Dirac fine structure of the kind $(\mathcal{D}, \mathcal{\overline{D}}, \mathcal{D}, \mathcal{\overline{D}} )$.

\section{Function spaces generated by the Fueter-Sce mapping theorem}\label{FSPAC}

This section contains some function spaces of Clifford algebra valued functions and
some notions of the spectral theory on the $S$-spectrum.
We start by fixing some notations.
 Let $\mathbb{R}_n$ be the real Clifford algebra over $n$ imaginary units $e_1,\ldots ,e_n$
satisfying the relations $e_\ell e_m+e_me_\ell=0$,\  $\ell\not= m$, $e_\ell^2=-1.$
An element in the Clifford algebra will be denoted by $\sum_A e_Ax_A$ where
$A=\{ \ell_1\ldots \ell_r\}\in \mathcal{P}\{1,2,\ldots, n\},\ \  \ell_1<\ldots <\ell_r$
is a multi-index
and $e_A=e_{\ell_1} e_{\ell_2}\ldots e_{\ell_r}$, $e_\emptyset =1$.
A point $(x_0,x_1,\ldots,x_n)\in \mathbb{R}^{n+1}$  will be identified with the element
$
x=x_0+\underline{x}=x_0+ \sum_{j=1}^n x_j e_j\in\mathbb{R}_n
$
called paravector and the real part $x_0$ of $x$ will also be denoted by $\Re(x)$.
The vector part of $x$ is defined by
$\underline{x}=x_1e_1+\ldots+x_ne_n$.
The conjugate of $x$ is denoted by $\overline{x}=x_0-\underline{x}$
and the Euclidean modulus of $x$ is given by $|x|=(x_0^2+\ldots +x_n^2)^{1/2}$.
The sphere of purely imaginary paravectors with modulus 1, is defined by
$$ \mathbb{S}:= \{\underline{x}= e_{1}x_1+ \ldots+e_n e_n\, | \, x_1^2+ \ldots +x_n^2=1\}.$$
Notice that if $J \in \mathbb{S}$, then $J^2=-1$. Therefore $J$ is an imaginary unit, and we denote by
$$ \mathbb{C}_J= \{u+Jv \, | \, u,v \in \mathbb{R}\},$$
an isomorphic copy of the complex numbers.

In order to give the definition of slice hyperholomorphic functions we need to
define the natural domains on which these functions are defined.

\begin{definition}
	Let $U \subseteq \mathbb{R}^{n+1}$.
	\begin{itemize}
		\item We say that $U$ is axially symmetric if, for every $u+Iv \in U$, all the elements $u+Jv$ for $J \in \mathbb{S}$ are contained in $U$.
		\item We say that $U$ is a slice domain if $U \cap \mathbb{R} \neq \emptyset$ and if $U \cap \mathbb{C}_J$ is a domain in $\mathbb{C}_J$ for every $J \in \mathbb{S}$.
	\end{itemize}
	
\end{definition}
\begin{definition}
	An axially symmetric open set $U \subset \mathbb{R}^{n+1}$ is called slice Cauchy domain if $U \cap \mathbb{C}_J$ is a Cauchy domain in $ \mathbb{C}_J$ for every $J \in \mathbb{S}$. More precisely, $U$ is a slice Cauchy domain if for every $J \in \mathbb{S}$ the boundary of $ U \cap \mathbb{C}_J$ is the union of a finite number of nonintersecting piecewise continuously differentiable Jordan curves in $ \mathbb{C}_J$.
\end{definition}
On axially symmetric open sets we define the class of slice hyperholomorphic functions, in the case of Clifford algebra valued functions they are often called slice monogenic functions.

\begin{definition}[Slice hyperholomorphic (or slice monogenic) functions]
	\label{hyper}
	Let $U \subseteq \mathbb{R}^{n+1}$ be an axially symmetric open set and let $ \mathcal{U}:=\{(u,v)\in\rr^2:\, u+Jv\in U\quad\forall J\in\mathbb S\} $. We say that a function $f: U \to \mathbb{R}_n$ of the form
	\begin{equation}\label{slicefunc}
 f(x)=\alpha(u,v)+J\beta(u,v)
 \end{equation}
	where $x=u+Jv$ for any $J\in\mathbb S$, is left slice hyperholomorphic if $\alpha$ and $\beta$ are $ \mathbb{R}_n$-valued differentiable functions such that
	\begin{equation}\label{eveodd}
		\alpha(u,v)=\alpha(u,-v), \ \ \ \  \beta(u,v)=-\beta(u,-v) \ \\ \ \  \hbox{for all} \, \, (u,v) \in \mathcal{U},
	\end{equation}
	and if $\alpha$ and $\beta$ satisfy the Cauchy-Riemann system
	$$ \partial_u \alpha(u,v)- \partial_v \beta(u,v)=0, \quad \partial_v \alpha(u,v)+ \partial_u \beta(u,v)=0.$$
	We recall that right slice hyperholomorphic functions are of the form
	$$ f(x)= \alpha(u,v)+\beta(u,v)J,$$
	where $\alpha$, $\beta$ satisfy the above conditions.
\end{definition}
\begin{definition}
	The set of left (resp. right) slice hyperholomorphic functions on $U$ is denoted with the symbol $\mathcal{SH}_{L}(U)$ (resp. $\mathcal{SH}_{R}(U)$). The subset of intrinsic functions consists of those slice hyperholomorphic functions such that $\alpha$, $\beta$ are real-valued function and it is denoted by $ \mathcal{N}(U)$.
\end{definition}
We introduce the monogenic functions.
\begin{definition}[monogenic functions]\label{d2}
	Let $U\subset \mathbb R^{n+1}$ be an open set. A real differentiable function $ f: U \to \mathbb{R}_n$ is called left monogenic if
	$$\mathcal{D} f (x):=\big(\partial_{x_0} +\sum_{i=1}^n e_i
\partial_{x_i}\big)f(x)=0.
$$
\end{definition}
In a similar way we define right monogenic functions.

\medskip
There are several possible definitions of slice hyperholomorphicity, that are not fully equivalent, but Definition \ref{hyper} of slice hyperholomorphic functions is the most appropriate for the operator theory; it comes from the extension of the Fueter mapping theorem from the quaternionic setting to the Clifford and real alternative algebra setting (see \cite{F}, \cite{GP}, \cite{Q1}, \cite{S}; for an English translation of paper \cite{S} see \cite{CSS3}).
In this paper we need Sce theorem for our considerations and we recall it below.
\begin{theorem}[Sce theorem, see \cite{S}]
	\label{FS}
	Let $n \geq 3$ be an odd number. Let $f(z)= \alpha(u,v)+i \beta(u,v)$ be a holomoprhic function defined in a domain (open and connected) $D$ in the upper-half complex plane and let
	$$ \Omega_D:= \{x=x_0+ \underline{x} \,\  |\  \, (x_0, |\underline{x}|) \in D\},$$
	be the open set induced by $D$ in $ \mathbb{R}^{n+1}$. The operator $T_{FS1}$ defined by
	\begin{equation}
		\label{slice}
		T_{FS1}(f)= \alpha(x_0, | \underline{x}|)+ \frac{\underline{x}}{| \underline{x}|} \beta(x_0, |\underline{x}|)
	\end{equation}
	maps the holomoprhic function $f(z)$ in the set of intrinsic slice hyperholomorphic function. Then the function
	$$ \breve{f}(x):= \Delta^{\frac{n-1}{2}} \left(\alpha(x_0, | \underline{x}|)+ \frac{\underline{x}}{| \underline{x}|} \beta(x_0, |\underline{x}|)\right)$$
	is in the kernel of the Dirac operator, i.e.,
	$$ \mathcal D \breve{f}(x)=0, \qquad \hbox{on} \qquad \Omega_D.$$
\end{theorem}

\begin{remark}\label{remevenodd}
	The assumption that the function $f(z)= \alpha(u,v)+i \beta(u,v)$ is holomorphic function
	in a domain  $D$ in the upper-half complex plane
	can be removed if we assume the even-odd conditions in (\ref{eveodd}).
\end{remark}

\begin{remark}
Because of the factorization
$$
\mathcal{D}\mathcal{\overline{D}}=\mathcal{\overline{D}}\mathcal{D}=\Delta
$$
of the Laplace operator
we define some the classes of functions that are strictly related to the Fueter-Sce theorem.
 We will see in the next sections that these function spaces are of crucial importance for our theory.
\end{remark}
\begin{definition}[holomorphic Cliffordian of order $k$]
	\label{hc}
	Let $U$ be an open set. A function $f:U \subset \mathbb{R}^{n+1} \to \mathbb{R}_n$ of class $\mathcal{C}^{2k+1}(U)$ is said to be (left) holomorphic Cliffordian of order $k$ if
	$$ \Delta^k \mathcal{D}  f(x)=0\quad \forall x\in U,$$
	where $0 \leq k \leq \frac{n-1}{2}$.
\end{definition}
\begin{remark}
	For $k:= \frac{n-1}{2}$ in Definition \ref{hc} we get the class of functions studied in \cite{LR}.
\end{remark}
\begin{remark}
	Every holomorphic Cliffordian function of order $k$ is holomorphic Cliffordian of order $k+1$. If $k=0$ in Definition \ref{hc} we get the set of (left) monogenic functions.
\end{remark}
\begin{definition}[anti-holomorphic Cliffordian of order $k$]
	\label{anti1}
	Let $U$ be an open set. A function $f:U \subset \mathbb{R}^{n+1} \to \mathbb{R}_n$ of class $\mathcal{C}^{2k+1}(U)$ is said to be (left) anti-holomorphic Cliffordian of order $k$ if
	$$ \Delta^k \mathcal{\overline{D}}  f(x)=0\quad \forall x\in U,$$
	where $0 \leq k \leq \frac{n-1}{2}$.
\end{definition}
\begin{definition}[polyharmonic of degree $k$]
	\label{harmo}
	Let $k \geq 1$. A function $f:U \subset \mathbb{R}^{n+1} \to \mathbb{R}_n$ of class $\mathcal{C}^{2k}(U)$ is called polyharmonic of degree $k$ in the open set $U \subset \mathbb{R}^{n+1}$ if
	$$ \Delta^k f(x)=0, \quad\forall x \in U.$$
\end{definition}
For $p=1$ the function is called harmonic and for $p=2$ the function is called bi-harmonic. The polyharmonic functions are studied in \cite{A}.
\begin{definition}[polyanalytic of order $m$]
	\label{poly}
	Let $ m \geq 1$. Let $ U \subset \mathbb{R}^{n+1}$ be an open set and let $f:U \to \mathbb{R}_n$ be a function of class $ \mathcal{C}^{m}(U)$. We say that $f$ is (left) polyanalytic of order $m$ on $U$ if
	$$ \mathcal{D}^{m}f(x)=0, \qquad \forall x \in U.$$
\end{definition}

The following result is a characterization to be a polyanalytic function of order $m$, see \cite{B1976}.
\begin{proposition}[Polyanalytic decomposition]
	\label{polydeco}
	Let $ U \subset \mathbb{R}^{n+1}$ be an open. A function $f:U \to \mathbb{R}_n$ is polyanalytic of order $m$ if and only if it can be decomposed in terms of some unique monogenic functions $g_0,\ldots, g_{m-1}$
	$$ g(x)= \sum_{k=0}^{m-1} x_0^k g_k(x).$$
\end{proposition}
See \cite{B1976, DB1978, T} for more information about polyanalytic functions.
\begin{definition}[polyanalytic holomorphic Cliffordian of order $(k, \ell)$]
	\label{polycl}
	Let $U$ be an open set. A function $f:U \subset \mathbb{R}^{n+1} \to \mathbb{R}_n$ of class $\mathcal{C}^{2k+ \ell}(U)$ is said to be (left) polyanalytic holomorphic Cliffordian of order $(k, \ell)$ if
	$$ \Delta^k \mathcal{D}^{\ell}  f(x)=0\quad \forall x\in U,$$
	where $0 \leq k \leq \frac{n-1}{2}$ and $ \ell \geq 0$. We denote the set of these functions as $\mathcal{PCH}_{(k,\ell)}(U)$.
\end{definition}
\begin{remark}
Similarly it is possible to define the  sets of right functions for the holomorphic Cliffordian of order $k$, anti-holomorphic Cliffordian of order $k$ and polyanalytic of order $m$.

\end{remark}
\begin{remark}
	If in Definition \ref{polycl} we set $ \ell=1$ we get Definition \ref{hc}, if $\ell=0$ we obtain Definition \ref{harmo} and if we consider $k=0$ we obtain Definition \ref{poly}.
\end{remark}

\begin{remark}
	In Theorem \ref{FS}, the case of $n$ odd is due to M. Sce and the operator $T_{FS2}$ is a differential operator; the case of $n$ even is due to T. Qian and, in this case, the operator $T_{FS2}$ is a fractional operator.
\end{remark}
\begin{remark}
Even though the above classes of functions are defined
for suitable regular functions $f$ such that we can apply operators of the form
$\Delta^k \mathcal{D}^{\ell}$ or $\Delta^k \overline{\mathcal{D}}^{\ell}$ according to the Fueter-Sce mapping theorem
we have to assume that the function $f$ is a slice functions
 of the form (\ref{slicefunc}).
\end{remark}

\medskip

The investigation of the fines structures is a quite involved matter and depends on the dimension $n$.
In the next sections we will concentrate on the dimension $n=5$. In a paper in preparation we aim to study the fine structures for any $n$ odd.

\section{Function spaces of axial type in dimension five}\label{FACTRFIVE}

Working in the Clifford  algebra with five imaginary units, i.e., $n=5$
 the second Fueter-Sce map is $\Delta^2$ where the Laplace operator $\Delta$ is in dimension 6.

We recall that there exist different possible factorizations of
 $\Delta^2$ in terms of the Dirac operator $ \mathcal{D}$ and its conjugate $ \mathcal{\overline{D}}$ choosing different configurations  of products of
 $\mathcal{D}$ and $ \mathcal{\overline{D}}$.

The case of dimension five is different from what happens in the quaternionic case (see \cite{CDPS,CDPS1}), in which the Fueter map can be factorized only as $ \mathcal{D} \mathcal{\overline{D}}$ and $ \mathcal{\overline{D}} \mathcal{D}$: here we obtain a reacher structure.

\medskip
In the setting of slice hyperholomorphic functions, functions of the form
 (\ref{slicefunc})together with the even-odd conditions are called slice functions.
 In the monogenic setting such functions, often considered only in the upper half space, are called
 function of axial type. We will use both terminology according to the setting.
Now, we assume that the axial functions (or slice functions) of the form
$$
f(x)=\alpha(x_0, | \underline{x}|)+ \frac{\underline{x}}{| \underline{x}|} \beta(x_0, |\underline{x}|)
$$
 are of class $\mathcal{C}^5(\Omega_D)$ where $\Omega_D$ is as in the Fueter-Sce mapping theorem.
 We will consider functions $f: \Omega_D \subseteq \mathbb{R}^{n+1}\to \mathbb{R}_n$
with values in the Clifford algebra $\mathbb{R}_n$ where we consider the case $n=5$.
\begin{definition}[$\mathcal{ABH}(\Omega_D)$ axially bi-harmonic function]
Let $f: \Omega_D \subseteq \mathbb{R}^{6}\to \mathbb{R}_5$ be of axial type and of class $\mathcal{C}^5(\Omega_D)$.
Then, the function
$$
\tilde{f}_1(x):= \mathcal{D}f(x) \qquad \hbox{on} \quad \Omega_D
$$
is called  an axially bi-harmonic function, since by the Fueter-Sce mapping theorem, it satisfies
$$
\Delta^2 \tilde{f}_1(x)=0 \qquad \hbox{on} \quad \Omega_D.
$$
We denote this set of functions by $ \mathcal{ABH}(\Omega_D)$.
\end{definition}

\begin{definition}[$\mathcal{ACH}_1(\Omega_D)$ axially Cliffordian functions of order one]
Consider the function  $\tilde{f}_1(x):= \mathcal{D}f(x)\in  \mathcal{ABH}(\Omega_D)$ and
apply the conjugate Dirac operator $ \mathcal{\overline{D}}$ to $\tilde{f}_1(x)$. Then we get
	\begin{equation}
		\label{cliff}
		f^{\circ}(x):= \mathcal{\overline{D}}\tilde{f}_1(x)=\Delta f(x)\qquad \hbox{on} \quad \Omega_D,
	\end{equation}
	which is an axially Cliffordian functions of order one (which is the short cut for order $(1,1)$)
by the Fueter-Sce mapping theorem, i.e.,
	$$ \Delta \mathcal{D} f^{\circ}(x)=0\qquad \hbox{on} \quad \Omega_D.$$
We denote this set of functions by $\mathcal{ACH}_1(\Omega_D)$.
\end{definition}

\begin{definition}[$\overline{\mathcal{ACH}_1(\Omega_D)}$ axially anti-Cliffordian functions of order one]
Consider the function  $\tilde{f}_1(x):= \mathcal{D}f(x)\in  \mathcal{ABH}(\Omega_D)$ and
apply the  Dirac operator $ \mathcal{{D}}$ to $\tilde{f}_1(x)$ we obtain
$$
f_{\circ}(x)=\mathcal{D}\tilde{f}_1(x)= \mathcal{D}^2 f(x),\qquad \hbox{on} \quad \Omega_D,
$$
	which is axially anti-Cliffordian functions of order one
by the Fueter-Sce mapping theorem, i.e.,
	$$ \Delta \mathcal{\overline{D}} \left(f_{\circ}(x)\right)=0 \qquad \hbox{on} \quad \Omega_D.$$
We denote this set of functions by $\overline{\mathcal{ACH}_1(\Omega_D)}$.
\end{definition}

\begin{definition}[$\mathcal{AH}(\Omega_D)$ axially harmonic functions]

Consider the function
$f^{\circ}(x):= \mathcal{\overline{D}}\tilde{f}_1(x)=\Delta f(x)\in \mathcal{ACH}_1(\Omega_D)$ and
 apply the Dirac operator $ \mathcal{D}$ to $f^{\circ}(x)$. We get,
	$$\tilde{f}_0(x)=  \mathcal{D}f^{\circ}(x)=\Delta \mathcal{D} f(x),$$
	which is an axially harmonic functions,  by the Fueter-Sce mapping theorem, i.e.,
	$$ \Delta \tilde{f}_0(x)=0 \qquad \hbox{on} \quad \Omega_D.$$
We denote this set of functions as $ \mathcal{AH}(\Omega_D)$.

\end{definition}
  \begin{definition}[$\mathcal{AP}_2( \Omega_D)$ axially polyanalytic functions of order two]
  If we apply the operator $ \mathcal{\overline{D}}$ to
  $f^{\circ}(x):= \mathcal{\overline{D}}\tilde{f}_1(x)=\Delta f(x)\in \mathcal{ACH}_1(\Omega_D)$ we obtain
	$$
\breve{f}^{\circ}_1(x)=  \mathcal{\overline{D}}f^{\circ}(x)=\Delta \mathcal{\overline{D}} f(x),
$$
	which is an axially polyanalytic functions of order two,
by the Fueter-Sce mapping theorem, i.e.,
	$$ \mathcal{D}^2 \breve{f}^{\circ}_1(x)=0 \qquad \hbox{on} \quad \Omega_D.$$
We denote this set of functions by $\mathcal{AP}_2( \Omega_D)$.
\end{definition}

\begin{definition}[$ \mathcal{APC}_{(1,2)}( \Omega_D)$ axially Cliffordian polyanalytic functions of order $(1,2)$]
Let $f: \Omega_D \subseteq \mathbb{R}^{6}\to \mathbb{R}_5$ be of axial type and of class $\mathcal{C}^5(\Omega_D)$.
Apply to \eqref{slice} the conjugate of the Dirac operator. In this case we obtain
\begin{equation}
	\label{polycliff}
	\breve{f}^{\circ}(x)=\mathcal{\overline{D}} f(x)  \qquad \hbox{on} \quad  \Omega_D,
\end{equation}
which is an axially Cliffordian polyanalytic functions of order $(1,2)$, by the Fueter-Sce mapping theorem, i.e.,
$$\Delta \mathcal{D}^2\breve{f}^{\circ}(x)=0\qquad \hbox{on} \quad  \Omega_D.$$
We denote this class of functions as $ \mathcal{APC}_{(1,2)}( \Omega_D)$.
\end{definition}

\begin{definition}[$\mathcal{AP}_3( \Omega_D)$ axially polyanalytic function of order three]
Let $\breve{f}^{\circ}(x)=\mathcal{\overline{D}} f(x)\in \mathcal{APC}_{(1,2)}( \Omega_D)$.
 Applying the operator conjugate Dirac operator $\mathcal{\overline{D}}$ to $\breve{f}^{\circ}(x)$, we get
$$\breve{f}^{\circ}_0(x)=\mathcal{\overline{D}}^2 f(x) \qquad \hbox{on} \quad  \Omega_D,$$
which an axially polyanalytic function of order three, i.e.,
$$
\mathcal{D}^3\breve{f}^{\circ}_0(x)=0 \qquad \hbox{on} \quad  \Omega_D.
$$
We denote this class of functions as $ \mathcal{AP}_3( \Omega_D)$.
\end{definition}
\begin{remark} Keeping in mind the above notations we have
$$ \breve{f}(x)=\Delta \mathcal{\overline{D}} \tilde{f}_0(x)=\mathcal{D}\breve{f}^{\circ}_1(x)= \mathcal{\overline{D}}^2f_{\circ}(x)=\Delta f^{\circ}(x)=\mathcal{\overline{D}} \tilde{f}_1(x),$$
where $\breve{f}$ is axially monogenic
and also
$$ \breve{f}(x)= \mathcal{D}^2 \breve{f}^{\circ}_0(x)= \Delta \mathcal{D} \breve{f}^{\circ}(x)\qquad \hbox{on} \quad  \Omega_D.$$
\end{remark}
\begin{remark}
In the general case appears the same classes of functions but with different orders.
\end{remark}
Taking advantage of the function spaces of axial functions defined in this section we can now
define the fine structure associated with this spaces that appear in the Clifford algebra $\mathbb R_5$.

\section{The fine structures in dimension five}\label{section4}
By applying the Fueter-Sce map $T_{FS2}:= \Delta_{\mathbb{R}^{n+1}}^{h}$, where $h:=\frac{n-1}{2}$ and is the Sce exponent, to a slice hyperholomorphic function $f(x)$ we get the monogenic function
$ \breve{f}(x)= \Delta_{\mathbb{R}^{n+1}}^hf(x)$.
\\ Due to the factorization of the Laplace operator in terms of $ \mathcal{D}$ and $ \mathcal{\overline{D}}$ it is possible to apply these two operators to a slice hyperholomorphic function $f(x)$ a number of times, until we reach the maximum power of the Laplacian, i.e., the Sce exponent.
\\ This implies the possibility to build different sets of functions between the set of slice hyperholomorphic functions and the set of axially monogenic functions, (see the previous section). This fact leads to the definition of fine structure of slice hyperholomorphic spectral theory.

\begin{definition}[Fine structure of slice hyperholomorphic spectral theory]
A fine structure of slice hyperholomorphic spectral theory is the set of functions spaces and the associated functional calculi induced by a factorization of the operator $T_{FS2}$, in the Fueter-Sce extension theorem.
\end{definition}
The factorization $ T_{FS2}=\Delta_{\mathbb{R}^{n+1}}^{h}= \mathcal{D} \mathcal{\overline{D}}... \mathcal{D}\mathcal{\overline{D}}$ is of particular interest.
\begin{definition}[Dirac fine structure]
The Dirac fine structure corresponds to an alternating sequence of products of the Dirac operator $\mathcal{D}$ and of its conjugate $ \mathcal{\overline{D}}$ until we obtain $\Delta_{\mathbb{R}^{n+1}}^{\frac{n-1}{2}}$.
\end{definition}
In \cite{CDPS} we studied the Dirac fine structure when $n=3$ this is also the quaternionic case. In particular we studied  the sequence represented by the following  diagram:
\begin{equation}
\label{fine1b}
\mathcal{O}(D) \overset{T_{FS1}}{\longrightarrow} \mathcal{SH}(\Omega_D)\overset{\mathcal{D}}{\longrightarrow} \mathcal{AH}(\Omega_D)\overset{\mathcal{\overline{D}}}{\longrightarrow}\mathcal{AM}(\Omega_D).
\end{equation}
The fine structure in (\ref{fine1b}) allows to obtain a product rule for the $F$-functional calculus, see \cite[Thm. 9.3]{CDPS}.
\\ However, since $\Delta=\mathcal{D} \mathcal{\overline{D}}= \mathcal{\overline{D}}\mathcal{D}$, we can exchange the roles of the operators $ \mathcal{D}$ and $ \mathcal{\overline{D}}$ in \eqref{fine1b}. This gives rise to the sequence represented by the following  diagram:
\begin{equation}
\label{fine2b}
\mathcal{O}(D) \overset{T_{FS1}}{\longrightarrow} \mathcal{SH}(\Omega_D)\overset{\mathcal{\overline{D}}}{\longrightarrow} \mathcal{AP}_2(\Omega_D)\overset{\mathcal{D}}{\longrightarrow}\mathcal{AM}(\Omega_D),
\end{equation}
which is investigated in \cite{CDPS1}.
 Even if the diagrams \eqref{fine1b} and \eqref{fine2b} come from the Fueter mapping theorem and the factorization of the Fueter operator $T_{F2}=\Delta$, we get two different fine structures.

\medskip
In each fine structure above and in all the fine structures we consider in the sequel
the final set of function spaces is always the set of axially monogenic functions.

\medskip
In the Clifford setting the splitting of the second Fueter-Sce mapping is more complicated, due to the fact that we are dealing with integer powers of the Laplacian.
Moreover, when $n$ is even the Laplace operator has a fractional power
and so we have to work in the space of distributions using the Fourier multipliers, see \cite{Q1}.

\medskip
Due to the fact that for $n=5$ we deal with the operator $ \Delta^2$, we get more Dirac fine structures, which are all different and important at the same time. In order to label all the fine structures, we will denote every fine structures, with an ordered sequence of the applied operators. For example, in the quaternionic case, we call \eqref{fine1b} the Dirac fine structure of the kind $(\mathcal{D}, \mathcal{\overline{D}})$ and \eqref{fine2b} the Dirac fine structure of the kind $(\mathcal{\overline{D}}, \mathcal{D})$.
Also in the case $n=5$ we have a structure in which we apply alternately the operators $ \mathcal{D}$ and $ \mathcal{\overline{D}}$ until we reach the second Fueter-Sce mapping.
\begin{equation}
\label{fine3}
\mathcal{O}(D) \overset{T_{FS1}}{\longrightarrow} \mathcal{SH}(\Omega_D)\overset{\mathcal{D}}{\longrightarrow} \mathcal{ABH}(\Omega_D)\overset{\mathcal{\overline{D}}}{\longrightarrow}\mathcal{AHC}_1(\Omega_D) \overset{\mathcal{D}}{\longrightarrow} \mathcal{AH}(\Omega_D) \overset{\mathcal{\overline{D}}}{\longrightarrow} \mathcal{AM}(\Omega_D).
\end{equation}
We call \eqref{fine3} the Dirac fine structure of the kind $(\mathcal{D}, \mathcal{\overline{D}}, \mathcal{D}, \mathcal{\overline{D}} )$.
\begin{remark}
Even when $n=5$  the Dirac fine structure \eqref{fine3} is of fundamental importance to obtain a product formula for the $F$-functional calculus (see Theorem \ref{prodo}).
\end{remark}
\begin{remark}
In order to avoid, at the end of the sequence of spaces the set of axially-anti monogenic function, we impose the condition that the composition of all the operators between spaces Clifford valued functions, must be equal to the operator $T_{FS2}=\Delta^{(n-1)/2} $ in the Fueter-Sce mapping theorem.
\end{remark}
 However, by rearranging the sequence of $ \mathcal{D}$ and $ \mathcal{\overline{D}}$ it is possible to obtain other fine structures, in which other sets of functions are involved. Thus, we have the Dirac fine structures $(\mathcal{D}, \mathcal{\overline{D}},\mathcal{\overline{D}}, \mathcal{D} )$

$$ \mathcal{O}(D) \overset{T_{FS1}}{\longrightarrow} \mathcal{SH}(\Omega_D)\overset{\mathcal{D}}{\longrightarrow} \mathcal{ABH}(\Omega_D)\overset{\mathcal{\overline{D}}}{\longrightarrow}	\mathcal{AHC}_1(\Omega_D) \overset{\mathcal{\overline{D}}}{\longrightarrow} \mathcal{AP}_2(\Omega_D) \overset{\mathcal{D}}{\longrightarrow} \mathcal{AM}(\Omega_D),$$
and the Dirac fine structure $(\mathcal{D}, \mathcal{D}, \mathcal{\overline{D}}, \mathcal{\overline{D}}  )$
$$ \mathcal{O}(D) \overset{T_{FS1}}{\longrightarrow} \mathcal{SH}(\Omega_D)\overset{\mathcal{D}}{\longrightarrow} \mathcal{ABH}(\Omega_D)\overset{\mathcal{D}}{\longrightarrow}\overline{\mathcal{AHC}_1(\Omega_D)} \overset{\mathcal{\overline{D}}}{\longrightarrow} \mathcal{AH}(\Omega_D) \overset{\mathcal{\overline{D}}}{\longrightarrow} \mathcal{AM}(\Omega_D).$$
All the previous Dirac fine structures are obtained by applying first the Dirac operator. Nevertheless, it is possible to apply the operator $ \mathcal{\overline{D}}$ as  first operator. In this case other three Dirac fine structures arise. We have the Dirac fine structure of the kind $(\mathcal{\overline{D}}, \mathcal{D}, \mathcal{\overline{D}},\mathcal{D} )$
$$ \mathcal{O}(D) \overset{T_{FS1}}{\longrightarrow} \mathcal{SH}(\Omega_D)\overset{\mathcal{\overline{D}}}{\longrightarrow} \mathcal{APC}_{(1,2)}(\Omega_D)\overset{\mathcal{D}}{\longrightarrow} \mathcal{AHC}_1(\Omega_D) \overset{\mathcal{\overline{D}}}{\longrightarrow} \mathcal{AP}_2(\Omega_D) \overset{\mathcal{D}}{\longrightarrow} \mathcal{AM}(\Omega_D),$$
the Dirac fine structure $(\mathcal{\overline{D}}, \mathcal{D},\mathcal{D},  \mathcal{\overline{D}} )$
$$ \mathcal{O}(D) \overset{T_{FS1}}{\longrightarrow} \mathcal{SH}(\Omega_D)\overset{\mathcal{\overline{D}}}{\longrightarrow} \mathcal{APC}_{(1,2)}(\Omega_D)\overset{\mathcal{D}}{\longrightarrow} \mathcal{AHC}_1(\Omega_D) \overset{\mathcal{D}}{\longrightarrow} \mathcal{AH}(\Omega_D) \overset{\mathcal{\overline{D}}}{\longrightarrow} \mathcal{AM}(\Omega_D),$$
and the Dirac fine structure $(\mathcal{\overline{D}},\mathcal{\overline{D}}, \mathcal{D},\mathcal{D})$
$$ \mathcal{O}(D) \overset{T_{FS1}}{\longrightarrow} \mathcal{SH}(\Omega_D)\overset{\mathcal{\overline{D}}}{\longrightarrow} \mathcal{APC}_{(1,2)}(\Omega_D)\overset{\mathcal{\overline{D}}}{\longrightarrow} \mathcal{AP}_3(\Omega_D) \overset{\mathcal{D}}{\longrightarrow} \mathcal{AH}(\Omega_D) \overset{\mathcal{D}}{\longrightarrow} \mathcal{AM}(\Omega_D).$$
The following diagram summarizes all the Dirac fine structures
with their function spaces:
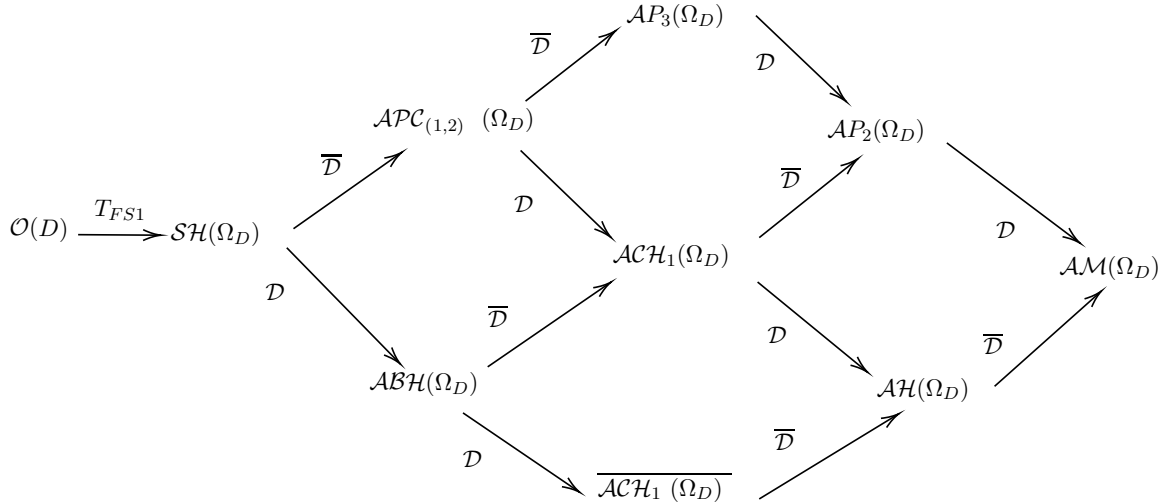
\begin{figure}[H]
\centering
\resizebox{0.95\textwidth}{!}{%
\tikzset{every picture/.style={line width=0.75pt}} 

\begin{tikzpicture}[x=0.75pt,y=0.75pt,yscale=-1,xscale=1]
	
	\draw    (61.22,178.04) -- (109.64,178.35) ;
	\draw [shift={(111.64,178.37)}, rotate = 180.37] [color={rgb, 255:red, 0; green, 0; blue, 0 }  ][line width=0.75]    (10.93,-3.29) .. controls (6.95,-1.4) and (3.31,-0.3) .. (0,0) .. controls (3.31,0.3) and (6.95,1.4) .. (10.93,3.29)   ;
	\draw    (192.08,174.42) -- (256.45,128.99) ;
	\draw [shift={(258.08,127.83)}, rotate = 144.79] [color={rgb, 255:red, 0; green, 0; blue, 0 }  ][line width=0.75]    (10.93,-3.29) .. controls (6.95,-1.4) and (3.31,-0.3) .. (0,0) .. controls (3.31,0.3) and (6.95,1.4) .. (10.93,3.29)   ;
	\draw    (310.08,258.42) -- (383.44,207.97) ;
	\draw [shift={(385.08,206.83)}, rotate = 145.48] [color={rgb, 255:red, 0; green, 0; blue, 0 }  ][line width=0.75]    (10.93,-3.29) .. controls (6.95,-1.4) and (3.31,-0.3) .. (0,0) .. controls (3.31,0.3) and (6.95,1.4) .. (10.93,3.29)   ;
	\draw    (295.08,286.42) -- (364.43,333.62) ;
	\draw [shift={(366.08,334.75)}, rotate = 214.25] [color={rgb, 255:red, 0; green, 0; blue, 0 }  ][line width=0.75]    (10.93,-3.29) .. controls (6.95,-1.4) and (3.31,-0.3) .. (0,0) .. controls (3.31,0.3) and (6.95,1.4) .. (10.93,3.29)   ;
	\draw    (474.08,206.83) -- (539.54,260.56) ;
	\draw [shift={(541.08,261.83)}, rotate = 219.38] [color={rgb, 255:red, 0; green, 0; blue, 0 }  ][line width=0.75]    (10.93,-3.29) .. controls (6.95,-1.4) and (3.31,-0.3) .. (0,0) .. controls (3.31,0.3) and (6.95,1.4) .. (10.93,3.29)   ;
	\draw    (475.08,179.83) -- (533.52,133.08) ;
	\draw [shift={(535.08,131.83)}, rotate = 141.34] [color={rgb, 255:red, 0; green, 0; blue, 0 }  ][line width=0.75]    (10.93,-3.29) .. controls (6.95,-1.4) and (3.31,-0.3) .. (0,0) .. controls (3.31,0.3) and (6.95,1.4) .. (10.93,3.29)   ;
	\draw    (333.08,96.83) -- (386.51,55.07) ;
	\draw [shift={(388.08,53.83)}, rotate = 141.98] [color={rgb, 255:red, 0; green, 0; blue, 0 }  ][line width=0.75]    (10.93,-3.29) .. controls (6.95,-1.4) and (3.31,-0.3) .. (0,0) .. controls (3.31,0.3) and (6.95,1.4) .. (10.93,3.29)   ;
	\draw    (473.08,44.83) -- (526.64,95.87) ;
	\draw [shift={(528.08,97.25)}, rotate = 223.62] [color={rgb, 255:red, 0; green, 0; blue, 0 }  ][line width=0.75]    (10.93,-3.29) .. controls (6.95,-1.4) and (3.31,-0.3) .. (0,0) .. controls (3.31,0.3) and (6.95,1.4) .. (10.93,3.29)   ;
	\draw    (589.08,122.25) -- (668.49,182.05) ;
	\draw [shift={(670.08,183.25)}, rotate = 216.98] [color={rgb, 255:red, 0; green, 0; blue, 0 }  ][line width=0.75]    (10.93,-3.29) .. controls (6.95,-1.4) and (3.31,-0.3) .. (0,0) .. controls (3.31,0.3) and (6.95,1.4) .. (10.93,3.29)   ;
	\draw    (475.08,338.75) -- (558.38,287.88) ;
	\draw [shift={(560.08,286.83)}, rotate = 148.58] [color={rgb, 255:red, 0; green, 0; blue, 0 }  ][line width=0.75]    (10.93,-3.29) .. controls (6.95,-1.4) and (3.31,-0.3) .. (0,0) .. controls (3.31,0.3) and (6.95,1.4) .. (10.93,3.29)   ;
	\draw    (188,186) -- (255.69,255.32) ;
	\draw [shift={(257.08,256.75)}, rotate = 225.68] [color={rgb, 255:red, 0; green, 0; blue, 0 }  ][line width=0.75]    (10.93,-3.29) .. controls (6.95,-1.4) and (3.31,-0.3) .. (0,0) .. controls (3.31,0.3) and (6.95,1.4) .. (10.93,3.29)   ;
	\draw    (330.08,126.83) -- (383.64,178.45) ;
	\draw [shift={(385.08,179.83)}, rotate = 223.94] [color={rgb, 255:red, 0; green, 0; blue, 0 }  ][line width=0.75]    (10.93,-3.29) .. controls (6.95,-1.4) and (3.31,-0.3) .. (0,0) .. controls (3.31,0.3) and (6.95,1.4) .. (10.93,3.29)   ;
	\draw    (618.08,270.25) -- (681.59,213.58) ;
	\draw [shift={(683.08,212.25)}, rotate = 138.26] [color={rgb, 255:red, 0; green, 0; blue, 0 }  ][line width=0.75]    (10.93,-3.29) .. controls (6.95,-1.4) and (3.31,-0.3) .. (0,0) .. controls (3.31,0.3) and (6.95,1.4) .. (10.93,3.29)   ;
	
	\draw (18.09,165.97) node [anchor=north west][inner sep=0.75pt]  [rotate=-0.38] [align=left] {$\displaystyle \mathcal{O}( D)$};
	\draw (69.72,155.35) node [anchor=north west][inner sep=0.75pt]  [rotate=-0.38]  {$T_{FS1}$};
	\draw (116.07,168.79) node [anchor=north west][inner sep=0.75pt]  [rotate=-0.38]  {$\mathcal{SH}( \Omega _{D})$};
	\draw (173.64,206.53) node [anchor=north west][inner sep=0.75pt]  [rotate=-0.38]  {$\mathcal{D}$};
	\draw (207.89,127.48) node [anchor=north west][inner sep=0.75pt]  [rotate=-0.38]  {$\mathcal{\overline{D}}$};
	\draw (236.84,259.59) node [anchor=north west][inner sep=0.75pt]  [rotate=-0.38]  {$\mathcal{ABH}( \Omega _{D})$};
	\draw (239,99.4) node [anchor=north west][inner sep=0.75pt]    {$\mathcal{APC}_{(}{}_{1,2) \ \ }( \Omega _{D})$};
	\draw (308.89,219.48) node [anchor=north west][inner sep=0.75pt]  [rotate=-0.38]  {$\mathcal{\overline{D}}$};
	\draw (293.64,307.53) node [anchor=north west][inner sep=0.75pt]  [rotate=-0.38]  {$\mathcal{D}$};
	\draw (384.58,182.4) node [anchor=north west][inner sep=0.75pt]    {$\mathcal{ACH}_{1}( \Omega _{D}) \ $};
	\draw (375.58,321.4) node [anchor=north west][inner sep=0.75pt]    {$\overline{\mathcal{ACH}_{1} \ ( \Omega _{D}) \ }$};
	\draw (478.64,232.53) node [anchor=north west][inner sep=0.75pt]  [rotate=-0.38]  {$\mathcal{D}$};
	\draw (486.89,135.48) node [anchor=north west][inner sep=0.75pt]  [rotate=-0.38]  {$\mathcal{\overline{D}}$};
	\draw (545,263.4) node [anchor=north west][inner sep=0.75pt]    {$\mathcal{AH}( \Omega _{D})$};
	\draw (515,106.4) node [anchor=north west][inner sep=0.75pt]    {$\mathcal{A} P_{2}( \Omega _{D})$};
	\draw (393,36.4) node [anchor=north west][inner sep=0.75pt]    {$\mathcal{A} P_{3}( \Omega _{D})$};
	\draw (334.89,53.48) node [anchor=north west][inner sep=0.75pt]  [rotate=-0.38]  {$\mathcal{\overline{D}}$};
	\draw (656,189.48) node [anchor=north west][inner sep=0.75pt]    {$\mathcal{AM}( \Omega _{D})$};
	\draw (617.64,168.53) node [anchor=north west][inner sep=0.75pt]  [rotate=-0.38]  {$\mathcal{D}$};
	\draw (609.89,236.48) node [anchor=north west][inner sep=0.75pt]  [rotate=-0.38]  {$\mathcal{\overline{D}}$};
	\draw (483,294.48) node [anchor=north west][inner sep=0.75pt]    {$\overline{\mathcal{D}}$};
	\draw (471.64,66.53) node [anchor=north west][inner sep=0.75pt]  [rotate=-0.38]  {$\mathcal{D}$};
	\draw (324.64,152.53) node [anchor=north west][inner sep=0.75pt]  [rotate=-0.38]  {$\mathcal{D}$};

\end{tikzpicture}
}
\end{figure}

\begin{remark}
In all the previous Dirac fine structures it is possible to combine the Dirac operator and its conjugate. In this way we get a fine structure which is "weaker" then the previous ones; in the sense that we are skipping some classes of functions. We call these kind of fine structures coarser. Up to now we have mentioned just some of them and in the Appendix
 we show the complete landscape of the fine structures in dimension five.
\end{remark}

The Laplace fine structure is of the kind $(\Delta, \Delta)$, which is a coarser fine structure with respect to the Dirac one, is given by
$$ \mathcal{O}(D) \overset{T_{FS1}}{\longrightarrow} \mathcal{SH}(\Omega_D)\overset{\Delta}{\longrightarrow}  \mathcal{ACH}_1(\Omega_D) \overset{\Delta}{\longrightarrow}  \mathcal{AM}(\Omega_D).$$
Other interesting coarser fine structures are the harmonic ones, in which appear only the harmonic and
bi-harmonic sets of functions
$$ \mathcal{O}(D) \overset{T_{FS1}}{\longrightarrow} \mathcal{SH}(\Omega_D)\overset{\mathcal{D}}{\longrightarrow}  \mathcal{ABH}(\Omega_D) \overset{\Delta}{\longrightarrow}  \mathcal{AH}(\Omega_D)\overset{\mathcal{D}}{\longrightarrow} \mathcal{AM}(\Omega_D),$$
and the polyanalytic one, in which there appear only the polyanalytic functions of order three and two
$$ \mathcal{O}(D) \overset{T_{FS1}}{\longrightarrow} \mathcal{SH}(\Omega_D)\overset{\mathcal{\overline{D}}^2}{\longrightarrow}  \mathcal{AP}_3(\Omega_D) \overset{\mathcal{D}}{\longrightarrow}  \mathcal{AP}_2(\Omega_D)\overset{\mathcal{D}}{\longrightarrow} \mathcal{AM}(\Omega_D).$$
We observe that it is not possible to have coarser fine structure in the quaternionic case. This is due to the fact that we are dealing with the Laplacian at power $1$.

\section{Systems of differential equations for fine structure spaces of axial type}\label{VEKUA}

In  analogy with the Vekua-type system of differential equations for axially monogenic functions in this section we give all the systems of differential equations for the fine structure spaces in dimension five.
Let $D$ be a domain in the upper-half complex plane.
Let $\Omega_D$ be an axially symmetric open set in $ \mathbb{R}^{6}$ and let $x=x_0+ \underline{x}=x_0+r \underline{\omega} \in \Omega_D$. A function $ f: \Omega_D \to \mathbb{R}_5$ is of axial type if there exist two functions $A=A(x_0,r)$ and $B=B(x_0,r)$, independent of $ \underline{\omega} \in \mathbb{S}$ and with values in $ \mathbb{R}_5$, such that
	$$ f(x)= A(x_0,r)+ \underline{\omega}B(x_0,r), \ \ {\rm where}\ \ r>0. $$
So we characterize the class of functions that lies between the set of slice hyperholomorphic and the set of axially monogenic functions, that we have denoted as axially functions. We recall by \cite{Dixan}, that if $f(x)=A(x_0,r)+ \underline{\omega} B(x_0, r)$ then
\begin{equation}
\label{Dirac}
\mathcal{D}f= \left(\partial_{x_0} A(x_0,r)- \partial_r B(x_0,r)- \frac{4}{r}B(x_0,r)\right)+ \underline{\omega}\left(\partial_{x_0}B(x_0,r)+ \partial_r A(x_0,r)\right),
\end{equation}
\begin{equation}
\label{NDirac}
\mathcal{\overline{D}}f= \left(\partial_{x_0} A(x_0,r)+ \partial_r B(x_0,r)+ \frac{4}{r}B(x_0,r)\right)+ \underline{\omega}\left(\partial_{x_0}B(x_0,r)- \partial_r A(x_0,r)\right).
\end{equation}

\begin{theorem}
\label{anti2}
Let $D \subseteq \mathbb{C}$. Let $\Omega_D$ be an axially symmetric open set in $ \mathbb{R}^6$ and let $f_{\circ}(x)=A(x_0,r)+ \underline{\omega}B(x_0,r)$ be an axially anti cliffordian holomorphic function of order 1. Then $A(x_0,r)$ and $B(x_0,r)$ satisfy the following system
$$ \begin{cases}
\partial_{x_0}^3A+ \partial_{x_0}\partial_r^2 A+ \frac{4}{r} \partial_{x_0} \partial_r A+ \partial_r \partial_{x_0}^2 B+ \partial_r^3 B+8 \frac{\partial_r^2B}{r}+8 \frac{\partial_r B}{r^2}-8 \frac{B}{r^3}+ \frac{4}{r} \partial_{x_0}^2B=0\\
\\
\partial_{x_0}^3 B+ \partial_{x_0}\partial_r^2 B-4 \partial_r \left( \frac{\partial_{x_0}B}{r}\right)-\partial_r \partial_{x_0}^2 A- \partial_r^3 A-4 \partial_r \left(\frac{\partial_r A_1}{r}\right)=0
\end{cases}
$$
\end{theorem}
\begin{proof}
Let us consider $f_{\circ}(x)=A+ \underline{\omega}B$. By similar computations done in \cite[Thm. 3.5]{CDPS} we have
\begin{equation}
\label{zero}
\Delta (f_{\circ}(x))= \left(\partial_{x_0}^2A+\partial_{r}^2A+ \frac{4}{r} \partial_r A\right)+\underline{\omega}\left(\partial_{x_0}^2B+\partial_{r}^2B+ 4 \partial_r \left(\frac{B}{r}\right) \right).
\end{equation}
Now, we set
$$ A':=\partial_{x_0}^2A+\partial_{r}^2A+ \frac{4}{r} \partial_r A\ \ \ {\rm and} \ \ \
 B':=\partial_{x_0}^2B+\partial_{r}^2B+ 4 \frac{\partial_r}{r}B- \frac{4}{r^2} B.$$
Then by formula \eqref{NDirac} we have
\begin{eqnarray}
\nonumber
\Delta \overline \bigD (f_{\circ}(x))&=&\overline \bigD(A'+\underline\omega B')=(\dx A'+\dr B'+\frac 4r B')+\underline\omega (\dx A' -\dr B')
\\
\nonumber
&=& \dx^3 A+\dx\dr^2 A+\frac 4r\dx\dr A +\dr^3B +\frac 4r\dr^2 B -\frac 4{r^2} \dr B
\\
\nonumber
&+&\frac 8{r^3} B-\frac 4{r^2} \dr B+\dr\dx^2 B+\frac 4r \dr^2 B
+\frac{16}{r^2} \dr B-\frac{16}{r^3} B+\frac 4r\dx^2B
\\
\nonumber
&+&\underline{\omega}(\dx\dr^2B+\frac 4r\dx\dr B-\frac 4{r^2}\dx B+\dx^3 B-\dr\dx^2 A-\dr^3 A
\\
&+&\frac 4{r^2}\dr A-\frac 4r \dr^2 A)
\end{eqnarray}
so we finally have
\begin{eqnarray}
\nonumber
\Delta \overline \bigD (f_{\circ}(x))&=& \dx^3 A +\dx\dr^2A +\frac 4r \dx\dr A+\dr^3 B+\frac 8r \dr^2 B+\frac 8{r^2} \dr B-\frac 8{r^3} B
\\\nonumber
&+&
\dr\dx^2 B+\frac 4r\dx^2 B +
\underline \omega (\dx\dr^2 B+\frac 4r\dx\dr B-\frac 4{r^2}\dx B
\\
&+&\dx^3B-\dr\dx^2A-\dr^3 A+\frac 4{r^2}\dr A-\frac 4r\dr^2 A).
\label{new1}
\end{eqnarray}
Since $({}^{\circ} f(x))$ is anti Cliffordian holomorphic of order one we have that $\Delta \mathcal{\overline{D}}(f_{\circ}(x))=0$.
\end{proof}
\begin{theorem}
Let $D \subseteq \mathbb{C}$.  Let $\Omega_D$ be an axially symmetric open set in $ \mathbb{R}^6$ and let $\tilde{f}_1(x)=A(x_0,r)+ \underline{\omega}B(x_0,r)$ be an axially bi-harmonic function. Then $A(x_0,r)$ and $B(x_0,r)$ satisfy the following system
\[
\begin{cases}
\dx^4 A+2\dx^2\dr^2 A+\dr^4 A-\frac 8{r^3}\dr A +\frac 8{r^2}\dr^2 A+\frac 8r\dr^3 A+\frac4r\dr\dx^2 A=0\\
\\
\dr^4B+\frac 8r\dr^3 B-\frac 24{r^3}\dr B+\frac{24}{r^4}B+2\dr^2\dx^2 B-\frac 8{r^2}\dx^2 B+\frac 8r\dx^2\dr B+\dx^4B=0.
\end{cases}
\]
\end{theorem}
\begin{proof}
By formula \eqref{new1} we have
$$
C:=\dx^3 A +\dx\dr^2A +\frac 4r \dx\dr A+\dr^3 B+\frac 8r \dr^2 B+\frac 8{r^2} \dr B-\frac 8{r^3} B+\dr\dx^2 B+\frac 4r\dx^2 B
$$
and
$$
D:=\dx\dr^2 B+\frac 4r\dx\dr B-\frac 4{r^2}\dx B+\dx^3B-\dr\dx^2A-\dr^3 A+\frac 4{r^2}\dr A-\frac 4r\dr^2 A,
$$
Therefore, by formula \eqref{Dirac} we have
\[
\begin{split}
&\Delta^2(f(x))=\bigD(\overline \bigD\Delta f(x))=(\dx C-\dr D-\frac 4r D)+\underline \omega(\dx D+\dr C)\\
&= \dx^4 A+\dx^2\dr^2 A+\frac 4r\dx^2\dr A+\dx\dr^3 B+\frac 8r\dx\dr^2 B+\frac{8}{r^2}\dx\dr B-\frac 8{r^3}\dx B+\dr\dx^3 B+\frac 4r\dx^3 B\\
&-\dx\dr^3 B-\frac 4r\dx\dr^2 B +\frac 8{r^2}\dx\dr B -\frac{8}{r^3}\dx B-\dr\dx^3 B+\dr^2\dx^2 A+\dr^4 A+\frac 8{r^3}\dr A-\frac{4}{r^2}\dr^2 A\\
&+\frac{4}r\dr^3 A-\frac 4{r^2}\dr^2 A-\frac 4r\dx\dr^2 B-\frac{16}{r^2}\dx\dr B+\frac{16}{r^3}\dx B-\frac 4r\dx^3 B+\frac 4r\dr\dx^2 A+\frac 4r\dr^3 A-\frac {16}{r^3}\dr A\\
& +\frac {16}{r^2}\dr^2 A+\underline{\omega}(\dr\dx^3 A+\dx\dr ^3 A +\frac 4r\dx\dr^2 A+\dr^4 B-\frac 8{r^2}\dr^2 B+\frac 8r\dr^3 B-\frac{16}{r^3}\dr B+\frac 8{r^2}\dr^2 B+\frac{24}{r^4}B\\
&- \frac{4}{r^2} \partial_{x_0}\partial_r A-\frac 8{r^3}\dr B+\dr^2\dx^2 B-\frac 4{r^2}\dx^2 B+\frac 4r\dr\dx^2 B+\dx^2\dr^2 B+\frac 4r\dx^2\dr B-\frac 4{r^2}\dx^2 B+\dx^4 B\\
&-\dr\dx^3 A-\dr^3\dx A+\frac 4{r^2}\dx\dr A-\frac 4r\dr^2\dx A)\\
&=\dx^4 A+2\dx^2\dr^2 A+\dr^4 A-\frac 8{r^3}\dr A +\frac 8{r^2}\dr^2 A+\frac 8r\dr^3 A+\frac 4r\dr\dx^2 A+\underline\omega(\dr^4B+\frac 8r\dr^3 B-\frac 24{r^3}\dr B\\
&+\frac{24}{r^4}B+2\dr^2\dx^2 B-\frac 8{r^2}\dx^2 B+\frac 8r\dx^2\dr B+\dx^4B).
\end{split}
\]
Since the function $\tilde{f}_1(x)$ is bi-harmonic, i.e. $\Delta^2\tilde{f}_1(x)=0$, we have the thesis.
\end{proof}

\begin{theorem}
Let $D \subseteq \mathbb{C}$. Let $\Omega_D$ be an axially symmetric open set in $ \mathbb{R}^6$ and let $\breve{f}^{\circ}_0(x)=A(x_0,r)+ \underline{\omega}B(x_0,r)$ be an axially polyanalytic function of order three. Then $A:=A(x_0,r)$ and $B:=B(x_0,r)$ satisfy the following system
$$ \begin{cases}
\partial_{x_0}^3A+\partial_r^3 B-3 \partial_{x_0}^2 \partial_r B-3 \partial_{x_0} \partial_r^2A-\frac{12}{r}  \partial_{x_0}^2 B- \frac{12}{r} \partial_{x_0}\partial_{r} A+ 8 \frac{\partial_r^2 B}{r}+8 \frac{\partial_r B}{r^2}-8\frac{B}{r^3}=0\\
\\
\partial_{x_0}^3 B- \partial_{r}^3 A +3 \partial_{x_0}^2 \partial_{r} A-3 \partial_{x_0} \partial_{r}^2 B- 12 \frac{\partial_{x_0} \partial_{r}B}{r}+12 \frac{\partial_{x_0} B}{r^2}-4 \frac{\partial_{r}^2 A}{r}+ \frac{4}{r^2}\partial_r A=0.
\end{cases}
$$
\end{theorem}
\begin{proof}
First of all we start by computing $\mathcal{D}^2\breve{f}^{\circ}_0(x)$.
$$ \mathcal{D}^2 \breve{f}^{\circ}_0(x)= \mathcal{D}(\mathcal{D}\breve{f}^{\circ}_0(x))= \mathcal{D}(A'+ \underline{\omega}B'),$$
where
$$ A':= \partial_{x_0} A- \partial_{r} B- \frac{4}{r} B\ \ \ \ \ {\rm and} \ \ \ \ \
 B':= \partial_{x_0} B+ \partial_r A.
 $$
By formula \eqref{Dirac} we get
\begin{eqnarray}
\label{zeros4}
\nonumber
\mathcal{D}^2 \breve{f}^{\circ}_0(x)&=& \left(\partial_{x_0}A'(x_0,r)- \partial_r B'(x_0,r)- \frac{4}{r} B'(x_0,r)\right)+ \underline{\omega} \left(\partial_{x_0}B'(x_0,r)+ \partial_r A'(x_0,r)\right)\\
\nonumber
&=& \left(\partial_{x_0}^2 A-2 \partial_{x_0}\partial_r B - \frac{8}{r} \partial_{x_0}B- \partial_r^2A- \frac{4}{r} \partial_r A\right)+ \underline{\omega} \biggl(\partial_{x_0}^2 B+2 \partial_{x_0}\partial A- \partial_r^2B+\\
&&- 4 \partial_r \frac{\partial_r B}{r}+ 4 \frac{B}{r^2} \biggl)\\
\nonumber
&=& A''+ \underline{\omega} B''.
\end{eqnarray}
where
$$ A'':=\partial_{x_0}^2 A-2 \partial_{x_0}\partial_r B - \frac{8}{r} \partial_{x_0}B- \partial_r^2A- \frac{4}{r} \partial_r A$$
and
$$ B'':=\partial_{x_0}^2 B+2 \partial_{x_0}\partial A- \partial_r^2B- 4 \partial_r \frac{\partial_r B}{r}+ 4 \frac{B}{r^2}.$$
Finally by applying another time formula \eqref{Dirac} we get
\begin{eqnarray*}
\mathcal{D}^3 \breve{f}^{\circ}_0(x)&=& \mathcal{D}(\mathcal{D}^2 \breve{f}^{\circ}_0(x))=\mathcal{D}(A''+ \underline{\omega} B'')\\
&=& \left(\partial_{x_0}A''- \partial_r B''- \frac{4}{r} B''\right)+ \underline{\omega} (\partial_{x_0}B''+ \partial_r A'')\\
&=& \biggl( \partial_{x_0}^3A+\partial_r^3 B-3 \partial_{x_0}^2 \partial_r B-3 \partial_{x_0} \partial_r^2A-\frac{12}{r}  \partial_{x_0}^2 B- \frac{12}{r} \partial_{x_0}\partial_{r} A+ 8 \frac{\partial_r^2 B}{r}+8 \frac{\partial_r B}{r^2}-8\frac{B}{r^3}\biggl)\\
&& + \underline{\omega} \biggl(\partial_{x_0}^3 B- \partial_{r}^3 A +3 \partial_{x_0}^2 \partial_{r} A -3 \partial_{x_0} \partial_{r}^2 B- 12 \frac{\partial_{x_0} \partial_{r}B}{r}+12 \frac{\partial_{x_0} B}{r^2}-4 \frac{\partial_{r}^2 A}{r} + \frac{4}{r^2}\partial_r A\biggl).
\end{eqnarray*}
We get the statement from the fact that the function $\breve{f}^{\circ}_0(x)$ is polyanalytic of order three, i.e. $\mathcal{D}^3 \breve{f}^{\circ}_0(x)=0.$
\end{proof}
\begin{theorem}
Let $D \subseteq \mathbb{C}$. Let $\Omega_D$ be an axially symmetric open set in $ \mathbb{R}^6$.
 Then
\begin{itemize}
\item $f^{\circ}(x)=A(x_0,r)+ \underline{\omega}B(x_0,r)$ is axially Cliffordian of order one if and only if $A:=A(x_0,r)$ and $B:=B(x_0,r)$ satisfy the following system
$$ \begin{cases}
\partial_{x_0}A+ \partial_{x_0}\partial_{r}^2A + \frac{4}{r} \partial_{x_0}\partial_{r} A- \partial_r \partial_{x_0}^2 B-\partial_r^3 B- 8\frac{ \partial_r B}{r^2}+8\frac{B}{r^3}-4\frac{\partial_{x_0}^2B}{r}=0\\
\\
\partial_{x_0}^3B+ \partial_{x_0}\partial_r^2B+4 \partial_{r} \left( \frac{\partial_{x_0} B}{r}\right)+\partial_r \partial_{x_0}^2 A+ \partial_{r}^3A+4 \partial_{r}^2 \left(\frac{A}{r}\right)=0.
\end{cases}
$$
\item $\tilde{f}_0(x)=A(x_0,r)+ \underline{\omega}B(x_0,r)$ is axially harmonic if and only if $A:=A(x_0,r)$ and $B:=B(x_0,r)$ satisfy the following system
$$
\begin{cases}
\partial_{x_0}^2 A+ \partial_r^2A+ \frac{4}{r} \partial_r A=0\\
\\
\partial_{x_0}^2 B+ \partial_r^2B+ 4 \partial_r \left(\frac{B}{r}\right) =0.
\end{cases}
$$
\item $\breve{f}^{\circ}_1(x)=A(x_0,r)+ \underline{\omega}B(x_0,r)$ is axially polyanalytic of order two if and only if $A:=A(x_0,r)$ and $B:=B(x_0,r)$ satisfy the following system
$$ \begin{cases}
\partial_{x_0}^2 A-2 \partial_{x_0}\partial_r B- \frac{8}{r} \partial_{x_0}B- \partial_r^2A- \frac{4}{r} \partial_r A=0\\
\\
\partial_{x_0}^2 B+2 \partial_{x_0}\partial_r A- \partial_r^2B- 4 \partial_r \left(\frac{B}{r} \right)=0.
\end{cases}
$$
\item $\breve{f}^{\circ}(x)=A(x_0,r)+ \underline{\omega}B(x_0,r)$ is axially Cliffordian  polyanalytic of order $(1,2)$ if and only if $A:=A(x_0,r)$ and $B:=B(x_0,r)$ satisfy the following system
\[
\begin{cases}
\partial_{x_0}^4A-2 \partial_{r} \partial_{x_0}^3B-2 \partial_{x_0} \partial_r^3 B- 8\frac{\partial_{x_0}^3B}{r}- 8\frac{\partial_{x_0}\partial_r^2 B}{r}- \partial_r^4 A-8 \frac{\partial_r^3 A}{r}
- 8 \frac{\partial_r A}{r^3}-4 \frac{\partial_r^2 A}{r^2}- 8\frac{A}{r^4}-16 \frac{\partial_{x_0}\partial_r B}{r^2}=0\\
\\
\partial_{x_0}^4B+2 \partial_r \partial_{x_0}^3A+2 \partial_{x_0}\partial_r^3 A+ 8 \frac{\partial_r^2 \partial_{x_0}A}{r}-12 \frac{\partial_r \partial_{x_0}A}{r^2}-4 \frac{\partial_r \partial_{x_0}B}{r^2}- \partial_r^4B+8 \frac{\partial_{x_0}A}{r^3}-8 \frac{\partial_{r}^2B}{r^2}+24 \frac{\partial_r B}{r^3}\\
-24 \frac{B}{r^4}+4 \frac{\partial_{x_0}^2B}{r^2}=0.
\end{cases}
\]
\end{itemize}
\end{theorem}
\begin{proof}
We do not give all the details of the proof because they are tedious computations we just mentions that:
the first system follows by applying the Dirac operator $ \mathcal{D}$ to \eqref{zero}. The computations are similar to that ones done in Theorem \ref{anti2}.
The second system follows by formula \eqref{zero}.
 The third system follows by formula \eqref{zeros4}.
 The fourth system follows by applying the operator $ \mathcal{D}$ to $ \Delta \mathcal{D}\breve{f}^{\circ}(x)$, which formula is possible to get by the first point of this theorem.
\end{proof}

In the next section we will give the integral representation of the functions
belonging to the function spaces associated with the fine structure.
These spaces are called, for short, fine structure spaces.

\section{Integral representation of the functions of the fine structure spaces}\label{INTREP}

We now recall the slice hyperholomorphic Cauchy formulas that are fundamental for the hyperholomorphic spectral theories.

\begin{theorem}\label{ts}
	Let $s$, $x\in \mathbb R^{n+1}$ with $|x|<|s|$, then
	$$\sum_{n=0}^{+\infty} x^ns^{-n-1}=-(x^2-2\Re(s)x+|s|^2)^{-1}(x-\overline s)$$
	and
	$$\sum_{n=0}^{+\infty} s^{-n-1}x^n=-(x-\overline s)(x^2-2\Re(s)x+|s|^2)^{-1}.$$
	Moreover, for any $s,\, x\in \mathbb R^{n+1}$ with $x\notin [s]$, we have
	$$ -(x^2-2\Re(s)x+|s|^2)^{-1}(x-\overline s)=(s-\overline x)(s^2-2\Re(x)s+|x|^2)^{-1} $$
	and
	$$ -(x-\overline s)(x^2-2\Re(s)x+|s|^2)^{-1}=(s^2-2\Re(x)s+|x|^2)^{-1}(s-\overline x).$$
\end{theorem}
In view of Theorem \ref{ts} there are two possible representations of the Cauchy kernels for left slice hyperholomorphic functions and two for right slice hyperholomorphic functions.
\begin{definition}
	Let $s,\, x\in \mathbb R^{n+1}$ with $x\notin [s]$ then we define the two functions
	$$ \PRes_{s}(x)^{-1}:=(x^2-2\Re(s)x+|s|^2)^{-1} ,\ \ \ \
	\PRes_{c,s}(x)^{-1}:=(s^2-2\Re(x)s+|x|^2)^{-1},
	$$
\end{definition}
that are called pseudo Cauchy kernel and commutative pseudo Cauchy kernel, respectively.
\begin{definition}\label{d1}
	Let $s,\, x\in \mathbb R^{n+1}$ with $x\notin [s]$ then
	\begin{itemize}
		\item We say that the left slice hyperholomorphic Cauchy kernel  $S^{-1}_L(s,x)$ is written in the form I if
		$$S^{-1}_L(s,x):=\PRes_{s}(x)^{-1}(\overline s-x).$$
		\item We say that the right slice hyperholomorphic Cauchy kernel  $S^{-1}_R(s,x)$ is written in the form I if
		$$S^{-1}_R(s,x):= (\overline s-x)\PRes_{s}(x)^{-1}.$$
		\item We say that the left slice hyperholomorphic Cauchy kernel  $S^{-1}_L(s,x)$ is written in the form II if
		$$S^{-1}_L(s,x):=(s-\overline x) \PRes_{c,s}(x)^{-1}.$$
		\item We say that the right slice hyperholomorphic Cauchy kernel  $S^{-1}_R(s,x)$ is written in the form II if
		$$S^{-1}_R(s,x):= \PRes_{c,s}(x)^{-1}(s- \overline x).$$
	\end{itemize}
\end{definition}
In this article, otherwise specified, we refer to $S^{-1}_L(s,x)$ and $S^{-1}_R(s,x)$ as written in the form II.

We have the following regularity for the (left and right) slice hyperholomorphic Cauchy kernels.
\begin{lemma} Let  $s\notin [x]$.
	The left slice hyperholomorphic Cauchy kernel $S_L^{-1}(s,x)$ is left slice hyperholomorphic in $x$ and right slice hyperholomorphic in $s$. The right slice hyperholomorphic Cauchy kernel $S_R^{-1}(s,x)$ is left slice hyperholomorphic in $s$ and right slice hyperholomorphic in $x$.
\end{lemma}

\begin{theorem}[The Cauchy formulas for slice hyperholomorphic functions]
	\label{Cauchy}
	Let $U\subset\mathbb{R}^{n+1}$ be a bounded slice Cauchy domain, let $J\in\mathbb{S}$ and set  $ds_J=ds (-J)$.
	If $f$ is a (left) slice hyperholomorphic function on a set that contains $\overline{U}$ then
	\begin{equation}
		f(x)=\frac{1}{2 \pi}\int_{\partial (U\cap \mathbb{C}_J)} S_L^{-1}(s,x)\, ds_J\,  f(s),\qquad\text{for any }\ \  x\in U.
	\end{equation}
	If $f$ is a right slice hyperholomorphic function on a set that contains $\overline{U}$,
	then
	\begin{equation}\label{Cauchyright}
		f(x)=\frac{1}{2 \pi}\int_{\partial (U\cap \mathbb{C}_J)}  f(s)\, ds_J\, S_R^{-1}(s,x),\qquad\text{for any }\ \  x\in U.
	\end{equation}
	These integrals  depend neither on $U$ nor on the imaginary unit $J\in\mathbb{S}$.
\end{theorem}

Now, we recall what happens when we apply the operator $T_{SF2}:= \Delta^{\frac{n-1}2}$, to the slice hyperholomorphic Cauchy kernel (see \cite{CSS10}).

\begin{proposition}
	\label{reg}
	Let $x$, $s \in \mathbb{R}^{n+1}$ and $x \not\in [s]$. Then:
	\begin{itemize}
		\item
		The function $\Delta^{\frac{n-1}{2}} S_L^{-1}(s,x)$ is a left monogenic function in the variable $x$ and right slice hyperholomorphic in $s$.
		\item
		The function  $\Delta^{\frac{n-1}{2}} S_R^{-1}(s,x)$ is a right monogenic function in the variable $x$ and left slice hyperholomorphic in $s$.
	\end{itemize}
\end{proposition}

\begin{definition}[$F_n$-kernels]
	Let $n$ be an odd number and let $x,\, s\in\rr^{n+1}$. We define, for $s\notin[x]$, the $F^L_n$-kernel and the $F^R_n$-kernel as
	$$
	F^L_n(s,x):=\Delta^{\frac{n-1}2}S^{-1}_L(s,x), \ \ \ \ \ F^R_n(s,x):=\Delta^{\frac{n-1}{2}}S^{-1}_R(s,x),
	$$
respectively.
\end{definition}
It is possible to compute explicitly the $F_n$-kernels, in particular, for $s\notin [x]$ we have
\begin{equation}
\label{fu1}
F^L_n(s,x)=\gamma_n(s-\bar x)Q_{c,s}(x)^{-\frac{n+1}2},
\end{equation}
and
\begin{equation}
\label{fu2}
F^R_n(s,x)=\gamma_n Q_{c,s}(x)^{-\frac{n+1}2}(s-\bar x),
\end{equation}
where
$$
\gamma_n:=(-1)^{\frac{n-1}2}2^{n-1}[(\frac 12(n-1))!]^2.
$$
\begin{remark}
Formula \eqref{fu1} and \eqref{fu2} hold also when $n$ is even, see \cite{CDQS}.
\end{remark}
 Moreover, for $s\notin [x]$, the $F_n$-kernels satisfy the following equations
\begin{equation}
	\label{p0}
	F_{n}^L(s,x)s-xF_{n}^L(s,x) =\gamma_n \mathcal{Q}_{c,s}^{-\frac{n-1}{2}}(x)
\end{equation}
and
\begin{equation}
	sF_{n}^R(s,x)-F_{n}^R(s,x) x=\gamma_n \mathcal{Q}_{c,s}^{-\frac{n-1}{2}}(x).
\end{equation}
The following result plays a key role (see \cite{CSS10}).
\begin{theorem}[The Fueter-Sce mapping theorem in integral form]
	\label{Fueter}
	Let $n$ be an odd number. Let $U\subset\mathbb{R}^{n+1}$ be a slice Cauchy domain, let $J\in\mathbb{S}$ and set  $ds_J=ds (-J)$.
	\begin{itemize}
		\item
		If $f$ is a (left) slice hyperholomorphic function on a set $W$, such that $\overline{U} \subset W$, then
		the left monogenic function  $\breve{f}(x)=\Delta^{\frac{n-1}2} f(x)$
		admits the integral representation
		\begin{equation}\label{FuetLSEC}
			\breve{f}(x)=\frac{1}{2 \pi}\int_{\partial (U\cap \mathbb{C}_J)} F_n^L(s,x)ds_J f(s).
		\end{equation}
		\item
		If $f$ is a right slice hyperholomorphic function on a set $W$, such that $\overline{U} \subset W$, then
		the right monogenic function $\breve{f}(x)=\Delta^{\frac{n-1}2} f(x)$
		admits the integral representation
		\begin{equation}\label{FuetRSCE}
			\breve{f}(x)=\frac{1}{2 \pi}\int_{\partial (U\cap \mathbb{C}_J)} f(s)ds_J F_n^R(s,x).
		\end{equation}
	\end{itemize}
	The integrals  depend neither on $U$ and nor on the imaginary unit $J\in\mathbb{S}$.
\end{theorem}

\subsection{The structure of the kernels of the fine structure spaces and their regularity}

In this subsection we construct the kernels associated with the
spaces of the fine structures.
The strategy follows the construction of the Fueter-Sce mapping theorem in integral form.
Precisely, we proceed by applying to the
left (and to the right) slice hyperholomorphic
 Cauchy kernels  some suitable operators that define the required kernels.
These new kernels
 will be used to give these functions the appropriate integral representations.
 In the proofs of the following theorems we consider just
 the left hyperholomorphic Cauchy kernel since for
  the right hyperholomorphic Cauchy kernel computations are similar.

\begin{theorem}[Structure of the slice $ \mathcal{D}$-kernels $S^{-1}_{\mathcal{D},L}$ and $S^{-1}_{\mathcal{D},R}$]
	\label{app1}
	Let $s,x \in \mathbb{R}^6$ be such that $x \notin [s]$, then
	\begin{equation}
		\label{one}
		S^{-1}_{\bigD ,L}(s,x):=\mathcal{D} \left(S^{-1}_L(s,x)\right)=-4 \mathcal{Q}_{c,s}(x)^{-1},
	\end{equation}
	and
	\begin{equation}
		\label{second}
		S^{-1}_{\bigD ,R}(s,x):=\left(S^{-1}_R(s,x)\right)\mathcal{D} =-4 \mathcal{Q}_{c,s}(x)^{-1}.
	\end{equation}
	We denote by $S^{-1}_{\mathcal{D},L}$ and $S^{-1}_{\mathcal{D},R}$ the left and the right slice $ \mathcal{D}$-kernels.
\end{theorem}
\begin{proof}
	We compute the following derivatives
	\begin{equation}
		\label{f1}
		\frac{\partial}{\partial x_0}S^{-1}_L(s,x)=- \mathcal{Q}_{c,s}(x)^{-1}+2(s- \bar{x}) \mathcal{Q}_{c,s}(x)^{-2}(s- x_0),
	\end{equation}
	and, for $1 \leq i \leq 5$, we get
	\begin{equation}
		\label{f2}
		\frac{\partial}{\partial x_i} S^{-1}_L(s,x)= e_i \mathcal{Q}_{c,s}(x)^{-1}-2 x_i (s- \bar{x}) \mathcal{Q}_{c,s}(x)^{-2}.
	\end{equation}
	Finally, we have
	\begin{eqnarray*}
		\mathcal{D} \left(S^{-1}_L(s,x)\right)&=& \frac{\partial}{\partial x_0}S^{-1}_L(s,x)+ \sum_{i=1}^5 e_i \frac{\partial}{\partial x_i} S^{-1}_L(s,x)\\
		&=& - \mathcal{Q}_{c,s}(x)^{-1}-2 x_0(s- \bar{x}) \mathcal{Q}_{c,s}(x)^{-2}+2(s- \bar{x}) s \mathcal{Q}_{c,s}(x)^{-2}+\\
		&& -5 \mathcal{Q}_{c,s}(x)^{-1}-2 \underline{x}(s- \bar{x}) \mathcal{Q}_{c,s}(x)^{-2} \\
		&=& - 6 \mathcal{Q}_{c,s}(x)^{-1}-2 x(s- \bar{x}) \mathcal{Q}_{c,s}(x)^{-2}+2(s- \bar{x}) s \mathcal{Q}_{c,s}(x)^{-2}\\
		&=&- 6 \mathcal{Q}_{c,s}(x)^{-1}+2(s^2- \bar{x}s-xs+|x|^2)\mathcal{Q}_{c,s}(x)^{-2}\\
		&=& -4 \mathcal{Q}_{c,s}(x)^{-1}.
	\end{eqnarray*}
\end{proof}
Now, we apply the Laplacian of $ \mathbb{R}^6$, i.e.,
$$ \Delta:= \sum_{i=0}^5 \frac{\partial^2}{\partial x_i^2},$$
to the slice hyperholomorphic Cauchy kernel
\begin{theorem}[Structure of the slice $ \Delta$-kernels $S^{-1}_{\Delta,L}$ and $S^{-1}_{\Delta,R}$]
	\label{app2}
	Let $s,x \in \mathbb{R}^6$ be such that $x \notin [s]$, then
	\begin{equation}
		\label{third}
		S^{-1}_{\Delta ,L}(s,x):=\Delta S^{-1}_L(s,x)=-8S^{-1}_L(s,x) \mathcal{Q}_{c,s}(x)^{-1},
	\end{equation}
	and
	\begin{equation}
		\label{fourth}
		S^{-1}_{\Delta ,R}(s,x):=\Delta S^{-1}_R(s,x)=-8 \mathcal{Q}_{c,s}(x)^{-1}S^{-1}_R(s,x) .
	\end{equation}
	We denote by $S^{-1}_{\Delta,L}$ and $S^{-1}_{\Delta,R}$ the left and the right slice $ \Delta$-kernels.
\end{theorem}
\begin{proof}
	By formula \eqref{f1} we get
	\begin{eqnarray*}
		\frac{\partial^2}{\partial x_0^2} S^{-1}_L(s,x)&=& (-2s+2x_0) \mathcal{Q}_{c,s}(x)^{-2}+(2x_0-2s) \mathcal{Q}_{c,s}(x)^{-2}-2(s-\bar{x}) \mathcal{Q}_{c,s}(x)^{-2}+\\
		&& +8 (s- \bar{x})(x_0-s)^2 \mathcal{Q}_{c,s}(x)^{-3}.
	\end{eqnarray*}
	By formula \eqref{f2}, for $1 \leq i \leq 5$, we get
	\begin{eqnarray*}
		\frac{\partial}{\partial x_i^2} S^{-1}_L(s,x)&=& -2x_i e_i \mathcal{Q}_{c,s}(x)^{-2}-2 x_i e_i \mathcal{Q}_{c,s}(x)^{-2}\\
		&& -2(s- \bar{x}) \mathcal{Q}_{c,s}(x)^{-2}+8 (s- \bar{x}) x_i^2 \mathcal{Q}_{c,s}(x)^{-3}.
	\end{eqnarray*}
	Finally, we get
	\begin{eqnarray*}
		\Delta S^{-1}_L(s,x)&=& \frac{\partial^2}{\partial x_0^2}S^{-1}_L(s,x)+ \sum_{i=1}^5 \frac{\partial^2}{\partial x_i^2}S^{-1}_L(s,x)\\
		&=& 4(x_0-s) \mathcal{Q}_{c,s}(x)^{-2}-2(s- \bar{x}) \mathcal{Q}_{c,s}(x)^{-2}+8(s- \bar{x}) (x_0-s)^2 \mathcal{Q}_{c,s}(x)^{-3}\\
		&& -4 \underline{x} \mathcal{Q}_{c,s}(x)^{-2}-10 (s- \bar{x}) \mathcal{Q}_{c,s}(x)^{-2}+8 | \underline{x}|^2 (s- \bar{x}) \mathcal{Q}_{c,s}(x)^{-3}\\
		&=& -4(s-x_0+ \underline{x}) \mathcal{Q}_{c,s}(x)^{-2}-12(s- \bar{x}) \mathcal{Q}_{c,s}(x)^{-2}+\\
		&& + 8(s- \bar{x}) [(x_0-s)^2+| \underline{x}|^2] \mathcal{Q}_{c,s}(x)^{-3}\\
		&=& -16(s- \bar{x}) \mathcal{Q}_{c,s}(x)^{-2}+8(s- \bar{x})(x_0^2+s^2-2x_0s+| \underline{x}|^2) \mathcal{Q}_{c,s}(x)^{-3}\\
		&=& -8(s- \bar{x}) \mathcal{Q}_{c,s}(x)^{-2}\\
		&=& -8 S_{L}^{-1}(s,x) \mathcal{Q}_{c,s}(x)^{-1}.
	\end{eqnarray*}
\end{proof}

\begin{theorem}[Structure of the slice $ \Delta\mathcal{D}$-kernels $S^{-1}_{\Delta\mathcal{D},L}$ and $S^{-1}_{\Delta\mathcal{D},R}$ ]
	\label{app3}
	Let $x$, $s \in \mathbb{R}^6$ be such that $x \notin [s]$, then
	\begin{equation}
		\label{five}
		S^{-1}_{\Delta\bigD ,L}(s,x):=\Delta \mathcal{D} \left(S^{-1}_L(s,x)\right)=16 \mathcal{Q}_{c,s}(x)^{-2},
	\end{equation}
	and
	\begin{equation}
		\label{six}
		S^{-1}_{\Delta\bigD ,R}(s,x):=\left(S^{-1}_R(s,x)\right)\mathcal{D} \Delta=16 \mathcal{Q}_{c,s}(x)^{-2}.
	\end{equation}
	We denote by $S^{-1}_{\Delta\mathcal{D},L}$ and $S^{-1}_{\Delta\mathcal{D},R}$ the left and the right slice $ \Delta\mathcal{D}$-kernels.
\end{theorem}
\begin{proof}
	In order to show formula \eqref{five} it is enough to apply the Dirac operator to \eqref{third}. Thus, we have
	$$\frac{\partial}{\partial x_0} \left( \Delta S^{-1}_L(s,x)\right)=-8[-\mathcal{Q}_{c,s}(x)^{-2}-2(s- \bar{x}) \mathcal{Q}_{c,s}(x)^{-3}(2x_0-2s)].$$
	For $1 \leq i \leq 5$, we have
	$$\frac{\partial}{\partial x_i} \left( \Delta S^{-1}_L(s,x)\right)=-8[e_i\mathcal{Q}_{c,s}(x)^{-2}-4(s- \bar{x}) \mathcal{Q}_{c,s}(x)^{-3}x_i].$$
	Finally by formula \eqref{p0}, with $n=5$, we have
	\begin{eqnarray*}
		\mathcal{D} \Delta S^{-1}_L(s,x)&=& \frac{\partial}{\partial x_0} \left( \Delta S^{-1}_L(s,x)\right)+ \sum_{i=1}^5 e_i \frac{\partial}{\partial x_i} \left( \Delta S^{-1}_L(s,x)\right)\\
		&=&-8[-6 \mathcal{Q}_{c,s}(x)^{-2}+\frac{4}{\gamma_5} F_{L}^5(s,x)(s-x_0)- \frac{4}{\gamma_5} \underline{x} F_L^5(s,x)]\\
		&=& -8 \left[-6 \mathcal{Q}_{c,s}(x)^{-2}+ \frac{4}{\gamma_5} \left( F_L^5(s,x)s-x F_L^5(s,x)\right)\right]\\
		&=&-8 \left(-6 \mathcal{Q}_{c,s}(x)^{-2}+4 \mathcal{Q}_{c,s}(x)^{-2}\right)\\
		&=& 16 \mathcal{Q}_{c,s}(x)^{-2}.
	\end{eqnarray*}
	Formula \eqref{six} follows with similar reasoning.
	
\end{proof}

\begin{theorem}[Structure of the slice $ \mathcal{\overline{D}}$-kernels $S^{-1}_{\mathcal{\overline{D}},L}$ and $S^{-1}_{\mathcal{\overline{D}},R}$]\label{app4}
	Let $x$, $s \in \mathbb{R}^6$ be such that $x \notin [s]$, then
	\begin{equation}\label{seven}
		S^{-1}_{\overline\bigD ,L}(s,x):=\overline \bigD (S^{-1}_L(s,x))=4(s-\bar x)\qcs^{-2}(s-x_0)+2\qcs^{-1}
	\end{equation}
	and
	\begin{equation}\label{eight}
		S^{-1}_{\overline\bigD ,R}(s,x):=(S^{-1}_R(s,x)) \overline \bigD=4(s-x_0)\qcs^{-2}(s-\bar x)+2\qcs^{-1}.
	\end{equation}
	We denote by $S^{-1}_{\mathcal{\overline{D}},L}$ and $S^{-1}_{\mathcal{\overline{D}},R}$ the left and the right slice $ \mathcal{\overline{D}}$-kernels.
\end{theorem}
\begin{proof}
	By the relations \eqref{f1} and \eqref{f2} and using the fact that $2x_0=x+\bar x$, we have that
	\[
	\begin{split}
		\overline \bigD(S^{-1}_L(s,x))&=\frac{\partial}{\partial x_0}(S^{-1}_L(s,x))-\sum_{i=1}^5 e_i\frac{\partial}{\partial x_i}(S^{-1}_L(s,x))\\
		&= -\qcs(x)^{-1}+5\qcs(x)^{-1}+2(s- \bar{x}) \mathcal{Q}_{c,s}(x)^{-2}(s- x_0) +2 \underline x (s- \bar{x}) \mathcal{Q}_{c,s}(x)^{-2}\\
		&=4\qcs^{-1}+2((s-\bar x)s-\bar x(s-\bar x))\qcs^{-2}\\
		&=4\qcs\qcs^{-2}+2((s-\bar x)s-\bar x(s-\bar x))\qcs^{-2}\\
		&=(6s^2-8sx_0+4|x|^2-4\bar xs +2\bar x^2)\qcs^{-2}\\
		&=(6s^2-4\bar xs-4xs+4|x|^2-4\bar xs +2\bar x^2)\qcs^{-2}\\
		&=[4(s^2-\bar xs)-2(\bar xs+xs-|x|^2-\bar x^2)+2(s^2-\bar x s+|x|^2 -xs)]\qcs^{-2}\\
		&=4(s-\bar x)s\qcs^{-2}-4x_0(s-\bar x)\qcs^{-2}+2((s-\bar x)s+x(\bar x -s))\qcs^{-2}\\
		&=4(s-\bar x)\qcs^{-2}(s-x_0)+2\qcs^{-1}.
	\end{split}
	\]
\end{proof}

\begin{theorem}[Structure of the slice $ \mathcal{\overline{D}}^2$-kernels $S^{-1}_{\mathcal{\overline{D}}^2,L}$ and $S^{-1}_{\mathcal{\overline{D}}^2,R}$]\label{app5}
	Let $x$, $s \in \mathbb{R}^6$ be such that $x \notin [s]$, then
	\begin{equation}\label{nine}
		S^{-1}_{\overline\bigD^2 ,L}(s,x):=\overline \bigD ^2 (S^{-1}_L(s,x))=32(s-\bar x)\qcs^{-3}(s-x_0)^2
	\end{equation}
	and
	\begin{equation}\label{ten}
		S^{-1}_{\overline\bigD^2 ,R}(s,x):=(S^{-1}_R(s,x)) \overline \bigD^2=32(s-x_0)^2\qcs^{-3} (s-\bar x).
	\end{equation}
	We denote by $S^{-1}_{\mathcal{\overline{D}}^2,L}$ and $S^{-1}_{\mathcal{\overline{D}}^2,R}$ the left and the right slice $ \mathcal{\overline{D}}^2$-kernels.
\end{theorem}
\begin{proof}
	It is useful to compute $\overline{\mathcal D}((s-\bar x)\qcs^{-2})$. By the relations:
	$$ \frac{\partial}{\partial{x_0}} ((s-\bar x)\qcs ^{-2})=-\qcs^{-2}-2(s-\bar x)\qcs^{-3}(-2s+2x_0) $$
	and
	$$ \frac{\partial}{\partial{x_i}}((s-\bar x)\qcs^{-2})=e_i\qcs^{-2}-4(s-\bar x)\qcs^{-3}(x_i), $$
	we have
	\begin{equation}\label{nine_bis}
		\begin{split}
			\overline\bigD((s-\bar x)\qcs^{-2})&=\frac{\partial}{\partial_{x_0}}((s-\bar x)\qcs^{-2})-\sum_{i=1}^5e_i\frac{\partial}{\partial_{x_i}}((s-\bar x)\qcs^{-2})\\
			&=4\qcs^{-2}+4((s-\bar x)s-\bar x(s-\bar x))\qcs^{-3}\\
			&=4\qcs\qcs^{-3}+4((s-\bar x)s-\bar x(s-\bar x))\qcs^{-3}\\
			&=(8s^2-8sx_0+4|x|^2-8\bar xs+4\bar x^2)\qcs^{-3}\\
			&=(8s^2-8\bar xs)\qcs^{-3} -(8x_0s-8x_0\bar x)\qcs^{-3}\\
			&=8(s-\bar x)\qcs^{-3}(s-x_0),
		\end{split}
	\end{equation}
	where in the fourth equality we used $|x|^2+\bar x^2=2x_0\bar x$. Now, using formula \eqref{seven} and Leibnitz formula for $\mathcal{\overline{D}}$, we can compute $\overline\bigD^2 (S^{-1}_L(s,x))$:
	\[
	\begin{split}
		&\overline\bigD^2 (S^{-1}_L(s,x))=\overline\bigD [4(s-\bar x)\qcs^{-2}(s-x_0)+2\qcs^{-1}]\\
		&= 4\overline\bigD [(s-\bar x)\qcs^{-2}] (s-x_0)-4(s-\bar x)\qcs^{-2}+2\overline \bigD[\qcs^{-1}]\\
		&= 4(8(s-\bar x)\qcs^{-3} s-8(s-\bar x)\qcs^{-3}x_0)(s-x_0)-4(s-\bar x)\qcs^{-2}+4(s-\bar x)\qcs^{-2}\\
		&=32(s-\bar x)\qcs^{-3}(s-x_0)^2.
	\end{split}
	\]
\end{proof}

\begin{theorem}[Structure of the slice $ \bigD^2$-kernels $S^{-1}_{\bigD^2,L}$ and $S^{-1}_{\bigD^2,R}$]\label{app6}
	Let $x$, $s \in \mathbb{R}^6$ be such that $x \notin [s]$, then
	\begin{equation}\label{eleven}
		S^{-1}_{\bigD^2 ,L}(s,x):=\bigD ^2 (S^{-1}_L(s,x))=8S^{-1}_L (s, \bar{x})\qcs^{-1}
	\end{equation}
	and
	\begin{equation}\label{twelve}
		S^{-1}_{\bigD^2 ,R}(s,x):=(S^{-1}_R(s,x)) \bigD^2=8\qcs^{-1} S^{-1}_R (s, \bar{x}).
	\end{equation}
	We denote by $S^{-1}_{\bigD^2,L}$ and $S^{-1}_{\bigD^2,R}$ the left and the right slice $ \bigD^2$-kernels.
\end{theorem}
\begin{proof}
	By \eqref{one} and relations \eqref{f1}, \eqref{f2}, we have that
	\[
	\begin{split}
		\bigD ^2 (S^{-1}_L(s,x))=-4\bigD (\qcs^{-1})=-4(2s-2x_0-2\sum_{i=1}^5e_ix_i)\qcs^{-2}=8(x-s)\qcs^{-2}.
	\end{split}
	\]
\end{proof}

\begin{theorem}[Structure of the slice $ \Delta\mathcal{\overline{D}}$-kernels
$S^{-1}_{\Delta \mathcal{\overline{D}},L}$ and $S^{-1}_{\Delta\mathcal{\overline{D}},R}$]\label{app7}
	Let $x$, $s \in \mathbb{R}^6$ be such that $x \notin [s]$, then
	\begin{equation}\label{thirteen}
		S^{-1}_{\Delta\overline\bigD ,L}(s,x):=\Delta \overline \bigD (S^{-1}_L(s,x))=-64(s-\bar x)\qcs^{-3}(s-x_0)
	\end{equation}
	and
	\begin{equation}\label{fourteen}
		S^{-1}_{\Delta\overline\bigD ,R}(s,x):= (S^{-1}_R(s,x)) \Delta \overline \bigD=-64 (s-x_0)\qcs^{-3} (s-\bar x).
	\end{equation}
	We denote by $S^{-1}_{\Delta \mathcal{\overline{D}},L}$ and $S^{-1}_{\Delta\mathcal{\overline{D}},R}$ the left and the right slice $ \Delta\mathcal{\overline{D}}$-kernels.
\end{theorem}
\begin{proof}
	We show only formula \eqref{thirteen} because it is possible to prove formula \eqref{fourteen} with similar arguments. By formulas \eqref{third} and \eqref{nine_bis}, we have that
	\[
	\begin{split}
		\Delta \overline \bigD (S^{-1}_L(s,x))=-8\overline\bigD ((s-\bar x)\qcs^{-2})=-64(s-\bar x)\qcs^{-3}(s-x_0).
	\end{split}
	\]
\end{proof}

Now, we study the regularity of the previous kernels.

\begin{proposition}\label{reg1}
	Let $x$, $s \in \mathbb{R}^6$ be such that $x \notin [s]$. We have that
	\begin{enumerate}
		\item $S^{-1}_{\Delta^{1-\ell}\bigD, L}(s,x)$ (resp. $S^{-1}_{\Delta^{2-\ell}\bigD, R}(s,x)$) is slice right (resp. left) hyperholomorphic in $s$ and it is $\ell+1$-harmonic in $x$ for $0\leq \ell\leq 1$;\\
		\item $S^{-1}_{\Delta, L}(s,x)$ (resp. $S^{-1}_{\Delta, R}(s,x)$) is slice right (resp. left) hyperholomorphic in $s$ and it is left (resp. right) holomorphic Cliffordian of order $1$ in $x$;\\
		\item $S^{-1}_{\overline\bigD, L}(s,x)$ (resp. $S^{-1}_{\overline\bigD, R}(s,x)$) is slice right (resp. left) hyperholomorphic in $s$ and it is left (resp. right) polyanalytic Cliffordian of order $(1,2)$ in $x$;\\
		\item $S^{-1}_{\Delta^\ell \overline\bigD^{2-l}, L}(s,x)$ (resp. $S^{-1}_{\Delta^\ell \overline\bigD^{2-l}, R}(s,x)$) is slice right (resp. left) hyperholomorphic in $s$ and it is left (resp. right) polyanalytic of order $3-\ell$ in $x$ for $0\leq \ell\leq 1$;\\
		\item $S^{-1}_{\bigD^{2}, L}(s,x)$ (resp. $S^{-1}_{\bigD^{2}, R}(s,x)$) is slice right (resp. left) hyperholomorphic in $s$ and it is left (resp. right) anti-holomorphic Cliffordian of order $1$ in $x$.\\
	\end{enumerate}
\end{proposition}
\begin{proof}
	We prove the regularity for the left kernels since for the right kernels the arguments are similar. The right slice hyperholomorphicity in $s$ of the left kernels is due to the fact that each kernel can be written as a left combination of intrinsic slice hyperholomorphic functions in $s$: $\qcs^{-m}$, $s\qcs^{-m}$, $s^2\qcs^{-m}$ and $s^3\qcs^{-m}$ for $1\leq m\leq 3$ (see Theorems from \ref{app1} to \ref{app7}).
	
	As the operators $\bigD$, $\overline\bigD$ and $\Delta$ can be commuted, we have:
	\begin{enumerate}
		\item $\Delta^{\ell+1}(S^{-1}_{\Delta^{1-\ell}\bigD}(s,x))=\Delta^{\ell+1}(\Delta^{1-\ell}\bigD (S^{-1}_L(s,x)))=\bigD \Delta^2(S^{-1}_L(s,x))=\bigD F^5_L(s,x)=0$, wich  implies $S^{-1}_{\Delta^{1-\ell}\bigD,L}(s,x)$ is $\ell+1$-harmonic in $x$ for $0\leq \ell\leq 1$;\\
		\item $\Delta\bigD(S^{-1}_{\Delta,L}(s,x))=\Delta\bigD(\Delta(S^{-1}_L(s,x)))=\bigD \Delta^2(S^{-1}_L(s,x))=\bigD F^5_L(s,x)=0$, wich   implies $S^{-1}_{\Delta,L}(s,x)$ is holomorphic Cliffordian of order $1$ in $x$;\\
		\item $\Delta\bigD^2(S^{-1}_{\overline\bigD,L}(s,x))=\Delta\bigD^2(\overline\bigD (S^{-1}_L(s,x)))=\bigD \Delta^2(S^{-1}_L(s,x))=\bigD F^5_L(s,x)=0$, wich   implies $S^{-1}_{\overline\bigD,L}(s,x)$ is polyanalytic Cliffordian of order $(1,2)$ in $x$;\\
		\item $\bigD^{3-\ell}(S^{-1}_{\Delta^\ell \overline\bigD^{2-l},L}(s,x))=\bigD^{3-\ell}(\Delta^\ell \overline\bigD^{2-l} (S^{-1}_L(s,x)))=\bigD \Delta^2(S^{-1}_L(s,x))=\bigD F^5_L(s,x)=0$, wich   implies $S^{-1}_{\Delta^\ell \overline\bigD^{2-l},L}(s,x)$ is polyanalytic of order $3-\ell$ in $x$ for $0\leq \ell\leq 1$;\\
		\item $\Delta\overline\bigD(S^{-1}_{\bigD^{2},L}(s,x))=\Delta\overline\bigD(\bigD^{2} (S^{-1}_L(s,x)))=\bigD \Delta^2(S^{-1}_L(s,x))=\bigD F^5_L(s,x)=0$, wich   implies $S^{-1}_{\bigD^{2},L}(s,x)$ is anti-holomorphic Cliffordian of order $1$ in $x$.
	\end{enumerate}
\end{proof}
\begin{remark}\label{fkernels}
We can write the formulas of left and right slice kernels in Theorem \ref{app1} up to Theorem \ref{app7} in terms of the $F_n$-kernels. By using  formula \eqref{fu1} and \eqref{fu2} we have
\begin{eqnarray*}
S^{-1}_{\mathcal D,L}(s,x)&=&-\frac 1{16}\left[F^L_5(s,x)s^3-(x+2x_0)F^L_5(s,x)s^2+(2x_0x+|x|^2)F^L_5(s,x)s-x|x|^2F^L_5(s,x)\right],\\
S^{-1}_{\mathcal D,R}(s,x)&=&-\frac 1{16}\left[s^3F^R_5(s,x)-s^2F^R_5(s,x)(x+2x_0)+sF^R_5(s,x)(2x_0x+|x|^2)-F^R_5(s,x)x|x|^2\right],
\end{eqnarray*}
\begin{eqnarray*}
S^{-1}_{\Delta,L}(s,x)&=&-\frac 1{8}\left[ F^L_5(s,x)s^2-2x_0F^L_5(s,x)s+|x|^2F^L_5(s,x) \right],\\
S^{-1}_{\Delta,R}(s,x)&=&-\frac 1{8}\left[ s^2F^R_5(s,x)-2sF^R_5(s,x)x_0+F^R_5(s,x)|x|^2 \right],\\
\end{eqnarray*}
\begin{eqnarray*}
S^{-1}_{\Delta\mathcal D,L}(s,x)&=&\frac {1}{4}\left[F^{L}_5(s,x)s-xF^L_5(s,x)\right],\\
S^{-1}_{\Delta\mathcal D,R}(s,x)&=&\frac {1}{4}\left[sF^{R}_5(s,x)-F^R_5(s,x)x\right],\\
\end{eqnarray*}

\begin{eqnarray*}
S^{-1}_{\overline{\mathcal D},L}(s,x)&=& \frac 1{32}\left[3F^L_5(s,x)s^3-(8x_0+x)F^L_5(s,x)s^2+(4x_0^2+2x_0x+3|x|^2)F^L_5(s,x)s\right.\\
&&\left.-(x|x|^2+2x_0|x|^2)F^L_5(s,x)\right],\\
S^{-1}_{\overline{\mathcal D},R}(s,x)&=& \frac 1{32}\left[3s^3F^R_5(s,x)-s^2F^R_5(s,x)(8x_0+x)+sF^R_5(s,x)(4x_0^2+2x_0x+3|x|^2)\right.\\
&&\left.-F^R_5(s,x)(x|x|^2+2x_0|x|^2)\right],\\
\end{eqnarray*}

\begin{eqnarray*}
S^{-1}_{\mathcal{D}^2,L}(s,x)&=&-\frac 1{8}\left[F^L_5(s,x)s^2-2x F^L_5(s,x)s+x^2F^L_5(s,x)\right],\\
S^{-1}_{\mathcal{D}^2,R}(s,x)&=&-\frac 1{8}\left[s^2F^R_5(s,x)-2 sF^R_5(s,x)x+F^R_5(s,x)x^2\right],\\
\end{eqnarray*}

\begin{eqnarray*}
S^{-1}_{\mathcal{\overline{D}}^{2},L}(s,x)&=&\frac{1}{2}[F_{5}^{L}(s,x)s^2-2x_0 F_{5}^{L}(s,x)s+ x_0^2 F_{5}^{L}(s,x)],\\
\nonumber
S^{-1}_{\mathcal{\overline{D}}^{2},R}(s,x)&=&\frac{1}{2}[s^2F_{5}^{R}(s,x)-2 sF_{5}^{R}(s,x)x_0+  F_{5}^{R}(s,x)x_0^2],\\
\nonumber
\end{eqnarray*}

\begin{eqnarray*}
S^{-1}_{\Delta\mathcal{\overline{D}},L}(s,x)&=&-F_{5}^L(s,x)s+x_0 F_{5}^L(s,x),\\
\nonumber
S^{-1}_{\Delta\mathcal{\overline{D}},R}(s,x)&=&-sF_{5}^R(s,x)+ F_{5}^R(s,x)x_0.
\end{eqnarray*}
The kernels $ S^{-1}_{\Delta^{\ell} \mathcal{\overline{D}}^{2},L}(s,x)$ and $S^{-1}_{\Delta \mathcal{\overline{D}},L}(s,x)$, and the right counterparts, are written in terms of their polyanalytic decomposition, see  Proposition \ref{polydeco}.
\end{remark}

\subsection{The integral representation for the fine structure functions spaces}

Now, we can give the integral representation for the functions of the fine structure spaces.
\begin{theorem}\label{int_rap_1}
	Let $W \subset \mathbb{R}^6$ be an open set. Let $U$ be a slice Cauchy domain such that $\bar{U} \subset W$. Then for $J \in \mathbb{S}$ and $ds_J=ds(-J)$ we have the following integral representation.
	\begin{itemize}
		\item
{\bf (Integral representation of $ \ell$-harmonic functions, $0 \leq \ell  \leq 1$)} If $f \in SH_L(W)$, the function $ \tilde{f}_\ell(x):=\Delta^{1-\ell} \mathcal{D}f(x)$ is $\ell+1$-harmonic for $0\leq \ell\leq 1$ and it admits the following integral representation
		$$ \tilde{f}_\ell(x)= \frac{1}{2\pi} \int_{\partial(U \cap \mathbb{C}_J)} S^{-1}_{\Delta^{1-\ell}\mathcal D,L}(s,x) ds_J f(s).$$
		If $f \in SH_R(W)$, the function $ \tilde{f}_\ell(x):=f(x)\Delta^{1-\ell} \mathcal{D}$ is $\ell+1$-harmonic for $0\leq \ell\leq 1$ and it admits the following integral representation
		$$ \tilde{f}_\ell(x)= \frac{1}{2\pi} \int_{\partial(U \cap \mathbb{C}_J)} f(s) ds_J  S^{-1}_{\Delta^{1-\ell}\mathcal D,R}(s,x).$$
		\item
 {\bf(Integral representation of holomorphic Cliffordian functions of order 1)}
  If $f \in SH_L(W)$, the function $ f^{\circ}(x):=\Delta f(x)$ is left holomorphic Cliffordian of order $1$ and it admits the following integral representation
		$$ f^{\circ}(x)=\frac{1}{2\pi} \int_{\partial(U \cap \mathbb{C}_J)} S^{-1}_{\Delta, L}(s,x)ds_J f(s).$$
		If $f \in SH_R(W)$, the function $ f^{\circ}(x):=\Delta f(x)$ is right holomorphic Cliffordian of order $1$ and it admits the following integral representation
		$$ f^{\circ}(x)=\frac{1}{2\pi} \int_{\partial(U \cap \mathbb{C}_J)} f(s) ds_J S^{-1}_{\Delta, R}(s,x) .$$
		\item
 {\bf(Integral representation of polyanalytic Cliffordian functions of order $(1,2)$)}
  If $f \in SH_L(W)$, the function $ \breve f^{\circ}(x):=\overline \bigD f(x)$ is left polyanalytic Cliffordian of order $(1,2)$ and it admits the following integral representation
		$$\breve f^{\circ}(x)=\frac{1}{2\pi} \int_{\partial(U \cap \mathbb{C}_J)} S^{-1}_{\overline \bigD, L}(s,x)ds_J f(s).$$
		If $f \in SH_R(W)$, the function $ \breve f^{\circ}(x):=\Delta f(x)$ is right polyanalytic Cliffordian of order $(1,2)$ and it admits the following integral representation
		$$\breve f^{\circ}(x)=\frac{1}{2\pi} \int_{\partial(U \cap \mathbb{C}_J)} f(s) ds_J S^{-1}_{\overline \bigD, R}(s,x) .$$
		\item
 {\bf(Integral representation of polyanalytic functions of order $3- \ell$, $0\leq \ell\leq 1$ )} If $f \in SH_L(W)$, the function $ \breve{f}_\ell(x):=\Delta^\ell \overline{\mathcal{D}}^{2-\ell}f(x)$ is left polyanalytic of order $3-\ell$ for $0\leq \ell\leq 1$ and it admits the following integral representation
		$$ \breve{f}_\ell(x)= \frac{1}{2\pi} \int_{\partial(U \cap \mathbb{C}_J)} S^{-1}_{\Delta^\ell \overline{\mathcal{D}}^{2-\ell},L}(s,x) ds_J f(s).$$
		If $f \in SH_R(W)$, the function $ \breve{f}_\ell(x):=f(x)\Delta^\ell \overline{\mathcal{D}}^{2-\ell}$ is right polyanalytic of order $3-\ell$ for $0\leq \ell\leq 1$ and it admits the following integral representation
		$$ \breve{f}_\ell(x)= \frac{1}{2\pi} \int_{\partial(U \cap \mathbb{C}_J)} f(s) ds_J  S^{-1}_{\Delta^\ell\overline{\mathcal D}^{2-\ell},R}(s,x).$$
		\item
 {\bf(Integral representation of anti-holomorphic Cliffordian functions of order $1$)}
  If $f \in SH_L(W)$, the function $ f_0 (x):=\bigD^2 f(x)$ is left anti-holomorphic Cliffordian of order $1$ and it admits the following integral representation
		$$f_{\circ}(x)=\frac{1}{2\pi} \int_{\partial(U \cap \mathbb{C}_J)} S^{-1}_{\bigD^2, L}(s,x)ds_J f(s).$$
		If $f \in SH_R(W)$, the function $ f_0(x):= f(x)\bigD^2$ is right anti-holomorphic Cliffordian of order $1$ and it admits the following integral representation
		$$ f_{\circ}(x)=\frac{1}{2\pi} \int_{\partial(U \cap \mathbb{C}_J)} f(s) ds_J S^{-1}_{\bigD^2, R}(s,x) .$$
	\end{itemize}
	Moreover, the integrals do not depend on $U$ nor on the imaginary unit $J \in \mathbb{S}$.
\end{theorem}
\begin{proof}
	We prove the integral representation for the $\ell+1$-harmonic functions starting from a left slice hyperholomorphic functions since the other cases can be proved with similar arguments. We start by using the Cauchy formula. By Theorem \ref{app1} and Theorem \ref{app3}, we have for $0\leq \ell\leq 1$
	$$ \tilde{f}_\ell(x)=\Delta^{1-\ell}\mathcal{D}f(x)= \frac{1}{2 \pi} \int_{\partial (U \cap \mathbb{C}_J)} \Delta^{1-\ell}\mathcal{D} S^{-1}_L(s,x) ds_I f(s)= \frac{1}{2\pi}\int_{\partial (U \cap \mathbb{C}_J)} S^{-1}_{\Delta^{1-\ell}\mathcal D,L}(s,x) ds_I f(s).$$
	By Proposition \ref{reg} the function $\tilde{f}_\ell(x)$ is $\ell+1$-harmonic.
\end{proof}

\section{Series expansion of the kernels of the fine structures spaces}\label{serieskernels}

In this section our aim is to write the kernel of the previous integral theorems in terms of convergent series of $x$ and $\bar{x}$. In order to do this we need to investigate the application of the operators $\mathcal{D}$, $ \Delta$, $\Delta \mathcal{D}$, $\mathcal{\overline{D}}$, $\mathcal{\overline{D}}^2$, $ \mathcal{D}^2$ and $ \Delta \overline{\mathcal{D}}$ to the monomial $x^m$, with $m \in \mathbb{N}$. We already know that:
\begin{lemma}{\cite[Lemma 1]{B} }\label{begher}
For $m \geq 1$ we have
\begin{equation}
\label{dirac}
\mathcal{D}(x^m)=(x^m) \mathcal{D}=-4\sum_{k=0}^{m-1}\overline x^{m-k-1}x^k=-4\sum_{k=1}^m x^{m-k}\overline x^{k-1}.
\end{equation}
\end{lemma}

We will use the following well-known equality
\begin{equation}
\label{laplacian_formula}
\Delta (xf(x))=x\Delta f(x)+2\mathcal{D}f(x),
\end{equation}
for any $x\in\mathbb R^6$ and for any $\mathcal{C}^2$ function $f$.
\begin{proposition}\label{laplacian}
Let $ x \in \mathbb{R}^6$ and $m\geq 2$. Then we have
\begin{equation}
\label{delta}
\Delta x^m=-8\sum_{k=1}^{m-1}(m-k)x^{m-k-1}\overline x^{k-1}=-8\sum_{k=1}^{m-1}k x^{k-1}\bar{x}^{m-k-1}.
\end{equation}

\end{proposition}
\begin{proof}
The proof is by induction on $m$. For $m=2$ and $x=x_0+\sum_{i=1}^5 e_i x_i=x_0+\underline x$ we have
$$\Delta x^2=\Delta (x_0^2+2x_0\underline x-|\underline x|^2)=-8.$$
Let us suppose that the statement is true for $m$, we want to prove it for $m+1$. Then, we have
\[
\begin{split}
\Delta	x^{m+1}	&=\Delta (x^{m} x)=2 \mathcal{D}(x^m)+x\Delta x^m\\
&=-8\sum_{k=1}^m x^{m-k}\overline x^{k-1}-8\sum_{k=1}^{m-1}(m-k)x^{m-k}\overline x^{k-1}\\
&=-8\sum_{k=1}^m x^{m-k}\overline x^{k-1}-8\sum_{k=1}^{m}(m-k)x^{m-k}\overline x^{k-1}\\
&=-8\sum_{k=1}^m (m-k+1)x^{m-k}\overline x^{k-1},
\end{split}
\]
where the first equality is an application of formula \eqref{laplacian_formula} and the second equality is consequence of the inductive hypothesis and formula \eqref{dirac}. The second equality follows by rearranging the indexes.
\end{proof}

\begin{proposition}
\label{double_dirac}
Let $x\in\mathbb R^6$, for $m\geq 2$ we have
$$ \mathcal{D}^2(x^m)=(x^m) \mathcal{D}^2=-8\sum_{k=1}^{m-1}k x^{m-k-1}\overline{x}^{k-1}.$$
\end{proposition}
\begin{proof}
We show the result by induction on $m$. For $m=2$, we have by formula \eqref{dirac} that
$$ \mathcal{D}^2(x^2)=\mathcal{D}^2(x_0^2+2x_0\underline x-|\underline x|^2)=-8\mathcal{D}(x_0)=-8.$$
We suppose the statement is true for $m$ and we prove it for $m+1$. First we observe that
$$ x^{m+1}=x^m(\overline x+x)-x^{m-1}|x|^2 .$$
Thus, by the Leibniz formula for the Dirac operator and the fact that $ \mathcal{D} |x|^2=2x$ we get
\[
\begin{split}
\mathcal{D}(x^{m+1})& =\mathcal{D}(x^m(\overline x+x))-\mathcal{D}(x^{m-1}|x|^2)\\
&= \mathcal{D}(x^m)(\overline x+x) +2x^m-\mathcal{D}(x^{m-1})|x|^2-2x^m\\
&= \mathcal{D}(x^m)(\overline x+x)-\mathcal{D}(x^{m-1})|x|^2.
\end{split}
\]
By using another time the Leibniz formula and the inductive hypothesis we get
\[
\begin{split}
&\mathcal{D}^2(x^{m+1})=\mathcal{D}\left (\mathcal{D}(x^m)(\overline x+x)-\mathcal{D}(x^{m-1})|x|^2\right)\\
&=(\mathcal{D}^2x^m)(x+\overline x)+2\mathcal{D}x^m-\mathcal{D}^2(x^{m-1})|x|^2-2x\mathcal{D}(x^{m-1})\\
&=-8\sum_{k=1}^{m-1}k x^{m-k}\overline x^{k-1}-8\sum_{k=1}^{m-1}k x^{m-k-1}\overline x^{k}-8\sum_{k=1}^{m} x^{m-k}\overline x^{k-1}+8\sum_{k=1}^{m-2}k x^{m-k-1}\overline x^{k}+8\sum_{k=1}^{m-1}x^{m-k}\overline x^{k-1}\\
&=-8\sum_{k=1}^{m-1}k x^{m-k}\overline x^{k-1}-8(m-1) \bar{x}^{n-1}-8 \bar{x}^{m-1}\\
&=-8\sum_{k=1}^{m}k x^{m-k}\overline x^{k-1}.
\end{split}
\]
Finally since $ \mathcal{D} (x^m)= (x^m) \mathcal{D}$ we get
$$ (x^m) \mathcal{D}^2=\left( (x^m) \mathcal{D}\right) \mathcal{D}= \mathcal{D} (x^m) \mathcal{D}=\mathcal{D}^2 (x^m).$$
\end{proof}

Now, we need some preliminaries results to get a formula for $ \Delta \mathcal{D} x^m$.
\begin{lemma}
\label{conju}
Let $f: \mathbb{R}^{n+1} \to \mathbb{R}_n$ then
$$ \overline{\Delta f(x)}= \Delta \overline{f}(x).$$
\end{lemma}
\begin{proof}
We know that we can write $f(x)= \sum_{A \subset \{1,...,n\}} e_A f_A$, thus we have
$$ \overline{\Delta f(x)}= \sum_{A \subset \{1,...,n\}} e_A \Delta f_A(x_A)= \sum_{A \subset \{1,...,n\}} \overline{e}_A \Delta f_A(x_A)= \Delta \left(\sum_{A \subset \{1,...,n\}} \overline{e}_A f_A(x_A)\right)=\Delta \overline{f}(x).$$
\end{proof}

\begin{corollary}
\label{plus1}
Let $m \geq 2$. Then for $x \in \mathbb{R}^6$ we have
\begin{equation}
\label{plus3}
\overline{\Delta x^m}= \Delta \overline{x}^m.
\end{equation}
\end{corollary}
\begin{proof}
If we consider $n=5$ and $f(x)=x^m$ in Lemma \ref{conju}, we get the statement.
\end{proof}
\begin{corollary}
\label{conj}
Let $m \geq 2$. Then for any $x \in \mathbb{R}^6$ we have
$$ \Delta \bar{x}^m= -8 \sum_{k=1}^{m-1} k x^{m-k-1} \bar{x}^{k-1}.$$
\end{corollary}
\begin{proof}
It follows by Corollary \ref{plus1} and Proposition \ref{laplacian_formula}.
\end{proof}
\begin{lemma}
\label{somma}
Let $m\geq 3$, for any $x\in\mathbb R^6$ we have
$$\sum_{k=1}^m \Delta (x^{m-k}\overline x^{k-1})=-4\sum_{k=1}^{m-2}(m-k-1)kx^{m-k-2}\overline x^{k-1}.$$
\end{lemma}
\begin{proof}
We shall prove this formula by induction on $m$. For $m=3$ we have
$$
\Delta(x^2+x\overline x+\overline x^2)=\Delta (3x_0^2-|\underline x|^2)=-4.
$$
Let us suppose the statement is true for $m$, we want to prove it for $m+1$. Now, formula \eqref{laplacian_formula} and Corollary \ref{conj} imply that
\[
\begin{split}
\sum_{k=1}^{m+1}\Delta(x^{m+1-k}\overline x^{k-1})&=\sum_{k=1}^m\Delta(x^{m+1-k}\overline x^{k-1})+\Delta \overline x^m\\
&= 2\sum_{k=1}^m \mathcal{D}(x^{m-k}\overline x^{k-1})+x\sum_{k=1}^m \Delta(x^{m-k}\overline x^{k-1})-8\sum_{k=1}^{m-1}kx^{m-k-1}\overline x^{k-1}.
\end{split}
\]
Finally, by formula \eqref{dirac}, the inductive hypothesis and Lemma \ref{double_dirac} we get
\begin{eqnarray*}
\sum_{k=1}^{m+1}\Delta(x^{m+1-k}\overline x^{k-1})&=& -\frac 12 \mathcal{D}^2x^m+x\sum_{k=1}^m\Delta (x^{m-k}\overline x^{k-1})-8\sum_{k=1}^{m-1}kx^{m-k-1}\overline x^{k-1}\\
&=& -\frac 12 \mathcal{D}^2x^m-4\sum_{k=1}^{m-2} (m-k-1)k x^{m-k-1}\overline x^{k-1}-8\sum_{k=1}^{m-1}kx^{m-k-1}\overline x^{k-1}\\
&=& 4 \sum_{k=1}^{m-1}kx^{m-k-1}\overline x^{k-1}-4\sum_{k=1}^{m-2} (m-k-1)k x^{m-k-1}\overline x^{k-1}-8\sum_{k=1}^{m-1}kx^{m-k-1}\overline x^{k-1}\\
&=& -4\sum_{k=1}^{m-1}(m-k)k x^{m-k-1}\overline x^{k-1}.
\end{eqnarray*}
\end{proof}

\begin{proposition}\label{dirac_laplacian}
Let $m\geq 3$, for any $x\in\mathbb R^6$ we have
$$ \Delta \mathcal{D}x^m=16\sum_{k=1}^{m-2}(m-k-1)k x^{m-k-2} \overline x^{k-1}.$$
\end{proposition}
\begin{proof}
It follows by formula \eqref{dirac} and Lemma \ref{somma}.
\end{proof}
Let us recall the following fact
\begin{equation}
\label{l1}
\mathcal{D}_{\underline{x}}(\underline{x}^m)= \begin{cases}
-m \underline{x}^{m-1} \qquad m \quad \hbox{even},\\
-(m+4) \underline{x}^{m-1} \qquad m \quad \hbox{odd},
\end{cases}
\end{equation}
where $ \mathcal{D}_{\underline{x}}= \sum_{j=1}^5 e_j\frac{\partial}{\partial x_j}$.
\begin{proposition}
\label{l6}
Let $ m \geq 1$. For any $x \in \mathbb{R}^6$, if $ \underline{x} \neq 0$ we have
\begin{equation}
\label{l9}
\mathcal{\overline{D}}x^m=2 \left[mx^{m-1}+2 \sum_{k=1}^m x^{m-k} \bar{x}^{k-1}\right].
\end{equation}
On the other hand if $ \underline{x}=0$ we have
$$ \mathcal{\overline{D}}x^m=6mx^{m-1}.$$
Moreover,
$$ \mathcal{\overline{D}}(x^m)=(x^m) \mathcal{\overline{D}}.$$
\end{proposition}
\begin{proof}
We will perform a direct computations. By the Binomial theorem and formula \eqref{l1} we get
\begin{eqnarray*}
\mathcal{\overline{D}}(x^m)&=& \left(\frac{\partial}{\partial x_0}- \mathcal{D}_{\underline{x}}\right) (x_0+ \underline{x})^m=\left(\frac{\partial}{\partial x_0}- \mathcal{D}_{\underline{x}}\right) \left( \sum_{k=0}^m \binom{m}{k} x_{0}^{m-k} \underline{x}^k\right)\\
&=& \sum_{k=0}^{m-1} \binom{m}{k} (m-k) x_{0}^{m-k-1} \underline{x}^{k}- \sum_{k=0}^m \binom{m}{k}x_{0}^{m-k} \mathcal{D}_{\underline{x}} (\underline{x}^k)\\
&=& m \sum_{k=0}^{m-1} \frac{(m-1)!}{k!(m-k-1)!} x_{0}^{m-k-1} \underline{x}^{k}+m \sum_{k=1}^{m} \frac{(m-1)!}{(k-1)!(m-k)!} x_{0}^{m-k} \underline{x}^{k-1}+\\
&& + 4\sum_{k=1, k \, \, \hbox{odd}}^m \binom{m}{k} x_{0}^{m-k} \underline{x}^{k-1}.
\end{eqnarray*}
By rearranging the indices of the sum and by using another time the binomial theorem we get
\begin{equation}
\label{l2}
\mathcal{\overline{D}}(x^m)=2mx^{m-1}+ 4\sum_{k=1, k \, \, \hbox{odd}}^m \binom{m}{k} x_{0}^{m-k} \underline{x}^{k-1}.
\end{equation}
If $ \underline{x} =0$ we get
$$ \overline{D}x^m=6mx^{m-1}.$$
On the other side, if $ \underline{x} \neq 0$ by the Binomial theorem we have
\begin{eqnarray*}
2 \sum_{k=1, k \hbox{ odd}}^m \binom{m}{k} x_{0}^{m-k} \underline{x}^{k-1}&=& \sum_{k=1}^m \binom{m}{k} x_{0}^{m-k} \underline{x}^{k-1}+ \sum_{k=1}^m \binom{m}{k} x_{0}^{m-k} (-\underline{x})^{k-1}\\
&=& (\underline{x})^{-1} \left[ \sum_{k=1}^m \binom{m}{k} x_{0}^{m-k} \underline{x}^k - \sum_{k=1}^{m} \binom{m}{k} x_{0}^{m-k} (-\underline{x})^{k}\right]\\
&=& (\underline{x})^{-1}(x^m- \bar{x}^m).
\end{eqnarray*}
Now, since $x^{m}- \bar{x}^m=2 \underline{x} \sum_{k=1}^m x^{m-k} \bar{x}^{k-1}$ we get
\begin{equation}
\label{l3}
2 \sum_{k=1}^m \binom{m}{k} x_{0}^{m-k} \underline{x}^{k-1}= 2 \sum_{k=1}^m x^{m-k} \bar{x}^{k-1}.
\end{equation}
By putting formula \eqref{l3} in \eqref{l2} we obtain formula \eqref{l9}.
\\ Finally, since $ \mathcal{D}_{\underline{x}}(\underline{x}^m)=(\underline{x}^m)\mathcal{D}_{\underline{x}}$, we can repeat the previous computations, thus we get
$$ \mathcal{\overline{D}}(x^m)=(x^m) \mathcal{\overline{D}}.$$
\end{proof}
To compute $ \mathcal{\overline{D}}^2x^m$ we need the following result.
\begin{lemma}
\label{l7}
Let $m \geq 2$. For any $x \in \mathbb{R}^6$ we have
$$ \mathcal{D}_{\underline{x}} \left( \sum_{k=1}^m x^{m-k} \bar{x}^{k-1}\right)= \sum_{k=1}^{m-1}(2k-m) x^{m-k-1} \bar{x}^{k-1}.$$
\end{lemma}
\begin{proof}
By Proposition \ref{double_dirac}  and formula \eqref{dirac} we have
\begin{eqnarray}
\nonumber
-8 \sum_{k=1}^{m-1} k x^{m-k-1} \bar{x}^{k-1}&=& \mathcal{D}(\mathcal{D} x^m)=-4 \mathcal{D} \left(\sum_{k=1}^m x^{m-k} \bar{x}^{k-1}\right)\\
&=& -4 \left[ \frac{\partial}{\partial x_0}\left(\sum_{k=1}^m x^{m-k} \bar{x}^{k-1}\right)+ \mathcal{D}_{\underline{x}}\left(\sum_{k=1}^m x^{m-k} \bar{x}^{k-1}\right) \right]. \label{l4}
\end{eqnarray}
Now, since $\sum_{k=1}^m x^{m-k} \bar{x}^{k-1}= (2 \underline{x})^{-1}(x^m- \bar{x}^n)$ we get
\begin{eqnarray}
\nonumber
\frac{\partial}{\partial x_0}\left(\sum_{k=1}^m x^{m-k} \bar{x}^{k-1}\right) &=& \frac{\partial}{\partial x_0} \left[(2 \underline{x})^{-1}(x^m- \bar{x}^m)\right]\\
\nonumber
&=& (2 \underline{x})^{-1} m(x^{m-1} - \overline{x}^{m-1})\\
\label{l5}
&=& m \sum_{k=1}^{m-1} x^{m-1-k} \bar{x}^{k-1}.
\end{eqnarray}
Therefore by formula \eqref{l4} and formula \eqref{l5} we get
\begin{eqnarray*}
\mathcal{D}_{\underline{x}} \left( \sum_{k=1}^m x^{m-k} \bar{x}^{k-1}\right)&=& 2 \sum_{k=1}^{m-1} k x^{m-k-1} \bar{x}^{k-1}-\frac{\partial}{\partial x_0}\left(\sum_{k=1}^m x^{m-k} \bar{x}^{k-1}\right) \\
&=& 2 \sum_{k=1}^{m-1} k x^{m-k-1} \bar{x}^{k-1}-m \sum_{k=1}^{m-1} x^{m-1-k} \bar{x}^{k-1}\\
&=& \sum_{k=1}^{m-1} (2k-m) x^{m-k-1} \bar{x}^{k-1}.
\end{eqnarray*}
\end{proof}
\begin{proposition}
\label{antidouble}
Let $m \geq 2$. Then for any $x \in \mathbb{R}^6$ we have
\begin{equation}
\label{poly1}
\mathcal{\overline{D}}^2 (x^m)= (x^m)\mathcal{\overline{D}}^2=4 \left[m(m-1) x^{m-2}+2 \sum_{k=1}^{m-1} (2m-k)x^{m-k-1} \bar{x}^{k-1}\right].
\end{equation}
\end{proposition}
\begin{proof}
By applying two times Proposition \ref{l6} and the fact that $ \mathcal{\overline{D}}=\frac{\partial}{\partial x_0}- \mathcal{D}_{\underline{x}}$ we have
\begin{eqnarray*}
\mathcal{\overline{D}}^2 (x^m)&=& \mathcal{\overline{D}}(\mathcal{\overline{D}} x^m)\\
&=& 2m \mathcal{\overline{D}}(x^{m-1})+4\mathcal{\overline{D}} \left( \sum_{k=1}^m x^{m-k} \bar{x}^{k-1} \right) \\
&=& 2m \left( 2(m-1) x^{m-2} +4 \sum_{k=1}^{m-1} x^{m-1-k} \bar{x}^{k-1}\right)+\\
&& +4 \left[ \frac{\partial}{\partial x_0}\left( \sum_{k=1}^m x^{m-k} \bar{x}^{k-1} \right) - \mathcal{D}_{\underline{x}}\left( \sum_{k=1}^m x^{m-k} \bar{x}^{k-1} \right) \right].
\end{eqnarray*}
By formula \eqref{l5} and Lemma \ref{l7} we get
\begin{eqnarray*}
\mathcal{\overline{D}}^2 x^m&=&4 \biggl[m(m-1) x^{m-2}+2m \sum_{k=1}^{m-1} x^{m-1-k} \bar{x}^{k-1}+\\
&& +m \sum_{k=1}^{m-1} x^{m-1-k} \bar{x}^{k-1}- \sum_{k=1}^{m-1}(2k-m) x^{m-k-1} \bar{x}^{k-1} \biggl]\\
&=& 4 \left[m(m-1) x^{m-2}+2 \sum_{k=1}^{m-1} (2m-k)x^{m-k-1} \bar{x}^{k-1}\right].
\end{eqnarray*}
Finally, since $ \mathcal{\overline{D}} (x^m)= (x^m)\mathcal{\overline{D}}$ we get
$$ \mathcal{\overline{D}}^2 (x^m)= (x^m)\mathcal{\overline{D}}^2.$$
\end{proof}

\begin{proposition}
\label{antidirac}
Let $ m \geq 3$. For any $x \in \mathbb{R}^6$ we have
\begin{equation}
\label{poly0}
\Delta \mathcal{\overline{D}}(x^m)=(x^m)\Delta \mathcal{\overline{D}}=-16 \sum_{k=1}^{m-2} (m-k-1)(m+k)x^{m-k-2} \bar{x}^{k-1}.
\end{equation}
\end{proposition}
\begin{proof}
By applying the operator $\Delta$ to formula \eqref{l7} we get
$$ \Delta \mathcal{\overline{D}}x^m=2m \Delta(x^{m-1})+4 \sum_{k=1}^m \Delta(x^{m-k} \bar{x}^{k-1}).$$
Therefore, by Proposition \ref{laplacian} and Lemma \ref{somma} we obtain
\begin{eqnarray*}
\Delta \mathcal{\overline{D}}(x^m)&=& -16m \sum_{k=1}^{m-2} (m-1-k)x^{m-2-k} \bar{x}^{k-1}+\\
&&-16 \sum_{k=1}^{m-2}(m-k-1)k x^{m-k-2} \bar{x}^{k-1}\\
&=&-16 \sum_{k=1}^{m-2}(m-k-1)(m+k)x^{m-2-k} \bar{x}^{k-1}.
\end{eqnarray*}
Finally, since the laplacian is a real operator and $ \mathcal{\overline{D}}(x^m)= (x^m)\mathcal{\overline{D}}$ we get
$$ \Delta \mathcal{\overline{D}}(x^m)=(x^m)\Delta \mathcal{\overline{D}}.$$
\end{proof}
\begin{remark}
The polynomials found in \eqref{poly1} and \eqref{poly0} are polyanalytic of order 2 and 3, respectively.  However, these polynomials do not coincide to which ones found in \cite{DMD3}, in which the authors obtained polyanalytic polynomials by applying the Fueter-polyanalytic maps, see \cite{ADS}.
\end{remark}

\begin{definition}[left and right kernel series]
Let $s,\, x\in\mathbb R^6$, then we define
\begin{itemize}
\item the left $ \mathcal{D}$-kernel series as
\begin{equation}
\label{q1res_ser}
-4\sum_{m=1}^{\infty}\sum_{k=1}^m x^{m-k}\overline x^{k-1} s^{-1-m},
\end{equation}
and the right $ \mathcal{D}$-kernel series as
$$ -4\sum_{m=1}^{\infty}\sum_{k=1}^m  s^{-1-m} x^{m-k}\overline x^{k-1},$$
\item the left $ \Delta$-kernel series as
\begin{equation}
\label{f1res_ser}
-8\sum_{m=2}^{\infty}\sum_{k=1}^{m-1} (m-k)x^{m-k-1}\overline x^{k-1}s^{-1-m},
\end{equation}
the right $ \Delta$-kernel series as
$$-8\sum_{m=2}^{\infty}\sum_{k=2}^{m-1} (m-k)s^{-1-m}x^{m-k-1}\overline x^{k-1},$$
\item the left $ \Delta \mathcal{D}$-kernel series as
\begin{equation}
\label{q2res_ser}
16\sum_{m=3}^\infty\sum_{k=1}^{m-2}(m-k-1)kx^{m-k-2}\overline x^{k-1} s^{-1-m},
\end{equation}
the right $ \Delta \mathcal{D}$-kernel series as
$$ 16\sum_{m=3}^\infty\sum_{k=3}^{m-2}(m-k-1)k s^{-1-m} x^{m-k-2}\overline x^{k-1},$$
\item the left $ \mathcal{D}^2$-kernel series as
\begin{equation}
-8 \sum_{m=2}^\infty \sum_{k=1}^{m-1} k x^{m-k-1} \bar{x}^{k-1}s^{-1-m},
\end{equation}
the right $ \mathcal{D}^2$-kernel series as
$$-8 \sum_{m=2}^\infty \sum_{k=1}^{m-1} k s^{-1-m} x^{m-k-1} \bar{x}^{k-1},$$
\item the left $ \mathcal{\overline{D}}$-kernel series as
\begin{equation}
\label{abs1}
2  \left[\sum_{m=1}^\infty \left(mx^{m-1}  +2  \sum_{k=1}^{m} x^{m-k} \bar{x}^{k-1} \right) \right] s^{-1-m},
\end{equation}
the right $ \mathcal{\overline{D}}$-kernel series as
$$2  \left[\sum_{m=1}^\infty s^{-1-m} \left(mx^{m-1}  +2  \sum_{k=1}^{m} x^{m-k} \bar{x}^{k-1} \right) \right] ,$$
\item the left $ \mathcal{\overline{D}}^2$-kernel series as
\begin{equation}
\label{abs2}
4 \left[\sum_{m=2}^\infty \left(m(m-1)x^{m-2}+2  \sum_{k=1}^{m-1} (2m-k)x^{m-k-1}\bar{x}^{k-1}\right) s^{-1-m} \right],
\end{equation}
the right $ \mathcal{\overline{D}}^2$-kernel series as
$$4 \left[\sum_{m=2}^\infty s^{-1-m} \left(m(m-1)x^{m-2}+2  \sum_{k=1}^{m-1} (2m-k)x^{m-k-1}\bar{x}^{k-1}\right)  \right],$$
\item the left $ \Delta \mathcal{\overline{D}}$-kernel series as
\begin{equation}
-16 \sum_{m=3}^\infty \sum_{k=1}^{m-2} (m-k-1)(m+k)x^{m-k-2} \bar{x}^{k-1}s^{-1-m}.
\end{equation}
the right $ \Delta \mathcal{\overline{D}}$-kernel series as
$$-16 \sum_{m=3}^\infty \sum_{k=1}^{m-2} (m-k-1)(m+k)s^{-1-m}x^{m-k-2} \bar{x}^{k-1}.$$
\end{itemize}
\end{definition}

\begin{remark}
The left and the right $ \mathcal{D}$-kernel series are equal, where they converge, as well as the left and the right $ \Delta \mathcal{D}$-kernel series.
\end{remark}

We collect some technical lemmas that we have used in the proofs of some of the following results.
\begin{lemma}
\label{sum}
For $ m \geq 3$ we have
$$ \sum_{k=1}^{m-2} (m-k-1)k= \frac{m(m-1)(m-2)}{6}.$$
\end{lemma}
\begin{proof}
We know that $ \sum_{k=1}^{m-2}k= \frac{(m-1)(m-2)}{2}$ and $ \sum_{k=1}^{m-2} k^2= \frac{(m-2)(m-1)(2m-3)}{6}$. Thus we have
\begin{eqnarray*}
\sum_{k=1}^{m-2} (m-k-1)k &=& (m-1)\sum_{k=1}^{m-2}k-\sum_{k=1}^{m-2} k^2\\
&=& \frac{(m-1)^2(m-2)}{2}- \frac{(m-2)(m-1)(2m-3)}{6}\\
&=& \frac{(m-1)(m-2)}{2} \left((m-1)- \frac{(2m-3)}{3}\right)\\
&=& \frac{m(m-1)(m-2)}{6}.
\end{eqnarray*}
\end{proof}
\begin{lemma}
\label{sum2}
For $ m \geq 3$ we have
$$ \sum_{k=1}^{m-2}(m-k-1)(m+k)= \frac{2m(m-1)(m-2)}{3}.$$
\end{lemma}
\begin{proof}
Since $ \sum_{k=1}^{m-2}k= \frac{(m-1)(m-2)}{2}$ we get
\begin{eqnarray}
\nonumber
m \sum_{k=1}^{m-2} (m-1-k)&=&m^2(m-2)-m(m-2)- \frac{m(m-1)(m-2)}{2}\\
\label{sum3}
&=& \frac{m(m-1)(m-2)}{2}.
\end{eqnarray}
Finally, by formula \eqref{sum3} and Lemma \ref{sum} we get
\begin{eqnarray*}
\sum_{k=1}^{m-2}(m-k-1)(m+k)&=& \sum_{k=1}^{m-2}(m-k-1)k +m\sum_{k=1}^{m-2}(m-k-1)\\
&=& \frac{m(m-1)(m-2)}{6}+\frac{m(m-1)(m-2)}{2}\\
&=& \frac{2m(m-1)(m-2)}{3}.
\end{eqnarray*}
\end{proof}

\begin{proposition}
\label{conve}
For $s$,$x\in\mathbb R^6$ with $|x|<|s|$, all the left kernel series are convergent.
\end{proposition}
\begin{proof}
In order to show the convergence of the series it is enough to show the convergence of the series of the moduli.	
\begin{itemize}
\item In order to prove that the left $ \mathcal{D}$-kernel series is convergent it is sufficient to show that the following series is convergent
\begin{equation}
\label{abs3}
\sum_{m=1}^{\infty}m|x|^{m-1} |s|^{-1-m}.
\end{equation}
This series converges by the ratio test. Indeed, since  $|x|<|s|$, we have
$$\lim_{m\to\infty}\frac{(m+1)|x|^m|s|^{-2-m}}{m|x|^{m-1}|s|^{-1-m}}=\lim_{m\to \infty}\frac{m+1}{m}|x||s|^{-1}<1.$$
\item To prove the convergence of the $\Delta$-kernel series, we have to show that the following series convergences
\begin{equation}
\label{abs4}
\sum_{m=2}^\infty \left(\sum_{k=1}^{m-1} (m-k) \right)|x|^{m-2}|s|^{-1-m}.
\end{equation}
Since $\sum_{k=1}^{m-1}( m-k)=\frac{m(m-1)}{2}$, we have
\begin{equation}
\label{limit}
\lim_{m\to\infty}\frac{(m+1)m|x|^{m-1}|s|^{-2-m}}{m(m-1)|x|^{m-2}|s|^{-1-m}}=|x||s|^{-1}<1.
\end{equation}
Therefore, by the ratio test the series is convergent.
\item
In order to prove the convergence of the $\Delta \mathcal{D}$-kernel series, we have to show the convergence of the following series
\begin{equation}
\label{abs5}
\sum_{m=1}^{\infty}\left(\sum_{k=1}^{m-2}(m-k-1)k\right) |x|^{m-3}|s|^{-1-m}.
\end{equation}
Now, since $\sum_{k=1}^{m-2}(m-k-1)k=\frac{m(m-1)(m-2)}{6}$ (see Lemma \ref{sum}), by applying the ratio test we get
$$
\lim_{m\to\infty}\frac{m(m+1)(m-1)|x|^{m-2}|s|^{-2-m}}{m(m-1)(m-2) |x|^{m-3}|s|^{-1-m}}=|x||s|^{-1}<1.
$$
Therefore the series is convergent.
\item To show the convergence of the $ \mathcal{D}^2$- kernel series we need to show the convergence of the following series
$$ \sum_{m=2}^{\infty} \left( \sum_{k=1}^{m-1} k\right) |x|^{m-2} |s|^{-1-m}.$$
Since $ \sum_{k=1}^{m-1} k= \frac{(m-1)m}{2}$, by applying the ratio test we re obtain the same limit of \eqref{limit}, and so the series is convergent
\item In order to prove the convergence of the $ \mathcal{\overline{D}}$-kernel series we put the modulus to the series in \eqref{abs1} and after some manipulations we get that
$$  \left| \sum_{m=1}^\infty \left(mx^{m-1}  +2  \sum_{k=1}^{m} x^{m-k} \bar{x}^{k-1}   \right) \right| |s|^{-1-m} \leq \sum_{m=1}^\infty 3m |x|^{m-1} |s|^{-1-m}.$$
The convergence of the series $\sum_{m=1}^\infty 3m |x|^{m-1} |s|^{-1-m}$ follows by similar arguments for the convregence of the series in \eqref{abs3}.
\item As done in the previous point, we insert the modulus in the series \eqref{abs2} and after some computations we get
\begin{eqnarray*}
&& \left|\sum_{m=2}^\infty \left(m(m-1)x^{m-2}+2  \sum_{k=1}^{m-1} (2m-k)x^{m-k-1}\bar{x}^{k-1}\right)\right| |s|^{-1-m} \\
&& \leq \sum_{m=2}^\infty  \left(m(m-1)+ \sum_{k=1}^{m-1} (2m-k) \right)|x|^{m-2}|s|^{-1-m}\\
&&=  \sum_{m=2}^\infty \frac{5m (m-1)}{2} |x|^{m-2}|s|^{-1-m}.
\end{eqnarray*}
The convergence of the series $\sum_{m=2}^\infty \frac{5m (m-1)}{2} |x|^{m-2}|s|^{-1-m}$ follows similarly as the series in \eqref{abs4}.
\item Finally, in order to show the convergence of the left $ \Delta \mathcal{\overline{D}}$ it is enough to show the convergence of the following series
$$ \sum_{m=3}^\infty \left( \sum_{k=1}^{m-2} (m-k-1)(m+k) \right) |x|^{m-3} |s|^{-1-m}.$$
Since $ \sum_{k=1}^{m-2} (m-k-1)(m+k)= \frac{2m(m-1)(m-2)}{3}$ (see Lemma \ref{sum2}), the convergence follows similarly as done for the series in \eqref{abs5}.
\end{itemize}
\end{proof}
\begin{remark}
Similarly, all the right kernel series are convergent.
\end{remark}
Now, we can expand in series the left and the right kernels, computed in the previous section.
\begin{lemma}
\label{lk}
For $s$, $x\in\mathbb R^6$ such that $|x|<|s|$ we can expand
\begin{itemize}
\item the left slice $ \mathcal{D}$-kernel as
$$ S^{-1}_{\mathcal{D},L}(s,x)= -4\sum_{m=1}^{\infty}\sum_{k=1}^m x^{m-k}\overline x^{k-1} s^{-1-m},$$
and the right slice $ \mathcal{D}$-kernel as
$$ S^{-1}_{\mathcal{D},R}(s,x)= -4\sum_{m=1}^{\infty}\sum_{k=1}^m s^{-1-m}x^{m-k}\overline x^{k-1},$$
\item the left slice $ \Delta$-kernel as
$$S^{-1}_{\Delta,L}(s,x)=-8\sum_{m=2}^\infty\sum_{k=1}^{m-1} (m-k)x^{m-k-1}\overline x^{k-1}s^{-1-m},$$
and the right slice $ \Delta$-kernel as
$$S^{-1}_{\Delta,R}(s,x)=-8\sum_{m=2}^\infty\sum_{k=1}^{m-1} (m-k)s^{-1-m}x^{m-k-1}\overline x^{k-1},$$
\item the left slice $ \Delta \mathcal{D}$-kernel as
$$S^{-1}_{\Delta \mathcal{D},L}(s,x)=16\sum_{m=3}^{\infty}\sum_{k=1}^{m-2}(m-k-1)k x^{m-k-2}\overline x^{k-1}s^{-1-m},$$
and the right slice $ \Delta \mathcal{D}$-kernel as
$$S^{-1}_{\Delta \mathcal{D},R}(s,x)=16 \sum_{m=3}^{\infty}\sum_{k=1}^{m-2}(m-k-1)k s^{-1-m}x^{m-k-2}\overline x^{k-1},$$
\item we can expand the left slice $ \mathcal{D}^2$-kernel as
$$S^{-1}_{\mathcal{D}^2,L}(s,x)=-8 \sum_{m=2}^\infty \sum_{k=1}^{m-1} k x^{m-k-1} \bar{x}^{k-1}s^{-1-m},$$
and the right slice $ \mathcal{D}^2$-kernel as
$$S^{-1}_{\mathcal{D}^2,R}(s,x)=-8 \sum_{m=2}^\infty \sum_{k=1}^{m-1} k s^{-1-m} x^{m-k-1} \bar{x}^{k-1},$$
\item the left slice $ \mathcal{\overline{D}}$-kernel as
$$S^{-1}_{\mathcal{\overline{D}},L}(s,x)=2  \left[\sum_{m=1}^\infty \left(mx^{m-1}  +2  \sum_{k=1}^{m} x^{m-k} \bar{x}^{k-1} \right) \right] s^{-1-m},$$
and the right slice $ \mathcal{\overline{D}}$-kernel as
$$S^{-1}_{\mathcal{\overline{D}},R}(s,x)=2  \left[\sum_{m=1}^\infty s^{-1-m} \left(mx^{m-1}  +2  \sum_{k=1}^{m} x^{m-k} \bar{x}^{k-1} \right) \right] ,$$
\item the left slice $ \mathcal{D}^2$-kernel as
$$S^{-1}_{\mathcal{\overline{D}}^2,L}(s,x)=4 \left[\sum_{m=2}^\infty \left(m(m-1)x^{m-2}+2  \sum_{k=1}^{m-1} (2m-k)x^{m-k-1}\bar{x}^{k-1}\right) s^{-1-m} \right],$$
and the right slice $ \mathcal{\overline{D}}^2$-kernel as
$$S^{-1}_{\mathcal{\overline{D}}^2,R}(s,x)=4 \left[\sum_{m=2}^\infty s^{-1-m} \left(m(m-1)x^{m-2}+2  \sum_{k=1}^{m-1} (2m-k)x^{m-k-1}\bar{x}^{k-1}\right)  \right],$$
\item the left slice $ \Delta \mathcal{\overline{D}}$-kernel as
$$S^{-1}_{\Delta \mathcal{\overline{D}},L}(s,x)=-16 \sum_{m=3}^\infty \sum_{k=1}^{m-2} (m-k-1)(m+k)x^{m-k-2} \bar{x}^{k-1}s^{-1-m},$$
and the right slice $ \Delta \mathcal{\overline{D}}$-kernel as
$$S^{-1}_{\Delta \mathcal{\overline{D}},R}(s,x)= -16 \sum_{m=3}^\infty \sum_{k=1}^{m-2} (m-k-1)(m+k)s^{-1-m}x^{m-k-2} \bar{x}^{k-1}.$$
\end{itemize}
\end{lemma}
\begin{proof}
We know that we can expand the left Cauchy kernel in the following way
\begin{equation}
\label{expa}
S^{-1}_L(s,x)=\sum_{m=0}^\infty x^ms^{-1-m}.
\end{equation}
In order to obtain all the expansions it is enough to apply to formula \eqref{expa} the operators $\mathcal{D}$, $\Delta$, $\Delta \mathcal{D}$, $\mathcal{D}^2$, $\mathcal{\overline{D}}$, $ \overline{D}^2$ and $ \Delta \mathcal{\overline{D}}$. Due to Proposition \ref{conve} we can exchange the roles of the operators with the sum. Finally, in order to get the expansions written in terms of $x$ and $ \bar{x}$ we apply formula \eqref{dirac}, Proposition \ref{laplacian}, Proposition \ref{dirac_laplacian}, Proposition \ref{double_dirac}, Proposition \ref{l6}, Proposition \ref{antidouble} and Proposition \ref{antidirac}. By similar arguments we have the result for the right kernels.
\end{proof}

\begin{remark}
By Lemma \ref{lk} together with Theorem \ref{app1}, Theorem \ref{app2}, Theorem \ref{app3} and Theorem \ref{app6}, we deduce the following equalities
$$\mathcal Q_{c,s}(x)^{-2}=\sum_{m=3}^{\infty}\sum_{k=1}^{m-2}(m-k-1)k x^{m-k-2}\overline x^{k-1}s^{-1-m}=\sum_{m=3}^{\infty}\sum_{k=1}^{m-2}(m-k-1)k s^{-1-m}x^{m-k-2}\overline x^{k-1},$$
$$ S^{-1}_L(s,x) \mathcal{Q}_{c,s}(x)^{-1}=\sum_{m=2}^\infty\sum_{k=1}^{m-1} (m-k)x^{m-k-1}\overline x^{k-1}s^{-1-m},$$
$$Q_{c,s}(x)^{-1}=\sum_{m=1}^\infty \sum_{k=1}^m x^{m-k}\overline x^{k-1}s^{-1-m}=\sum_{m=1}^\infty \sum_{k=1}^m s^{-1-m} x^{m-k}\overline x^{k-1},$$
$$ S^{-1}_L (s, \bar{x})\qcs^{-1}=\sum_{m=2}^\infty \sum_{k=1}^{m-1} k x^{m-k-1} \bar{x}^{k-1}s^{-1-m}.$$
It is possible to obtain similar equalities with the right kernels.
\end{remark}

\section{Preliminaries on the $SC$-functional calculus and the $F$-functional calculus}\label{prelSCFFC}

We now recall some basic facts of the $SC$-functional calculus, see \cite{CS}. This functional calculus is the commutative version of the $S$-functional calculus (see \cite{CGKBOOK}). By $V$ we denote a real Banach space over $\rr$ with norm $\|\cdot\|$. It is possible to endow $V$ with an operation of multiplication by elements of $\rr_n$ that gives a two sided module over $\rr_n$. We denote by $V_n$ the two sided Banach module $V\otimes \rr_n$. An element of $V_n$ is of the form $\sum_{A\subset\{1,\ldots, n\}}e_A v_A$ with $v_A\in V$, where $e_\emptyset =1$ and $e_A=e_{i_1}\cdots e_{i_r}$ for $A=\{i_1,\ldots, i_r\}$ with $i_1<\cdots<i_r$. The multiplication (right and left) of an element $v\in V_n$ with a scalar $a\in\rr_n$ are defined as
$$ va=\sum_A v_A\otimes(e_A a)\quad\textrm{and}\quad av=\sum_A v_A\otimes (ae_A).$$
For short, we shall write $\sum_A v_Ae_A$ instead of $\sum_A v_A\otimes e_A$. Moreover, we define
$$\|v\|_{V_n}=\sum_A\|v_A\|_V.$$
Let $\mathcal B(V)$ be the space of bounded $\rr$-homomorphism of the Banach space $V$ into itself endowed with the natural norm denoted by $\|\cdot\|_{\mathcal B(V)}$.

If $T_A\in\mathcal B(V)$, we can define the operator $T=\sum_AT_A e_A$ and its action on
$$ v=\sum_{B}v_Be_B $$
as
$$ T(v)=\sum_{A,B} T_A(v_B)e_Ae_B.$$

The set of all such bounded operators is denoted by $\mathcal B(V_n)$. The norm is defined by
$$ \|T\|_{\mathcal B(V_n)}=\sum_A\|T_A\|_{\mathcal B(V)}. $$

In the following, we shall only consider operators of the form $T=T_0+\sum_{j=1}^ne_jT_j$, where $T_j\in\mathcal B(V)$ for $j=0,1,\ldots, n$, and we recall that the conjugate is defined by
$$\overline T=T_0-\sum_{j=1}^ne_j T_j.$$
The set of such operators in $\mathcal B(V_n)$ will be denoted by $\mathcal B^{0,1}(V_n)$. In this section we shall always consider $n$-tuples of bounded commuting operators, in paravector form, and we shall denote the set of such operators as $\mathcal{BC}^{0,1}(V_n)$.

\begin{definition}[the $S$-spectrum and the $S$-resolvent sets]
	Let $T\in \mathcal {BC}^{0,1}(V_n)$. We define the $S$-spectrum $\sigma_S(T)$ of $T$ as
	$$ \sigma_S(T)=\{s\in\rr^{n+1}:\, s^2\mathcal I -s(T+\overline T)+T\overline T\quad\textrm{is not invertible}\}. $$
	The $S$-resolvent set $\rho_S(T)$ is defined by
	$$ \rho_S(T)=\rr^{n+1}\setminus \sigma_S(T). $$
\end{definition}
For $s\in\rho_S(T)$, the operator
\begin{equation}\label{PSERES}
\mathcal Q_{c,s}(T)^{-1}:=(s^2\mathcal I-s(T+\overline T)+T\overline T)^{-1}
\end{equation}
is called the commutative pseudo $SC$-resolvent operator of $T$ at $s$.

\begin{theorem}
	Let $T\in\mathcal {BC}^{0,1}(V_n)$ and $s\in\mathbb \rr^{n+1}$ with $\|T\|< |s|$. Then we have
	$$ \sum_{n=0}^{\infty}T^n s^{-n-1}=-(s\mathcal I-\overline T)\mathcal{Q}_{c,s}(T)^{-1},$$
	and
	$$ \sum_{n=0}^{\infty}s^{-n-1}T^n=-\mathcal{Q}_{c,s}(T)^{-1}(s\mathcal I-\overline T). $$
\end{theorem}
According to the left or right slice hyperholomorphicity, there exist two different resolvent operators.
\begin{definition}[$SC$-resolvent operators]
	Let $T\in\mathcal{BC}^{0,1}(V_n)$ and $s\in\rho_S(T)$. We define the left $SC$-resolvent operator as
	$$ S^{-1}_{L}(s,T):=(s\mathcal I-\overline T)(s^2\mathcal I-s(T+\overline T)+T\overline T)^{-1}, $$
	and the right $SC$-resolvent operator as
	$$ S^{-1}_R(s, T):=(s^2\mathcal I-s(T+\overline T)+T\overline T)^{-1}(s\mathcal I-\overline T). $$
\end{definition}

An equation that involves both $SC$-resolvent operators is the so called $SC$-resolvent equation (see \cite{ACGS15}). In order to give a definition of $SC$-functional calculus we need the following classes of functions.

\begin{definition}
	Let $T \in \mathcal{BC}^{0,1}(X)$. We denote by $\mathcal{SH}_L(\sigma_S(T))$, $\mathcal{SH}_R(\sigma_S(T))$ and $N(\sigma_S(T))$ the sets of all left, right and intrinsic slice hyperholomorphic functions $f$ with $ \sigma_S(T) \subset dom(f)$.
\end{definition}

\begin{definition}[$SC$-functional calculus]
	\label{Sfun}
	Let $T \in \mathcal{BC}^{0,1}(V_n)(X)$. Let $U$ be a slice Cauchy domain that contains $\sigma_S(T)$  and $\overline{U}$ is contained in the domain of $f$.  Set $ds_J=-dsJ$. We define
	\begin{equation}
		\label{Scalleft}
		f(T):={{1}\over{2\pi }} \int_{\partial (U\cap \mathbb{C}_J)} S_L^{-1} (s,T)\  ds_J \ f(s), \ \ {\it for\ every}\ \ f\in \mathcal{SH}_L(\sigma_S(T)).
	\end{equation}
	We define
	\begin{equation}
		\label{Scalright}
		f(T):={{1}\over{2\pi }} \int_{\partial (U\cap \mathbb{C}_J)} \  f(s)\ ds_J
		\ S_R^{-1} (s,T),\ \  {\it for\ every}\ \ f\in \mathcal{SH}_R(\sigma_S(T)).
	\end{equation}
\end{definition}
It is possible to prove that the $SC$-functional calculus is well posed since the integrals in (\ref{Scalleft}) and (\ref{Scalright}) depend neither on $U$ and nor on the imaginary unit $J\in\mathbb{S}$.

Now we want to introduce the $F$-functional calculus (for a more complete presentation of this topic see \cite{CG}). The definition of the $F_n$-resolvent operators is suggested by the Fueter-Sce mapping theorem in integral form.
\begin{definition}[$F_n$-resolvent operators] Let $n$ be an odd number and let $T \in \mathcal{BC}^{0,1}(V_n)$. For $s\in\rho_F(T)$ we define the left $F$-resolvent operator as
	$$ F_n^{L}(s,T)=\gamma_n(s\mathcal{I}- \bar{T}) \mathcal{Q}_{c,s}(T)^{-\frac{n+1}2}, \qquad s \in \rho_{S}(T),$$
	and the right $F$-resolvent operator as
	$$ F_n^{R}(s,T)=\gamma_n\mathcal{Q}_{c,s}(T)^{-\frac{n+1}2}(s\mathcal{I}- \bar{T}) , \qquad s \in \rho_{S}(T).$$
\end{definition}

With the above definitions and Theorem \ref{Fueter} at hand, we can recall the $F$-functional calculus.

\begin{theorem}[The $F$-functional calculus for bounded operators]
	Let $n$ be an odd number, let $T\in\mathcal{BC}^{0,1}(V_n)$ and set $ds_J=ds/J$. Then, for any $f\in \mathcal{SH}_L(\sigma_S(T))$, we define
	\begin{equation}\label{DefFCLUb}
		\breve{f}(T):=\frac{1}{2\pi}\int_{\pp(U\cap \mathbb{C}_J)} F_n^L(s,T) \, ds_J\, f(s),
	\end{equation}
	and, for any $f\in \mathcal{SH}_R(\sigma_S(T))$, we define
	\begin{equation}\label{SCalcMON}
		\breve{f}(T):=\frac{1}{2\pi}\int_{\pp(U\cap \mathbb{C}_J)} f(s) \, ds_J\, F_n^R(s,T).
	\end{equation}
	The previous integrals depend neither on the imaginary unit $J\in\mathbb S$ nor on the set $U$.
\end{theorem}

\section{The functional calculi of the fine structures }\label{FCFSTRU}

Using the expressions of the kernels written in terms of $x$ and $ \bar{x}$ and the fine Fueter-Sce integral theorems, we can define the fine Fueter-Sce functional calculi.

\begin{definition}
	Let $T=T_0+\sum_{i=1}^5 e_i T_i \in \mathcal{BC}^{0,1}(V_5)$, $s \in \mathbb{R}^6$. We formally define
	\begin{itemize}
		\item the left and the right $\bigD$-kernel operator series as
		$$ -4 \sum_{m=1}^ \infty \sum_{k=1}^m T^{m-k} \bar{T}^{k-1}s^{-1-m}, $$
		and
		$$ -4 \sum_{m=1}^ \infty \sum_{k=1}^m s^{-1-m}T^{m-k} \bar{T}^{k-1}; $$
		\item the left and the right $\Delta$-kernel operator series as
		$$ -8 \sum_{m=2}^ \infty \sum_{k=1}^{m-1} (m-k)T^{m-k-1} \bar{T}^{k-1}s^{-1-m},$$
		and
		$$ -8 \sum_{m=2}^ \infty \sum_{k=1}^{m-1} (m-k)s^{-1-m}T^{m-k-1} \bar{T}^{k-1};$$
		\item the left and the right $\Delta\bigD$-kernel operator series as
		$$ 16 \sum_{m=3}^ \infty \sum_{k=1}^{m-2} (m-k-1)kT^{m-k-2} \bar{T}^{k-1}s^{-1-m},$$
		and
		$$ 16 \sum_{m=3}^ \infty \sum_{k=1}^{m-2} (m-k-1)k s^{-1-m} T^{m-k-2} \bar{T}^{k-1}.$$
		\item the left and the right $\bigD^2$-kernel operator series as
		$$ -8 \sum_{m=2}^\infty \sum_{k=1}^{m-1} k  T^{m-k-1} \bar{T}^{k-1} s^{-1-m}, $$
		and
		$$ -8 \sum_{m=2}^\infty \sum_{k=1}^{m-1} k s^{-1-m} T^{m-k-1} \bar{T}^{k-1}; $$
		\item the left and the right $\overline \bigD$-kernel operator series as
		$$ 2  \left[\sum_{m=1}^\infty  \left(mT^{m-1}  +2  \sum_{k=1}^{m} T^{m-k} \bar{T}^{k-1} \right) s^{-1-m} \right], $$
		and
		$$ 2  \left[\sum_{m=1}^\infty s^{-1-m} \left(mT^{m-1}  +2  \sum_{k=1}^{m} T^{m-k} \bar{T}^{k-1} \right) \right]; $$
		\item the left and the right $\overline \bigD^2$-kernel operator series as
		$$ 4 \left[\sum_{m=2}^\infty  \left(m(m-1)T^{m-2}+2  \sum_{k=1}^{m-1} (2m-k)T^{m-k-1}\bar{T}^{k-1}\right) s^{-1-m} \right], $$
		and
		$$ 4 \left[\sum_{m=2}^\infty s^{-1-m} \left(m(m-1)T^{m-2}+2  \sum_{k=1}^{m-1} (2m-k)T^{m-k-1}\bar{T}^{k-1}\right)  \right]; $$
		\item the left and the right $\Delta \overline\bigD$-kernel operator series as
		$$ -16 \sum_{m=3}^\infty \sum_{k=1}^{m-2} (m-k-1)(m+k)T^{m-k-2} \bar{T}^{k-1} s^{-1-m}, $$
		and
		$$ -16 \sum_{m=3}^\infty \sum_{k=1}^{m-2} (m-k-1)(m+k)s^{-1-m}T^{m-k-2} \bar{T}^{k-1}. $$
	\end{itemize}
\end{definition}

\begin{proposition}\label{sk}
	Let us consider $T \in \mathcal{BC}^{0,1}(V_5)$, $s \in \mathbb{R}^6$ and $ \| T\| < |s|$ then the series previously defined are convergent and, in particular, we can expand
	\begin{itemize}
		\item the left and right $\mathcal D$-resolvent operator as
		\begin{eqnarray*}
			S^{-1}_{\bigD,L}(s,T)&=& S^{-1}_{\bigD,R}(s,T) =-4 \sum_{m=1}^ \infty \sum_{k=1}^m T^{m-k} \bar{T}^{k-1}s^{-1-m}\\
			&=& -4 \sum_{m=1}^ \infty \sum_{k=1}^m s^{-1-m}T^{m-k} \bar{T}^{k-1};
		\end{eqnarray*}
		\item the left and right $\Delta$-resolvent operator as
		$$ S^{-1}_{\Delta, L}(s,T)= -8 \sum_{m=2}^ \infty \sum_{k=1}^{m-1} (m-k)T^{m-k-1} \bar{T}^{k-1}s^{-1-m},$$
		and
		$$ S^{-1}_{\Delta, R}(s,T)=  -8 \sum_{m=2}^ \infty \sum_{k=1}^{m-1} (m-k)s^{-1-m}T^{m-k-1} \bar{T}^{k-1};$$
		\item the left and right $\Delta\bigD$-resolvent operator as
		\begin{eqnarray*}
			S^{-1}_{\Delta\bigD,L}(s,T)&=& S^{-1}_{\Delta\bigD,R}(s,T)=16 \sum_{m=3}^ \infty \sum_{k=1}^{m-2} (m-k-1)kT^{m-k-2} \bar{T}^{k-1}s^{-1-m}\\
			&=& 16 \sum_{m=3}^ \infty \sum_{k=1}^{m-2} (m-k-1)k s^{-1-m} T^{m-k-2} \bar{T}^{k-1};
		\end{eqnarray*}
		\item the left and right $\bigD^2$-resolvent operator as
		$$ S^{-1}_{\bigD^2,L}(s,T) = -8 \sum_{m=2}^\infty \sum_{k=1}^{m-1} k  T^{m-k-1} \bar{T}^{k-1} s^{-1-m}, $$
		and
		$$ S^{-1}_{\bigD^2,R}(s,T) = -8 \sum_{m=2}^\infty \sum_{k=1}^{m-1} k s^{-1-m} T^{m-k-1} \bar{T}^{k-1}; $$
		\item the left and right $\overline\bigD$-resolvent operator as
		$$ S^{-1}_{\overline\bigD,L}(s,T) = 2  \left[\sum_{m=1}^\infty  \left(mT^{m-1}  +2  \sum_{k=1}^{m} T^{m-k} \bar{T}^{k-1} \right) s^{-1-m} \right], $$
		and
		$$ S^{-1}_{\overline\bigD,R}(s,T) = 2  \left[\sum_{m=1}^\infty s^{-1-m} \left(mT^{m-1}  +2  \sum_{k=1}^{m} T^{m-k} \bar{T}^{k-1} \right) \right]; $$
		\item the left and right $\overline\bigD^2$-resolvent operator as
		$$ S^{-1}_{\overline\bigD^2,L}(s,T) = 4 \left[\sum_{m=2}^\infty  \left(m(m-1)T^{m-2}+2  \sum_{k=1}^{m-1} (2m-k)T^{m-k-1}\bar{T}^{k-1}\right) s^{-1-m} \right], $$
		and
		$$ S^{-1}_{\overline\bigD^2,R}(s,T) = 4 \left[\sum_{m=2}^\infty s^{-1-m} \left(m(m-1)T^{m-2}+2  \sum_{k=1}^{m-1} (2m-k)T^{m-k-1}\bar{T}^{k-1}\right)  \right]; $$
		\item the left and right $\Delta\overline \bigD$-resolvent operator as
		$$ S^{-1}_{\Delta\overline \bigD,L}(s,T) = -16 \sum_{m=3}^\infty \sum_{k=1}^{m-2} (m-k-1)(m+k)T^{m-k-1} \bar{T}^{k-1} s^{-1-m}, $$
		and
		$$ S^{-1}_{\Delta\overline \bigD,R}(s,T) = -16 \sum_{m=3}^\infty \sum_{k=1}^{m-2} (m-k-1)(m+k)s^{-1-m}T^{m-k-1} \bar{T}^{k-1}. $$
	\end{itemize}
\end{proposition}
\begin{proof}
	The convergences of the series for $\|T\|<|s|$ can be proved considering the series of the operator norms and reasoning as in Proposition \ref{conve}. We prove only the first equality between the kernels and the series because the other equalities follow by similar arguments. Since $$S^{-1}_{\mathcal D,L}(s,T)=S^{-1}_{\mathcal D,R}(s,T)=-4\mathcal Q_{c,s}(T)^{-1},$$ it is sufficient to prove
	\begin{equation}\label{es1bis}
		\mathcal Q_{c,s}(T)\left(\sum_{m=1}^\infty  \sum_{k=1}^m T^{m-k} \overline T^{k-1} s^{-1-m}\right)=\left(\sum_{m=1}^\infty  \sum_{k=1}^m T^{m-k} \overline T^{k-1} s^{-1-m}\right)\mathcal Q_{c,s}(T)=\mathcal I .
	\end{equation}
	The first equality in \eqref{es1bis} is a consequence of the following facts: for any positive integer $m$ the sum $\sum_{k=1}^m T^{m-k}\overline T^{k-1}$ is an operator of the components of $T$ with real coefficients which then commute with any power of $s$, the components of $T$ are commuting among each other and the operator $\mathcal Q_{c,s}(T)$ can be written in the following form: $s^2\mathcal I-2s T_0+\sum_{i=0}^5 T_i^2$. Now we want to prove the second equality in \eqref{es1bis}. First we observe that
	\[
	\begin{split}
		&\left(\sum_{m=1}^\infty  \sum_{k=1}^m T^{m-k} \overline T^{k-1} s^{-1-m}\right)\mathcal Q_{c,s}(T)=\left(\sum_{m=1}^\infty  \sum_{k=1}^m T^{m-k} \overline T^{k-1} s^{-1-m}\right)(s^2-s(T+\overline T)+T\overline T)\\
		&= \sum_{m=1}^\infty  \sum_{k=1}^m T^{m-k}\overline{T}^{k-1}s^{1-m}- \sum_{m=1}^\infty  \sum_{k=1}^m T^{m+1-k}\overline T^{k-1}s^{-m} - \sum_{m=1}^\infty \sum_{k=1}^m T^{m-k}\overline T^k s^{-m}\\
		&\, \, \, \, \, \,+ \sum_{m=1}^{\infty} \sum_{k=1}^m T^{m-k+1} \bar{T}^{k}s^{-1-m}.
	\end{split}
	\]
	Making the change of index $m'=1+m$ in the second and fourth series, we have
	\[
	\begin{split}
		&\left(\sum_{m=1}^\infty  \sum_{k=1}^m T^{m-k} \overline T^{k-1} s^{-1-m}\right)\mathcal Q_{c,s}(T)=\\
		&= \sum_{m=1}^\infty  \sum_{k=1}^m T^{m-k}\overline{T}^{k-1}s^{1-m}- \sum_{m'=2}^\infty  \sum_{k=1}^{m'-1} T^{m'-k}\overline T^{k-1}s^{1-m'} - \sum_{m=1}^\infty \sum_{k=1}^m T^{m-k}\overline T^k s^{-m}\\
		&\, \, \, \, \, \,+\sum_{m'=2}^\infty  \sum_{k=1}^{m'-1} T^{m'-k} \overline T^k s^{-m'}\\
		&=\mathcal{I}+\sum_{m=2}^\infty  \sum_{k=1}^m T^{m-k}\overline{T}^{k-1}s^{1-m}- \sum_{m'=2}^\infty  \sum_{k=1}^{m'-1} T^{m'-k}\overline T^{k-1}s^{1-m'} +\\
		&\, \, \, \, \, \,-\bar{T}s^{-1}- \sum_{m=2}^\infty \sum_{k=1}^m T^{m-k}\overline T^k s^{-m}+\sum_{m'=2}^\infty  \sum_{k=1}^{m'-1} T^{m'-k} \overline T^k s^{-m'}.\\
	\end{split}
	\]
	Simplifying the opposite terms in the first and second series and in the third and fourth series, in the end we get
	\[
	\left(\sum_{m=1}^\infty  \sum_{k=1}^m T^{m-k} \overline T^{k-1} s^{-1-m}\right)\mathcal Q_{c,s}(T)=\mathcal I+\sum_{m=2}^\infty\overline T^{m-1}s^{1-m}-\sum_{m=2}^\infty\overline T^{m-1}s^{1-m}=\mathcal I.\\
	\]
\end{proof}

We can now define the resolvent operators of the fine structure and study their regularity.
Based on the previous series expansions and the structure of the kernels of the function spaces
we can now define the resolvent operators associated to the fine structure spaces.
Using these resolvent operators associated with the $S$-spectrum we will define the functional calculi associated with the respective functions spaces.

\begin{definition}[The resolvent operators associated with the fine structure]
	\label{DEFRES}
	Let $T\in \mathcal {BC}^{0,1}(V_5)$ and $s\in \rho_S(T)$, we recall that
$$
\mathcal Q_{c,s}(T)^{-1}:=(s^2\mathcal I-s(T+\overline T)+T\overline T)^{-1}.
$$
\begin{itemize}
\item
The left and the right  {\bf $\mathcal{D}$-resolvent operators} $S^{-1}_{\bigD ,L}(s,T)$ and $S^{-1}_{\bigD ,R}(s,T)$ are defined as
	\begin{equation}
		\label{oneRES}
		S^{-1}_{\bigD ,L}(s,T):=-4 \mathcal{Q}_{c,s}(T)^{-1},
	\end{equation}
	and
	\begin{equation}
		\label{secondRES}
		S^{-1}_{\bigD ,R}(s,T):=-4 \mathcal{Q}_{c,s}(T)^{-1}.
	\end{equation}
\item
 The left and the right {\bf $\Delta$-resolvent operators} $S^{-1}_{\Delta,L}(s,T)$ and $S^{-1}_{\Delta,R}(s,T)$ are defined as
		\begin{equation}
		\label{thirdRES}
		S^{-1}_{\Delta ,L}(s,T):=-8S^{-1}_L(s,T) \mathcal{Q}_{c,s}(T)^{-1},
	\end{equation}
	and
	\begin{equation}
		\label{fourthRES}
		S^{-1}_{\Delta ,R}(s,T):=-8 \mathcal{Q}_{c,s}(T)^{-1}S^{-1}_R(s,T) .
	\end{equation}
\item
	The left and the right
{\bf $\Delta\mathcal{D}$-resolvent operators} $S^{-1}_{\Delta\mathcal{D},L}(s,T)$ and $S^{-1}_{\Delta\mathcal{D},R}(s,T)$ are defined as
	\begin{equation}
		\label{fiveRES}
		S^{-1}_{\Delta\bigD ,L}(s,T):=16 \mathcal{Q}_{c,s}(T)^{-2},
	\end{equation}
	and
	\begin{equation}
		\label{sixRES}
		S^{-1}_{\Delta\bigD ,R}(s,T):=16 \mathcal{Q}_{c,s}(T)^{-2}.
\end{equation}
\item
	The left and the right {\bf $ \mathcal{\overline{D}}$-resolvent operators} $S^{-1}_{\mathcal{\overline{D}},L}(s,T)$ and $S^{-1}_{\mathcal{\overline{D}},R}(s,T)$ are defined as
	\begin{equation}\label{sevenRES}
		S^{-1}_{\overline\bigD ,L}(s,T):=4(s-\bar T)\mathcal{Q}_{c,s}(T)^{-2}(s-T_0)+2\mathcal{Q}_{c,s}(T)^{-1}
	\end{equation}
	and
	\begin{equation}\label{eightRES}
		S^{-1}_{\overline\bigD ,R}(s,T):=4(s-T_0)\mathcal{Q}_{c,s}(T)^{-2}(s-\bar T)+2\mathcal{Q}_{c,s}(T)^{-1}.
	\end{equation}
		\item
The left and the right {\bf $\mathcal{\overline{D}}^2$-resolvent operators} $S^{-1}_{\mathcal{\overline{D}}^2,L}(s,T)$ and $S^{-1}_{\mathcal{\overline{D}}^2,R}(s,T)$ are defined as
	\begin{equation}\label{nineRES}
		S^{-1}_{\overline\bigD^2 ,L}(s,T):=32(s-\bar T)\mathcal{Q}_{c,s}(T)^{-3}(s-T_0)^2
	\end{equation}
	and
	\begin{equation}\label{tenRES}
		S^{-1}_{\overline\bigD^2 ,R}(s,T):=32(s-T_0)^2\mathcal{Q}_{c,s}(T)^{-3} (s-\bar T).
	\end{equation}
	
\item
The left and the right  {\bf $ \bigD^2$-resolvent operators} $S^{-1}_{\bigD^2,L}(s,T)$ and $S^{-1}_{\bigD^2,R}(s,T)$ are defined as
	\begin{equation}\label{elevenRES}
		S^{-1}_{\bigD^2 ,L}(s,T):=8S^{-1}_L (s, \bar{T})\mathcal{Q}_{c,s}(T)^{-1}
	\end{equation}
	and
	\begin{equation}\label{twelveRES}
		S^{-1}_{\bigD^2 ,R}(s,T):=8\mathcal{Q}_{c,s}(T)^{-1} S^{-1}_R (s, \bar{T}).
	\end{equation}
	\item
The left and the right  {\bf $ \Delta\mathcal{\overline{D}}$-resolvent operators}
$S^{-1}_{\Delta \mathcal{\overline{D}},L}(s,T)$ and $S^{-1}_{\Delta\mathcal{\overline{D}},R}(s,T)$ are defined as
	\begin{equation}\label{thirteenRES}
		S^{-1}_{\Delta\overline\bigD ,L}(s,T):=-64(s-\bar T)\mathcal{Q}_{c,s}(T)^{-3}(s-T_0)
	\end{equation}
	and
	\begin{equation}\label{fourteenRES}
		S^{-1}_{\Delta\overline\bigD ,R}(s,x):= -64 (s-T_0)\mathcal{Q}_{c,s}(T)^{-3} (s-\bar T).
	\end{equation}
\end{itemize}
\end{definition}

Now, we study the regularity of the previous kernels.

\begin{proposition}\label{reg1RES}
	Let $T\in \mathcal {BC}^{0,1}(V_5)$.  Then
the resolvent operators associated with the fine structure
in Definition \ref{DEFRES}
are slice hyperholomorphic operators valued functions for 	
and $s\in \rho_S(T)$.
\end{proposition}
\begin{proof}
It follows by a direct computations and in the case of the $S$-resolvent operators or the $F$-resolvent operators.
\end{proof}

Now, we can to define the functional calculi associated with the fine structure.

\begin{definition}[The functional calculi of the fine structure in dimension five]\label{fcalc}
	Let $T \in \mathcal{BC}^{0,1}(V_5)$ and set $ds_J=ds(-J)$ for $J \in \mathbb{S}$. Let $f$ be a function that belongs to $\mathcal{SH}_L(\sigma_S(T))$ or belongs to $\mathcal{SH}_R(\sigma_S(T))$. Let $U$ be a bounded slice Cauchy domain with $\sigma_S(T)\subset U$ and $\overline U\subset\operatorname{dom}(f)$.

Keeping in mind the resolvent operators associated with the fine structure in  Definition \ref{DEFRES} we
define functional calculi associated of the fine structure as:
	\begin{itemize}
		\item ({\bf The $\ell+1$-harmonic functional calculus for $0\leq\ell\leq 1$}) For every function $\tilde{f}_\ell=\Delta^{1-\ell}\mathcal{D}f$ with $f \in \mathcal{SH}_L(\sigma_S(T))$ and $0\leq\ell\leq 1$, we set
		\begin{equation}
			\label{fun1}
			\tilde{f}_\ell(T):=\frac 1{2\pi} \int_{\partial(U \cap \mathbb{C}_J)} S^{-1}_{\Delta^{1-\ell}\bigD,L}(s,T) ds_J f(s),
		\end{equation}
		and, for every function $ \tilde{f}_\ell=f\Delta^{1-\ell}\mathcal{D}$ with $f \in \mathcal{SH}_R(\sigma_S(T))$, we set
		\begin{equation}
			\label{fun2}
			\tilde{f}_\ell(T):=\frac 1{2\pi} \int_{\partial(U \cap \mathbb{C}_J)} f(s) ds_JS^{-1}_{\Delta^{1-\ell}\bigD,R}(s,T).
		\end{equation}
		\item ({\bf The holomorphic Cliffordian functional calculus of order 1}) For every function $f^\circ=\Delta f$ with $f \in \mathcal{SH}_L(\sigma_S(T))$, we set
		\begin{equation}
			\label{fun3}
			f^\circ(T):= \frac{1}{2\pi} \int_{\partial(U \cap \mathbb{C}_J)}
S^{-1}_{\Delta,L}(s,T) ds_J f(s),
		\end{equation}
		and, for every function $f^\circ=\Delta f$ with $f \in \mathcal{SH}_R(\sigma_S(T))$, we set
		\begin{equation}
			\label{fun4}
			f^\circ(T):= \frac{1}{2\pi} \int_{\partial(U \cap \mathbb{C}_J)} f(s) ds_JS^{-1}_{\Delta,R}(s,T).
		\end{equation}
		\item ({\bf The polyanalytic Cliffordian functional calculus of order $(1,2)$}) For every function $\breve f^\circ=\overline{\mathcal{D}}f$ with $f \in \mathcal{SH}_L(\sigma_S(T))$, we set
		\begin{equation}
			\label{fun5}
			\breve{f}^\circ(T):= \frac{1}{2\pi} \int_{\partial(U \cap \mathbb{C}_J)} S^{-1}_{\overline\bigD,L}(s,T) ds_J f(s),
		\end{equation}
		and, for every function $ \breve{f}^\circ=f\overline{\mathcal{D}}$ with $f \in \mathcal{SH}_R(\sigma_S(T))$, we set
		\begin{equation}
			\label{fun6}
			\breve{f}^\circ (T):= \frac{1}{2\pi} \int_{\partial(U \cap \mathbb{C}_J)} f(s) ds_JS^{-1}_{\overline\bigD,R}(s,T).
		\end{equation}
		\item ({\bf The polyanalytic functional calculus of order $3-\ell$ for $0\leq\ell\leq 1$}) For every function $\breve f_\ell=\Delta^\ell \overline\bigD^{2-\ell}f$ with $f \in \mathcal{SH}_L(\sigma_S(T))$, we set
		\begin{equation}
			\label{fun7}
			\breve{f}_\ell(T):= \frac{1}{2\pi} \int_{\partial(U \cap \mathbb{C}_J)} S^{-1}_{\Delta^\ell \overline\bigD^{2-\ell},L}(s,T) ds_J f(s),
		\end{equation}
		and, for every function $ \breve{f}_\ell=f\overline{\mathcal{D}}$ with $f \in \mathcal{SH}_R(\sigma_S(T))$, we set
		\begin{equation}
			\label{fun8}
			\breve{f}^\circ (T):=\frac{1}{2\pi} \int_{\partial(U \cap \mathbb{C}_J)} f(s) ds_JS^{-1}_{\Delta^\ell \overline\bigD^{2-\ell},R}(s,T) .
		\end{equation}
		\item ({\bf The anti-holomorphic Cliffordian functional calculus of order $1$}) For every function $f_0=\mathcal{D}^2f$ with $f \in \mathcal{SH}_L(\sigma_S(T))$, we set
		\begin{equation}
			\label{fun9}
			f_0 (T):= \frac{1}{2\pi} \int_{\partial(U \cap \mathbb{C}_J)} S^{-1}_{\bigD^2,L}(s,T) ds_J f(s),
		\end{equation}
		and, for every function $ f_0=f\mathcal{D}^2$ with $f \in \mathcal{SH}_R(\sigma_S(T))$, we set
		\begin{equation}
			\label{fun10}
			f_0(T):= \frac{1}{2\pi} \int_{\partial(U \cap \mathbb{C}_J)} f(s) ds_JS^{-1}_{\bigD^2,R}(s,T) .
		\end{equation}
	\end{itemize}
\end{definition}

\begin{theorem}
	The previous functional calculi are well defined, in the sense that the integrals in Definition \ref{fcalc} depend neither on the imaginary unit $J \in \mathbb{S}$ and nor on the slice Cauchy domain $U$.
\end{theorem}
\begin{proof}
	We prove the result for the functional calculi defined using the left slice hyperholomorphic functions since the right counterpart can be proved with similar computations. The only property of the kernels that we shall use to prove the theorem is that they are all right slice hyperholomoprhic in the variable $s$ (see Proposition \ref{reg1}). For this reason, in what follows, we can refer to an arbitrary left kernel described in Proposition \ref{sk} with the symbol $K_L(s,T)$.
	\\ Since $ K_L(s,T)$ is a right-slice hyperholomorphic function in $s$ and $f$ is left slice hyperholomorphic, the independence from the set $U$ follows by the Cauchy integral formula (see Theorem \ref{Cauchy}).
	\\ Now, we want to show the independence from the imaginary unit, let us consider two imaginary units $J$, $I \in \mathbb{S}$ with $J \neq I$ and two bounded slice Cauchy domains $U_x$, $U_s$ with $ \sigma_{S}(T) \subset U_x$, $\overline{U}_x \subset U_s$ and $\overline{U}_s \subset dom(f)$. Then every $s \in \partial (U_s \cap \mathbb{C}_J)$ belongs to the unbounded slice Cauchy domain $\mathbb{R}^6\setminus  U_x $.
	\\ Since $ \lim_{x \to + \infty} K_{L}(x,T)=0$, the slice hyperholomorphic Cauchy formula implies
	\begin{eqnarray}
		\nonumber
		K_L(s,T)&=& \frac{1}{2 \pi} \int_{\partial (\mathbb{R}^6 \setminus (U_x \cap \mathbb{C}_I))} K_L(x,T) dx_I S^{-1}_R(x,s)\\
		\label{inte3}
		&=& \frac{1}{2 \pi} \int_{\partial (U_x \cap \mathbb{C}_I)} K_L(x,T) dx_I S^{-1}_L(s,x).
	\end{eqnarray}
	The last equality is due to the facts that $ \partial (\mathbb{R}^6 \setminus (U_x \cap \mathbb{C}_I))=-\partial (U_x \cap \mathbb{C}_I)$ and $S^{-1}_R(x,s)=-S^{-1}_L(s,x).$ Thus, by formula \eqref{inte3} we get
	\begin{eqnarray*}
		\tilde{f}(T)&=& \int_{\partial(U_s \cap \mathbb{C}_J)} K_L(s,T) ds_J f(s)\\
		&=& \int_{\partial(U_s \cap \mathbb{C}_J)} \left( \frac{1}{2 \pi} \int_{\partial(U_x \cap \mathbb{C}_I)}  K_L(x,T) dx_I S_{L}^{-1}(s,x)\right) ds_J f(s).
	\end{eqnarray*}
	Due to Fubini's theorem we can exchange the order of integration and by the Cauchy formula we obtain
	\begin{eqnarray*}
		\tilde{f}(T)&=&  \int_{\partial (U_x \cap \mathbb{C}_I)} K_L(x,T)dx_I  \left( \frac{1}{2 \pi} \int_{\partial (U_s \cap \mathbb{C}_J)}   S_{L}^{-1}(s,x) ds_J f(s) \right)\\
		&=&  \int_{\partial(U_x \cap \mathbb{C}_I)} K_L(x,T) dx_I f(x).
	\end{eqnarray*}
	This proves the statement.
\end{proof}

In what follows we denote by $\mathsf P$ one of the operators: $\Delta^{1-\ell}\mathcal D$, $\Delta$, $\overline{\mathcal D}$, $\Delta^\ell\overline{\mathcal D}^{2-\ell}$ and $\mathcal D^2$ for $\ell=0,\, 1$.
\begin{prob}
	\label{three}
  Let $ \Omega$ be a slice Cauchy domain. It might happen that $f,g\in \mathcal{SH}_L(\Omega)$ (resp. $f,g\in \mathcal {SH}_R(\Omega)$) and $\mathsf Pf=\mathsf Pg$  (resp. $f\mathsf P=g\mathsf P$). Is it possible to show that for any $T\in\mathcal{BC}^{0,1}(V_5)$, with $\sigma_S(T)\subset \Omega$, we have $(\mathsf Pf)(T)=(\mathsf Pg)(T)$ (resp. $(f\mathsf P)(T)=(g\mathsf P)(T)$)?
\end{prob}

In order to address the problem  we need an auxiliary result. We start by observing that by hypothesis we have $\mathsf P(f-g)=0$ (resp. $(f-g)\mathsf P=0$).
Therefore it is necessary to study the set
$$ (\ker{\mathsf P})_{\mathcal{SH}_L(\Omega)}:=\{f\in \mathcal{SH}_L(\Omega): \mathsf P(f)=0\}\quad\textrm{resp. $(\ker{\mathsf P})_{\mathcal{SH}_R(\Omega)}:=\{f\in \mathcal{SH}_R(\Omega): (f)\mathsf P=0\}$} .$$
\begin{theorem}\label{Tcost} Let $\Omega$ be a connected slice Cauchy domain of $\mathbb R^6$, then
	\begin{eqnarray*}
		& (\ker{\mathsf P})_{\mathcal{SH}_L(\Omega)} =\{f\in \mathcal{SH}_L(\Omega): f\equiv \alpha_0+\dots+x^t \alpha_t\quad\textrm{for some $\alpha_0,\dots,\alpha_t\in\mathbb R_n$}\}\\
		&=\{f\in \mathcal{SH}_R(\Omega): f\equiv \alpha_0+\dots+\alpha_tx^t\quad\textrm{for some $\alpha_0,\dots,\alpha_t\in\mathbb R_n$}\}=(\ker{\mathsf P})_{\mathcal {SH}_R(\Omega)}
	\end{eqnarray*}
where $t$ is equal to the degree of $P$ minus 1. 	
\end{theorem}
\begin{proof}
	We prove the result in the case $f\in \mathcal{SH}_L(\Omega)$ since the case $f\in \mathcal{SH}_R(\Omega)$ follows by similar arguments. We proceed by double inclusion. The fact that
	$$
	(\ker{\mathsf P})_{\mathcal{SH}_L(\Omega)}\supseteq\{f\in \mathcal{SH}_L(\Omega): f\equiv \alpha_0+\dots+x^t \alpha_t\quad\textrm{for some $\alpha_0,\dots,\alpha_t\in\mathbb R_n$}\}
	$$ is obvious. The other inclusion can be proved observing that if $f\in (\ker{\mathsf P})_{\mathcal{SH}_L(\Omega)}$, after a change of variable if needed, there exists $r>0$ such that the function $f$ can be expanded in a convergent series at the origin
	$$f(x)=\sum_{k=0}^{\infty}x^k \alpha_k\quad\textrm{for $(\alpha_k)_{k\in\mathbb N_0}\subset\mathbb R_n$ and for any $x\in B_r(0)$}$$
	where $B_r(0)$ is the ball centered at $0$ and of radius $r$. We have
	$$
	0=\mathsf Pf(x)\equiv\sum_{k=\operatorname{deg}(\mathsf P)}^{\infty} \mathsf P(x^k)\alpha_k,\quad \forall x\in B_r(0).
	$$
	By Lemma \ref{begher}, Proposition \ref{laplacian}, Proposition \ref{double_dirac}, Proposition \ref{dirac_laplacian}, Proposition \ref{l6}, Proposition \ref{antidouble} and Proposition \ref{antidirac} we can compute explicitly the expressions $\mathsf P(x^k)$ and, when they are restricted to a neighborhood of zero of the real axis, they coincide up to a constant to $x^{k-\operatorname{deg}(\mathsf P)}$. Since the power series is identically zero, its coefficients $\alpha_k$'s must be zero for any $k\geq \operatorname{deg}(\mathsf P)$. This yields $f(x)\equiv \alpha_0+\dots +x^t\alpha_t$ in $B_r(0)$ and, since $\Omega$ is connected, we also have $f(x)\equiv \alpha_0+\dots +x^t\alpha_t$ for any $x\in\Omega$.
\end{proof}
We solve the problem \ref{three} in the case in which $ \Omega$ is connected.
\begin{proposition}\label{con}
	Let $T\in\mathcal{BC}^{0,1}(V_5)$ and let $U$ be a  connected slice   Cauchy domain with $\sigma_S(T)\subset U$. If $f,g\in \mathcal{SH}_L(U)$ (resp. $f,g\in \mathcal{SH}_R(U)$) satisfy the property $\mathsf Pf=\mathsf Pg$ (resp. $f\mathsf P=g\mathsf P$) then $ (\mathsf Pf)(T)=(\mathsf Pg)(T)$ (resp. $(f\mathsf P)(T)=(g\mathsf P)(T)$).
\end{proposition}
\begin{proof}
	We prove the theorem in the case $f,g\in \mathcal{SH}_L(\Omega)$ since the case $f,g\in\mathcal{SH}_R(\Omega)$ follows by similar arguments. By Definition \ref{fcalc}, we have
	$$
	(\mathsf P f)(T)-(\mathsf Pg)(T)=\frac 1{2\pi}\int_{\partial (U\cap\cc_J)}S_{\mathsf P,L}^{-1}(s,T)ds_J(f(s)-g(s)).
	$$
	Since $S_{\mathsf P,L}^{-1}(s,T)$ is slice hyperholomorphic in the variable $s$ by Proposition \ref{reg1RES}, we can change the domain of integration to $B_r(0)\cap\cc_J$ for some $r>0$ with $ \| T \| <r$. Moreover, by hypothesis we have that $f(s)-g(s) \in (\ker{\mathsf P})_{\mathcal{SH}_L(\Omega)}$, thus by Theorem \ref{Tcost} and Proposition \ref{sk} we get
	\[
	\begin{split}
		&(\mathsf Pf)(T)-(\mathsf Pg)(T)=\frac 1{2\pi}\int_{\partial (B_r(0)\cap\cc_J)} S^{-1}_{\mathsf P,L}(s,T)ds_J(f(s)-g(s))
		\\
		&
		=\frac 1{2\pi}\int_{\partial (B_r(0)\cap\cc_J)} S^{-1}_{\mathsf P,L}(s,T) ds_J(\alpha_0+\dots+s^t\alpha_t)\\
		&= \frac 1{2\pi} \sum_{m=\operatorname{deg}(\mathsf P)}^\infty\left(g_{\mathsf P,m}(T,\bar T)\right) \int_{\partial (B_r(0)\cap \cc_J)}s^{-1-m} ds_J (\alpha_0+\dots+s^t\alpha_t)=0.
	\end{split}
	\]
	where $g_{\mathsf P,m}(T,\bar T)$ is a polynomial in $T$ and $\bar T$ (see Proposition \ref{sk}) and $t:=\operatorname{deg}(\mathsf P)-1$.
\end{proof}
Now, we write the resolvent operators associated with the fine structures in terms of the $F_n$- resolvent operators.

\begin{proposition}
\label{Fkernel}
Let $T \in \mathcal{BC}^{0,1}(V_5)$ and $s \in \rho_S(T)$. Then, we have
\begin{eqnarray*}
S^{-1}_{\mathcal D,L}(s,T)&=&-\frac 1{16}\left[F^L_5(s,T)s^3-(T+2T_0)F^L_5(s,T)s^2+(2T_0x+|x|^2)F^L_5(s,T)s-T|T|^2F^L_5(s,T)\right],\\
S^{-1}_{\mathcal D,R}(s,T)&=&-\frac 1{16}\left[s^3F^R_5(s,T)-s^2F^R_5(s,T)(T+2T_0)+sF^R_5(s,T)(2T_0T+|T|^2)-F^R_5(s,T)T|T|^2\right],
\end{eqnarray*}
\begin{eqnarray*}
S^{-1}_{\Delta,L}(s,T)&=&-\frac 1{8}\left[ F^L_5(s,T)s^2-2T_0F^L_5(s,T)s+|T|^2F^L_5(s,x) \right],\\
S^{-1}_{\Delta,R}(s,T)&=&-\frac 1{8}\left[ F^R_5(s,T)T^2-2sF^R_5(s,T)T_0+F^R_5(s,x)|T|^2 \right],\\
\end{eqnarray*}
\begin{eqnarray*}
S^{-1}_{\Delta\mathcal D,L}(s,T)&=&\frac {1}{4}\left[F^{L}_5(s,T)s-TF^L_5(s,T)\right],\\
S^{-1}_{\Delta\mathcal D,R}(s,T)&=&\frac {1}{4}\left[sF^{R}_5(s,T)-F^R_5(s,T)T\right],\\
\end{eqnarray*}

\begin{eqnarray*}
S^{-1}_{\overline{\mathcal D},L}(s,T)&=& \frac 1{32}\left[3F^L_5(s,T)s^3-(8T_0+T)F^L_5(s,T)s^2+(4T_0^2+2T_0T+3|T|^2)F^L_5(s,T)s\right.\\
&&\left.-(T|T|^2+2T_0|T|^2)F^L_5(s,T)\right],\\
S^{-1}_{\overline{\mathcal D},R}(s,T)&=& \frac 1{32}\left[3s^3F^R_5(s,T)-s^2F^R_5(s,T)(8T_0+T)+sF^R_5(s,T)(4T_0^2+2T_0T+3|T|^2)\right.\\
&&\left.-F^R_5(s,x)(T|T|^2+2T_0|T|^2)\right],\\
\end{eqnarray*}

\begin{eqnarray*}
S^{-1}_{\mathcal{D}^2,L}(s,T)&=&-\frac 1{8}\left[F^L_5(s,T)s^2-2T F^L_5(s,T)s+T^2F^L_5(s,T)\right],\\
S^{-1}_{\mathcal{D}^2,R}(s,T)&=&-\frac 1{8}\left[s^2F^R_5(s,T)-2 sF^R_5(s,T)T+F^R_5(s,T)T^2\right],\\
\end{eqnarray*}

\begin{eqnarray*}
S^{-1}_{\mathcal{\overline{D}}^{2},L}(s,T)&=&\frac{1}{2}[F_{5}^{L}(s,T)s^2-2T_0 F_{5}^{L}(s,T)s+ x_0^2 F_{5}^{L}(s,T)],\\
\nonumber
S^{-1}_{\mathcal{\overline{D}}^{2},R}(s,T)&=&\frac{1}{2}[s^2F_{5}^{R}(s,T)-2 sF_{5}^{R}(s,T)T_0+  F_{5}^{R}(s,T)T_0^2],\\
\nonumber
\end{eqnarray*}

\begin{eqnarray*}
S^{-1}_{\Delta\mathcal{\overline{D}},L}(s,T)&=&-F_{5}^L(s,x)s+T_0 F_{5}^L(s,T),\\
\nonumber
S^{-1}_{\Delta\mathcal{\overline{D}},R}(s,T)&=&-sF_{5}^R(s,T)+ F_{5}^R(s,T)T_0.
\end{eqnarray*}
\end{proposition}
\begin{proof}
it follows by formally replacing the variable $x$ of Remark \ref{fkernels} by the paravector operator $T$.
\end{proof}
In order to solve Problem \ref{three}, in the case $ \Omega$ not connected, we need the following lemma, which is based on the monogenic functional calculus developed in \cite{J,JM}. We chose to annihilate the last component of the operator $T$, namely $T_4=0$. In the monogenic functional calculus  McIntosh and collaborators consider zero the first component $T_0=0$. However, in our case this is a drawback due to the structure of the polyanalytic resolvent operators, see Proposition \ref{Fkernel}.

\begin{lemma}
	\label{mono}
	Let $T \in \mathcal{BC}^{0,1}(V_5)$ be such that $T= T_0e_0+ T_1e_1+ T_2 e_2 + T_3 e_3$, and assume that the operators $T_{\ell}$, $\ell=0,1,2,3$, have real spectrum. Let $G$ be a bounded slice Cauchy domain such that $(\partial G) \cap \sigma_{S}(T)= \emptyset$. For every $J \in \mathbb{S}$ we have
	\begin{equation}
		\int_{\partial{(G \cap \mathbb{C}_J)}} S^{-1}_{\mathsf P,L}(s,T) ds_J (\alpha_0+\dots+ s^t\alpha_t)=0,
	\end{equation}
and
	\begin{equation}
	\label{right}
	\int_{\partial{(G \cap \mathbb{C}_J)}} (\alpha_0+\dots+ s^t\alpha_t) ds_J S^{-1}_{\mathsf P,R}(s,T)=0,
\end{equation}
where $t=\operatorname{deg}(\mathsf P)-1$ and $\alpha_j\in\mathbb R_n$ for any $0\leq j\leq t$.
\end{lemma}
\begin{proof}
	Since $\Delta^2(1)=0$, $\Delta^2(x)=0$, $\Delta^2(x^2)=0$ and $\Delta^2(x^3)=0$ by Theorem \ref{Fueter} we also have
	\begin{equation}\label{nova}
		\int_{\partial{(G \cap \mathbb{C}_J)}} F_5^L(s,x) ds_J=\Delta^2(1)=0,
	\end{equation}
	and
	\begin{equation}\label{nova_bis}
		\int_{\partial{(G \cap \mathbb{C}_J)}} F_5^L(s,x) ds_Js =\Delta^2(x)=0,
	\end{equation}
	and
	\begin{equation}\label{nova_tris}
		\int_{\partial{(G \cap \mathbb{C}_J)}} F_5^L(s,x) ds_Js^2 =\Delta^2(x^2)=0,
	\end{equation}
	and
	\begin{equation}\label{nova_quattro}
		\int_{\partial{(G \cap \mathbb{C}_J)}} F_5^L(s,x) ds_Js^3 =\Delta^2(x^3)=0
	\end{equation}
	for all $x \notin \partial G$ and $J \in \mathbb{S}$. By the monogenic functional calculus \cite{J, JM} we have
	$$ F^L_5(s,T)= \int_{\partial \Omega} G(\omega,T) \mathbf{D}\omega F^L_5(s,\omega),$$
	where $\mathbf{D}\omega$ is a suitable differential form, the open set $ \Omega$ contains the left spectrum of $T$ and $G(\omega,T)$ is the Fueter resolvent operator. By Proposition \ref{Fkernel} we can write
	$$S^{-1}_{\mathsf P,L}(s,T)=\sum_{\ell=0}^{4-\operatorname{deg}(\mathsf P)} g_{\mathsf P,\ell}(T,\bar T)F^L_5(s,T) s^{\ell},$$
	and thus we have
	$$S^{-1}_{\mathsf P,L}(s,T) (\alpha_0+\dots+ s^t\alpha_t) =\sum_{\ell=0}^{3} g'_{\mathsf P,\ell}(T,\bar T)F^L_5(s,T) s^{\ell}\beta_\ell$$
	for appropriate polynomial in $T, \, \bar T$: $g'_{\mathsf P,\ell}(T,\bar T)$, and $\beta_{\ell}\in\mathbb R_n$. We can conclude the proof of the theorem observing that for any $\ell=0,1,2,3$ we have
	\begin{eqnarray*}
		&& g_{\mathsf P,\ell}(T,\bar T) \int_{\partial{(G \cap \mathbb{C}_J)}}  F^L_5(s,T) ds_J s^\ell \beta_\ell=- \left(g_{\mathsf P,\ell}(T,\bar T) \int_{\partial{(G \cap \mathbb{C}_J)}}  \int_{\partial\Omega} G(\omega, T)\mathbf{D}\omega F_L(s,\omega)s^{\ell}\, ds_J \right)\beta_\ell\\
		&&=\left( g_{\mathsf P,\ell}(T,\bar T) \int_{\partial \Omega} G(\omega,T) \mathbf{D}\omega \left(\int_{\partial{(G \cap \mathbb{C}_J)}} F_L(s,\omega) ds_J s^\ell \right)\right)\beta_\ell= 0
	\end{eqnarray*}
	where the second equality is a consequence of the Fubini's Theorem and the last equality is a consequence of formulas \eqref{nova}-\eqref{nova_quattro}. Therefore, we get
$$	\int_{\partial{(G \cap \mathbb{C}_J)}} S^{-1}_{\mathsf P,L}(s,T) ds_J (\alpha_0+\dots+ s^t\alpha_t)=0.	$$
By similar computations it is possible to show \eqref{right}.
\end{proof}

Finally in the following result we give an answer to the question in Problem \ref{three}.
\begin{proposition}
	Let $T \in \mathcal{BC}^{0,1}(V_5)$ be such that $T= T_0e_0+T_1e_1+T_2 e_2+T_3e_3$, and assume that the operators $T_{\ell}$, $\ell=0,1,2,3$, have real spectrum. Let $U$ be a slice Cauchy domain with $\sigma_S(T)\subset U$. If $f,g\in \mathcal{SH}_L(U)$ (resp. $f,g\in \mathcal{SH}_R(U)$) satisfy the property $\mathsf Pf=\mathsf Pg$ (resp $f\mathsf P=g\mathsf P$) then $(\mathsf Pf)(T)=(\mathsf Pg)(T)$ (resp. $(f\mathsf P)(T)=(g\mathsf P)(T)$).
\end{proposition}
\begin{proof}
	If $U$ is connected we can use Proposition \ref{con}. If $U$ is not connected then $U=\cup_{\ell=1}^m U_\ell$ where the $U_\ell$ are the connected components of $U$. Hence, there exist constants $\alpha_{\ell,i}\in\mathbb R_n$ for $\ell=1,\dots, m$ and $i=0,1,2,3$ such that $f(s)-g(s)=\sum_{\ell=1}^m \sum_{i=0}^t \chi_{U_\ell}(s) s^i \alpha_{\ell,i}$ where $t=\operatorname{deg}(\mathsf P)-1$. Thus we can write
	$$
	(\mathsf P f)(T)-(\mathsf Pg)(T)=\sum_{\ell=1}^m\frac 1{2\pi}\int_{\partial(U_\ell\cap\mathbb C_J)}S^{-1}_{\mathsf P,L}(s,T)ds_J(\alpha_{\ell,0}+\dots+s^t\alpha_{\ell,t}).
	$$
	The last summation is zero by Lemma \ref{mono}.
\end{proof}

\begin{remark}
	It is possible to prove some other important properties for these functional calculi, this will be investigated in a forthcoming work.
\end{remark}

\section{The product rule for the $F$-functional calculus}\label{PRDRULE}

In order to obtain a product rule for the $F$-functional calculus in dimension five, it is crucial the Dirac fine structure of the kind $(\mathcal{D}, \mathcal{\overline{D}}, \mathcal{D}, \mathcal{\overline{D}} )$, see \eqref{fine3}.
\begin{lemma}
	\label{comm}
	Let $B \in \mathcal{B}^{0,1}(V_n)$. Let $G$ be a bounded slice Cauchy domain and let $f$ be an intrinsic slice monogenic function whose domain contains $G$. Then for $p \in G$, and for any $I \in \mathbb{S}$ we have
	$$ \frac{1}{2 \pi} \int_{\partial(G_1 \cap \mathbb{C}_J)} f(s) ds_I (\bar{s}B-Bp)(p^2-2s_0p+|s|^2)^{-1}=Bf(p).$$
\end{lemma}

In order to show a product rule for the $F$-functional calculus we need the following result, see \cite[Lemma 4.1]{CDS}.
\begin{lemma}[the $F$-resolvent equation for $n=5$]
	\label{reseq}
	Let $T \in \mathcal{BC}^{0,1}(V_5)$. Then for $p$, $s \in \rho_S(T)$ the following equation holds
	\begin{eqnarray*}
		&&F_5^R(s,T)S^{-1}_L(s,T)+S^{-1}_R(s,T) F_5^L(p,T)+2^6 \mathcal{Q}_{c,s}(T)^{-1}S^{-1}_R(s,T) S^{-1}_{L}(p,T) \mathcal{Q}_{c,p}(T)^{-1}+\\
		&&+2^{6}[ \mathcal{Q}_{c,s}(T)^{-2} \mathcal{Q}_{c,p}(T)^{-1}+ \mathcal{Q}_{c,s}(T)^{-1} \mathcal{Q}_{c,p}(T)^{-2}]= \{ [F_{5}^R(s,T)- F_5^{L}(p,T)]p+\\
		&&- \bar{s}[F_5(s,T)^R-F_5^{L}(p,T)]\} \mathcal{Q}_s(p)^{-1},
	\end{eqnarray*}
	where $ \mathcal{Q}_s(p)=p^2-2s_0p+|s|^2$.
\end{lemma}

\begin{theorem}[Product rule for the $F$-functional calculus for $n=5$]
	\label{prodo}
	Let $T \in \mathcal{BC}(V_5)$. We assume $f \in \mathcal{N}(\sigma_S(T))$ and $g \in \mathcal{SH}_L(\sigma_S(T))$, then we have
	\[
\begin{split}
&\Delta^2(fg)(T)= \Delta^2(f)(T) g(T)+f(T) \Delta^2(g)(T)+\Delta(f)(T) \Delta (g)(T)\\
&\qquad\qquad\qquad \qquad \qquad \qquad \qquad \qquad \qquad \qquad \qquad -\mathcal{D} \Delta (f)(T) \mathcal{D}(g)(T)-\mathcal{D}(f)(T)\Delta \mathcal{D}(g)(T).
	\end{split}
	\]
\end{theorem}
\begin{proof}
	Let $G_1$ and $G_2$ be two bounded slice Cauchy domain such that contain $\sigma_S(T)$ and $ \bar{G}_1 \subset G_2$, with $\bar{G}_2 \subset dom(f) \cap dom(g)$. We choose $p \in \partial(G_1 \cap \mathbb{C}_J)$ and $s \in \partial (G_2 \cap \mathbb{C}_J)$.
	\\ For every $I \in \mathbb{S}$, from the definitions of $F$-functional calculus, $SC$-functional calculus, holomorphic Cliffordian functional calculus of order 1 and $ \ell+1$-harmonic functional calculus ($ 0 \leq \ell \leq 1$) we get
	\[
	\begin{split}
		&\Delta^2(f)(T) g(T)+f(T) \Delta^2(g)(T)+\Delta(f)(T) \Delta (g)(T)-\mathcal{D} \Delta (f)(T) \mathcal{D}(g)(T)-\mathcal{D}(f)(T)\Delta \mathcal{D}(g)(T)\\
		&= \frac{1}{(2 \pi)^2} \int_{\partial(G_2 \cap \mathbb{C}_J)} f(s)ds_I F_5^R(s,T) \int_{\partial(G_1 \cap \mathbb{C}_J)} S^{-1}_L(p,T) dp_I g(p)\\
		&+\frac{1}{(2 \pi)^2} \int_{\partial(G_2 \cap \mathbb{C}_J)} f(s)ds_I S^{-1}_R(s,T) \int_{\partial(G_1 \cap \mathbb{C}_J)} F_5^L(p,T) dp_I g(p)\\
		&+\frac{16}{\pi^2} \int_{\partial(G_2 \cap \mathbb{C}_J)} f(s)ds_I \mathcal{Q}_{c,s}(T)^{-1} S^{-1}_R(s,T) \int_{\partial(G_2 \cap \mathbb{C}_J)} S^{-1}_L(p,T) \mathcal{Q}_{p,T}(T)^{-1}dp_I g(p)\\
		& + \frac{16}{\pi^2}\int_{\partial(G_2 \cap \mathbb{C}_J)} f(s)ds_I \mathcal{Q}_{c,s}(T)^{-2}  \int_{\partial(G_2 \cap \mathbb{C}_J)}  \mathcal{Q}_{c,p}(T)^{-1}dp_I g(p)\\
		& + \frac{16}{\pi^2}\int_{\partial(G_2 \cap \mathbb{C}_J)} f(s)ds_I \mathcal{Q}_{c,s}(T)^{-1}  \int_{\partial(G_2 \cap \mathbb{C}_J)}  \mathcal{Q}_{c,p}(T)^{-2}dp_I g(p)\\
		&= \frac{1}{(2 \pi)^2}\int_{\partial(G_2 \cap \mathbb{C}_J)} \int_{\partial(G_1 \cap \mathbb{C}_J)} f(s) ds_I\{ F_5^R(s,T)S^{-1}_L(s,T)+S^{-1}_R(s,T) F_5^L(p,T)+\\
		&+2^6[ \mathcal{Q}_{c,s}(T)^{-1}S^{-1}_R(s,T) S^{-1}_{L}(p,T) \mathcal{Q}_{c,p}(T)^{-1}+ \mathcal{Q}_{c,s}(T)^{-2} \mathcal{Q}_{c,p}(T)^{-1}+\\
		& + \mathcal{Q}_{c,s}(T)^{-1} \mathcal{Q}_{c,p}(T)^{-2}] \} dp_I g(p)
	\end{split}	
	\]
	By Lemma \ref{reseq} we get
	\[
	\begin{split}
		& \Delta^2(f)(T) g(T)+f(T) \Delta^2(g)(T)+\Delta(f)(T) \Delta (g)(T)-\mathcal{D} \Delta (f)(T) \mathcal{D}(g)(T)-\mathcal{D}(g)(T)\Delta \mathcal{D}g(T)\\
		&= \frac{1}{(2 \pi)^2}\int_{\partial(G_2 \cap \mathbb{C}_J)} \int_{\partial(G_1 \cap \mathbb{C}_J)}f(s)) ds_I \{ [F_{5}^R(s,T)- F_5^{L}(p,T)]p- \bar{s}[F_5^R(s,T)-F_5^{L}(p,T)]\}\times\\
		&\qquad\qquad\qquad\qquad\qquad\qquad\qquad\qquad\qquad\qquad\qquad\qquad\qquad\qquad\qquad\qquad\times \mathcal{Q}_s(p)^{-1} \}dp_I g(p) .
	\end{split}
	\]
	Now, since the functions $p \mapsto p \mathcal{Q}_s(p)^{-1}$, $p \mapsto \mathcal{Q}_s(p)^{-1}$ are intrinsic slice hyperholomoprhic on $ \bar{G_1}$ by the Cauchy integral formula we have
	$$ \int_{\partial(G_2 \cap \mathbb{C}_J)} f(s)ds_I \int_{\partial(G_1 \cap \mathbb{C}_J)} F_5^R(s,T)p \mathcal{Q}_s(p)^{-1} dp_I g(p)=0,$$
	$$ \int_{\partial(G_2 \cap \mathbb{C}_J)} f(s)ds_I \int_{\partial(G_1 \cap \mathbb{C}_J)} \bar{s}F_5^R(s,T) \mathcal{Q}_s(p)^{-1} dp_I g(p)=0.$$	
	Therefore we obtain
	\[
	\begin{split}
		& \Delta^2(f)(T) g(T) +f(T) \Delta^2(g)(T)+\Delta(f)(T) \Delta (g)(T)-\mathcal{D} \Delta (f)(T) \mathcal{D}(g)(T)-\mathcal{D}(f)(T)\Delta \mathcal{D}(g)(T)\\
		&=\frac{1}{(2 \pi)^2}\int_{\partial(G_2 \cap \mathbb{C}_J)} f(s)ds_I \int_{\partial(G_1 \cap \mathbb{C}_J)} [ \bar{s} F_5(p,T)^L-F_5^L(p,T)p]\mathcal{Q}_{s}(p)^{-1}dp_I g(p).
	\end{split}
	\]
	From Lemma \ref{comm} with $B:=	F_5^L(p,T)$ we get
	\[
	\begin{split}
		& \Delta^2(f)(T) g(T)+f(T) \Delta^2(g)(T)+\Delta(f)(T) \Delta (g)(T)-\mathcal{D} \Delta (f)(T) \mathcal{D}(g)(T)-\mathcal{D}(f)(T)\Delta \mathcal{D}(g)(T)\\
		&=	\frac{1}{(2 \pi)}\int_{\partial(G_1 \cap \mathbb{C}_J)} F_5^L(p,T) dp_If(p)g(p)\\
		&=	\frac{1}{(2 \pi)}\int_{\partial(G_1 \cap \mathbb{C}_J)} F_5^L(p,T) dp_I(fg)(p)\\
		&= \Delta^2(fg)(T).
	\end{split}	
	\]

\end{proof}

\section{Appendix: visualization of all possible fine structures in dimension five}

In this appendix we show all the possible fine structures in dimension five. Firstly, we recall the symbols of the classes of functions involved
\newline
\newline
$ \mathcal{ABH}(\Omega_D)$: axially bi-harmonic functions,
\newline
$ \mathcal{ACH}_1(\Omega)$: axially Cliffordian holomorphic functions of order 1
\newline
$ \mathcal{AH}(\Omega_D)$: axially harmonic functions,
\newline
$ \mathcal{AP}_2(\Omega_{D})$: axially polyanalytic of order 2,
\newline
$ \mathcal{ACH}_1(\Omega_D)$: axially anti Cliffordian of order 1,
\newline
$ \mathcal{ACP}_{(1,2)}$: axially polyanalytic Cliffordian of order $(1,2)$,
\newline
$ \mathcal{AP}_3(\Omega_{D})$: axially polyanalytic of order 3.
\newline
\newline
If we apply first the Dirac operator we have the following diagram
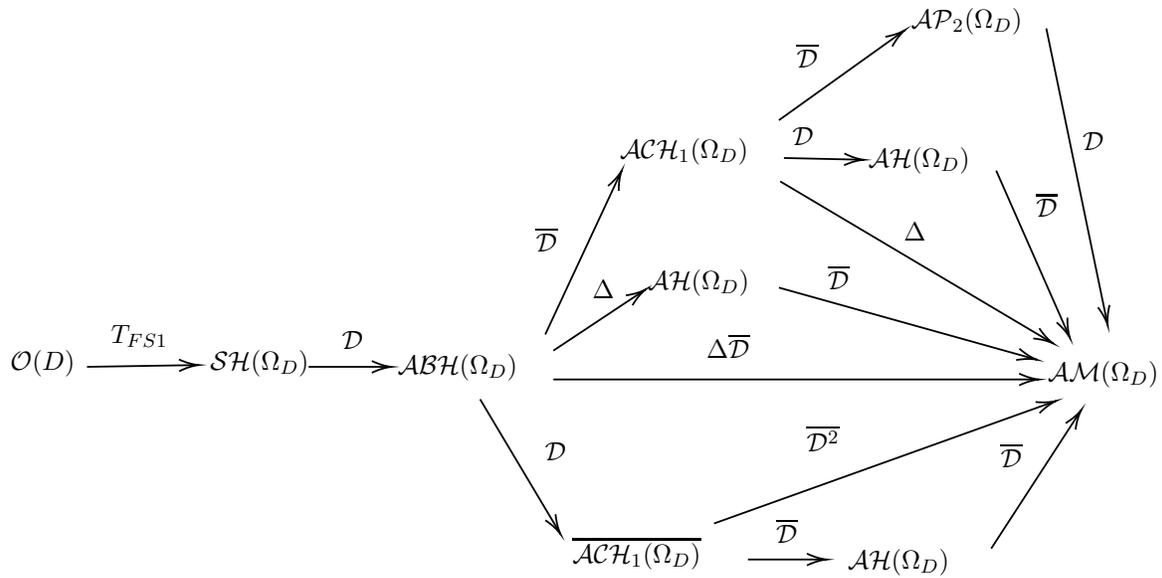
\begin{figure}[H]
\centering
\resizebox{0.95\textwidth}{!}{%
\tikzset{every picture/.style={line width=0.75pt}} 

\begin{tikzpicture}[x=0.75pt,y=0.75pt,yscale=-1,xscale=1]
	
	\draw    (270.16,240) -- (313.78,313.2) ;
	\draw [shift={(314.8,314.92)}, rotate = 239.21] [color={rgb, 255:red, 0; green, 0; blue, 0 }  ][line width=0.75]    (10.93,-3.29) .. controls (6.95,-1.4) and (3.31,-0.3) .. (0,0) .. controls (3.31,0.3) and (6.95,1.4) .. (10.93,3.29)   ;
	\draw    (310.81,229) -- (578.24,229) ;
	\draw [shift={(580.24,229)}, rotate = 180] [color={rgb, 255:red, 0; green, 0; blue, 0 }  ][line width=0.75]    (10.93,-3.29) .. controls (6.95,-1.4) and (3.31,-0.3) .. (0,0) .. controls (3.31,0.3) and (6.95,1.4) .. (10.93,3.29)   ;
	\draw    (310.81,213) -- (360.33,180.11) ;
	\draw [shift={(362,179)}, rotate = 146.41] [color={rgb, 255:red, 0; green, 0; blue, 0 }  ][line width=0.75]    (10.93,-3.29) .. controls (6.95,-1.4) and (3.31,-0.3) .. (0,0) .. controls (3.31,0.3) and (6.95,1.4) .. (10.93,3.29)   ;
	\draw    (306.32,204) -- (347.89,114.81) ;
	\draw [shift={(348.73,113)}, rotate = 114.99] [color={rgb, 255:red, 0; green, 0; blue, 0 }  ][line width=0.75]    (10.93,-3.29) .. controls (6.95,-1.4) and (3.31,-0.3) .. (0,0) .. controls (3.31,0.3) and (6.95,1.4) .. (10.93,3.29)   ;
	\draw    (52,222) -- (112,221.03) ;
	\draw [shift={(114,221)}, rotate = 179.08] [color={rgb, 255:red, 0; green, 0; blue, 0 }  ][line width=0.75]    (10.93,-3.29) .. controls (6.95,-1.4) and (3.31,-0.3) .. (0,0) .. controls (3.31,0.3) and (6.95,1.4) .. (10.93,3.29)   ;
	\draw    (418.8,329) -- (460.46,329) ;
	\draw [shift={(462.46,329)}, rotate = 180] [color={rgb, 255:red, 0; green, 0; blue, 0 }  ][line width=0.75]    (10.93,-3.29) .. controls (6.95,-1.4) and (3.31,-0.3) .. (0,0) .. controls (3.31,0.3) and (6.95,1.4) .. (10.93,3.29)   ;
	\draw    (553.8,322.92) -- (602.71,247.59) ;
	\draw [shift={(603.8,245.92)}, rotate = 123] [color={rgb, 255:red, 0; green, 0; blue, 0 }  ][line width=0.75]    (10.93,-3.29) .. controls (6.95,-1.4) and (3.31,-0.3) .. (0,0) .. controls (3.31,0.3) and (6.95,1.4) .. (10.93,3.29)   ;
	\draw    (436,178) -- (577.14,217.46) ;
	\draw [shift={(579.07,218)}, rotate = 195.62] [color={rgb, 255:red, 0; green, 0; blue, 0 }  ][line width=0.75]    (10.93,-3.29) .. controls (6.95,-1.4) and (3.31,-0.3) .. (0,0) .. controls (3.31,0.3) and (6.95,1.4) .. (10.93,3.29)   ;
	\draw    (436.71,119) -- (586.49,207.98) ;
	\draw [shift={(588.21,209)}, rotate = 210.71] [color={rgb, 255:red, 0; green, 0; blue, 0 }  ][line width=0.75]    (10.93,-3.29) .. controls (6.95,-1.4) and (3.31,-0.3) .. (0,0) .. controls (3.31,0.3) and (6.95,1.4) .. (10.93,3.29)   ;
	\draw    (438.77,106) -- (481.15,106.95) ;
	\draw [shift={(483.15,107)}, rotate = 181.29] [color={rgb, 255:red, 0; green, 0; blue, 0 }  ][line width=0.75]    (10.93,-3.29) .. controls (6.95,-1.4) and (3.31,-0.3) .. (0,0) .. controls (3.31,0.3) and (6.95,1.4) .. (10.93,3.29)   ;
	\draw    (436,85) -- (504.23,36.16) ;
	\draw [shift={(505.86,35)}, rotate = 144.41] [color={rgb, 255:red, 0; green, 0; blue, 0 }  ][line width=0.75]    (10.93,-3.29) .. controls (6.95,-1.4) and (3.31,-0.3) .. (0,0) .. controls (3.31,0.3) and (6.95,1.4) .. (10.93,3.29)   ;
	\draw    (584.39,34) -- (618.07,197.04) ;
	\draw [shift={(618.48,199)}, rotate = 258.33] [color={rgb, 255:red, 0; green, 0; blue, 0 }  ][line width=0.75]    (10.93,-3.29) .. controls (6.95,-1.4) and (3.31,-0.3) .. (0,0) .. controls (3.31,0.3) and (6.95,1.4) .. (10.93,3.29)   ;
	\draw    (556.56,113) -- (596.71,204.17) ;
	\draw [shift={(597.52,206)}, rotate = 246.23] [color={rgb, 255:red, 0; green, 0; blue, 0 }  ][line width=0.75]    (10.93,-3.29) .. controls (6.95,-1.4) and (3.31,-0.3) .. (0,0) .. controls (3.31,0.3) and (6.95,1.4) .. (10.93,3.29)   ;
	\draw    (175,222) -- (219,222) ;
	\draw [shift={(221,222)}, rotate = 180] [color={rgb, 255:red, 0; green, 0; blue, 0 }  ][line width=0.75]    (10.93,-3.29) .. controls (6.95,-1.4) and (3.31,-0.3) .. (0,0) .. controls (3.31,0.3) and (6.95,1.4) .. (10.93,3.29)   ;
	\draw    (400,309) -- (587.92,241.59) ;
	\draw [shift={(589.8,240.92)}, rotate = 160.27] [color={rgb, 255:red, 0; green, 0; blue, 0 }  ][line width=0.75]    (10.93,-3.29) .. controls (6.95,-1.4) and (3.31,-0.3) .. (0,0) .. controls (3.31,0.3) and (6.95,1.4) .. (10.93,3.29)   ;
	
	\draw (8,211.4) node [anchor=north west][inner sep=0.75pt]    {$\mathcal{O}( D)$};
	\draw (319.95,315.4) node [anchor=north west][inner sep=0.75pt]    {$\overline{\mathcal{ACH}_{1}( \Omega _{D})}$};
	\draw (223.02,213.4) node [anchor=north west][inner sep=0.75pt]    {$\mathcal{ABH}( \Omega _{D})$};
	\draw (192,200.4) node [anchor=north west][inner sep=0.75pt]    {$\mathcal{D}$};
	\draw (304.7,262.4) node [anchor=north west][inner sep=0.75pt]    {$\mathcal{D}$};
	\draw (472.78,321.4) node [anchor=north west][inner sep=0.75pt]    {$\mathcal{AH}( \Omega _{D})$};
	\draw (432.97,305.4) node [anchor=north west][inner sep=0.75pt]    {$\overline{\mathcal{D}}$};
	\draw (583.81,217.4) node [anchor=north west][inner sep=0.75pt]    {$\mathcal{AM}( \Omega _{D})$};
	\draw (394.26,202.4) node [anchor=north west][inner sep=0.75pt]    {$\Delta \overline{\mathcal{D}}$};
	\draw (361.81,166.4) node [anchor=north west][inner sep=0.75pt]    {$\mathcal{AH}( \Omega _{D})$};
	\draw (330.93,172.4) node [anchor=north west][inner sep=0.75pt]    {$\Delta $};
	\draw (461.79,164.4) node [anchor=north west][inner sep=0.75pt]    {$\overline{\mathcal{D}}$};
	\draw (346.23,93.4) node [anchor=north west][inner sep=0.75pt]    {$\mathcal{ACH}_{1}( \Omega _{D})$};
	\draw (503.88,140.4) node [anchor=north west][inner sep=0.75pt]    {$\Delta $};
	\draw (484.19,98.4) node [anchor=north west][inner sep=0.75pt]    {$\mathcal{AH}( \Omega _{D})$};
	\draw (507.97,21.4) node [anchor=north west][inner sep=0.75pt]    {$\mathcal{AP}_{2}( \Omega _{D})$};
	\draw (557.66,263.4) node [anchor=north west][inner sep=0.75pt]    {$\overline{\mathcal{D}}$};
	\draw (299.81,143.4) node [anchor=north west][inner sep=0.75pt]    {$\overline{\mathcal{D}}$};
	\draw (441.98,87.4) node [anchor=north west][inner sep=0.75pt]    {$\mathcal{D}$};
	\draw (577.29,124.4) node [anchor=north west][inner sep=0.75pt]    {$\overline{\mathcal{D}}$};
	\draw (444.5,41.4) node [anchor=north west][inner sep=0.75pt]    {$\overline{\mathcal{D}}$};
	\draw (603.18,90.4) node [anchor=north west][inner sep=0.75pt]    {$\mathcal{D}$};
	\draw (119,211.4) node [anchor=north west][inner sep=0.75pt]    {$\mathcal{SH}( \Omega _{D})$};
	\draw (64,197.4) node [anchor=north west][inner sep=0.75pt]    {$T_{FS1}$};
	\draw (449.97,253.4) node [anchor=north west][inner sep=0.75pt]    {$\overline{\mathcal{D}^{2}}$};

\end{tikzpicture}
}
\end{figure}
If we apply fist the conjugate of the Dirac operator we get
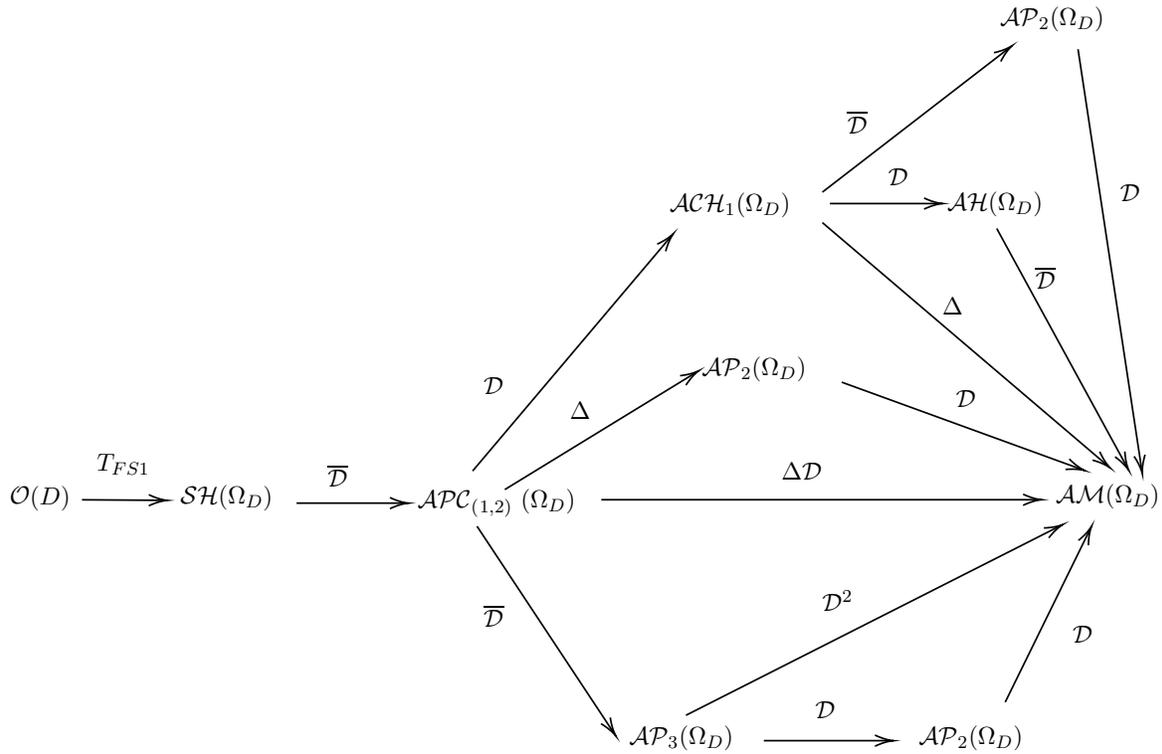
\begin{figure}[H]
\centering
\resizebox{0.95\textwidth}{!}{%
\tikzset{every picture/.style={line width=0.75pt}} 

\begin{tikzpicture}[x=0.75pt,y=0.75pt,yscale=-1,xscale=1]
	
	\draw    (59.22,327.04) -- (107.64,327.35) ;
	\draw [shift={(109.64,327.37)}, rotate = 180.37] [color={rgb, 255:red, 0; green, 0; blue, 0 }  ][line width=0.75]    (10.93,-3.29) .. controls (6.95,-1.4) and (3.31,-0.3) .. (0,0) .. controls (3.31,0.3) and (6.95,1.4) .. (10.93,3.29)   ;
	\draw    (186,329) -- (250.08,329) ;
	\draw [shift={(252.08,329)}, rotate = 180] [color={rgb, 255:red, 0; green, 0; blue, 0 }  ][line width=0.75]    (10.93,-3.29) .. controls (6.95,-1.4) and (3.31,-0.3) .. (0,0) .. controls (3.31,0.3) and (6.95,1.4) .. (10.93,3.29)   ;
	\draw    (366,327) -- (624.08,327) ;
	\draw [shift={(626.08,327)}, rotate = 180] [color={rgb, 255:red, 0; green, 0; blue, 0 }  ][line width=0.75]    (10.93,-3.29) .. controls (6.95,-1.4) and (3.31,-0.3) .. (0,0) .. controls (3.31,0.3) and (6.95,1.4) .. (10.93,3.29)   ;
	\draw    (290,310) -- (407.29,171.19) ;
	\draw [shift={(408.58,169.67)}, rotate = 130.2] [color={rgb, 255:red, 0; green, 0; blue, 0 }  ][line width=0.75]    (10.93,-3.29) .. controls (6.95,-1.4) and (3.31,-0.3) .. (0,0) .. controls (3.31,0.3) and (6.95,1.4) .. (10.93,3.29)   ;
	\draw    (292.58,342.67) -- (373.48,465) ;
	\draw [shift={(374.58,466.67)}, rotate = 236.52] [color={rgb, 255:red, 0; green, 0; blue, 0 }  ][line width=0.75]    (10.93,-3.29) .. controls (6.95,-1.4) and (3.31,-0.3) .. (0,0) .. controls (3.31,0.3) and (6.95,1.4) .. (10.93,3.29)   ;
	\draw    (309,321) -- (421.88,251.8) ;
	\draw [shift={(423.58,250.75)}, rotate = 148.49] [color={rgb, 255:red, 0; green, 0; blue, 0 }  ][line width=0.75]    (10.93,-3.29) .. controls (6.95,-1.4) and (3.31,-0.3) .. (0,0) .. controls (3.31,0.3) and (6.95,1.4) .. (10.93,3.29)   ;
	\draw    (508,257.5) -- (649.7,309.07) ;
	\draw [shift={(651.58,309.75)}, rotate = 200] [color={rgb, 255:red, 0; green, 0; blue, 0 }  ][line width=0.75]    (10.93,-3.29) .. controls (6.95,-1.4) and (3.31,-0.3) .. (0,0) .. controls (3.31,0.3) and (6.95,1.4) .. (10.93,3.29)   ;
	\draw    (496.58,163.17) -- (665.06,306.87) ;
	\draw [shift={(666.58,308.17)}, rotate = 220.46] [color={rgb, 255:red, 0; green, 0; blue, 0 }  ][line width=0.75]    (10.93,-3.29) .. controls (6.95,-1.4) and (3.31,-0.3) .. (0,0) .. controls (3.31,0.3) and (6.95,1.4) .. (10.93,3.29)   ;
	\draw    (501,151.92) -- (565.58,151.92) ;
	\draw [shift={(567.58,151.92)}, rotate = 180] [color={rgb, 255:red, 0; green, 0; blue, 0 }  ][line width=0.75]    (10.93,-3.29) .. controls (6.95,-1.4) and (3.31,-0.3) .. (0,0) .. controls (3.31,0.3) and (6.95,1.4) .. (10.93,3.29)   ;
	\draw    (496.58,145.17) -- (607,59.89) ;
	\draw [shift={(608.58,58.67)}, rotate = 142.32] [color={rgb, 255:red, 0; green, 0; blue, 0 }  ][line width=0.75]    (10.93,-3.29) .. controls (6.95,-1.4) and (3.31,-0.3) .. (0,0) .. controls (3.31,0.3) and (6.95,1.4) .. (10.93,3.29)   ;
	\draw    (599.58,166.67) -- (676.63,307.91) ;
	\draw [shift={(677.58,309.67)}, rotate = 241.39] [color={rgb, 255:red, 0; green, 0; blue, 0 }  ][line width=0.75]    (10.93,-3.29) .. controls (6.95,-1.4) and (3.31,-0.3) .. (0,0) .. controls (3.31,0.3) and (6.95,1.4) .. (10.93,3.29)   ;
	\draw    (647.58,60.67) -- (685.28,309.69) ;
	\draw [shift={(685.58,311.67)}, rotate = 261.39] [color={rgb, 255:red, 0; green, 0; blue, 0 }  ][line width=0.75]    (10.93,-3.29) .. controls (6.95,-1.4) and (3.31,-0.3) .. (0,0) .. controls (3.31,0.3) and (6.95,1.4) .. (10.93,3.29)   ;
	\draw    (462,469.42) -- (536.58,469.42) ;
	\draw [shift={(538.58,469.42)}, rotate = 180] [color={rgb, 255:red, 0; green, 0; blue, 0 }  ][line width=0.75]    (10.93,-3.29) .. controls (6.95,-1.4) and (3.31,-0.3) .. (0,0) .. controls (3.31,0.3) and (6.95,1.4) .. (10.93,3.29)   ;
	\draw    (604.58,446.67) -- (654.7,344.46) ;
	\draw [shift={(655.58,342.67)}, rotate = 116.12] [color={rgb, 255:red, 0; green, 0; blue, 0 }  ][line width=0.75]    (10.93,-3.29) .. controls (6.95,-1.4) and (3.31,-0.3) .. (0,0) .. controls (3.31,0.3) and (6.95,1.4) .. (10.93,3.29)   ;
	\draw    (416,454.42) -- (636.8,342.57) ;
	\draw [shift={(638.58,341.67)}, rotate = 153.14] [color={rgb, 255:red, 0; green, 0; blue, 0 }  ][line width=0.75]    (10.93,-3.29) .. controls (6.95,-1.4) and (3.31,-0.3) .. (0,0) .. controls (3.31,0.3) and (6.95,1.4) .. (10.93,3.29)   ;
	
	\draw (15.09,316.97) node [anchor=north west][inner sep=0.75pt]  [rotate=-0.38] [align=left] {$\displaystyle \mathcal{O}( D)$};
	\draw (66.72,298.35) node [anchor=north west][inner sep=0.75pt]  [rotate=-0.38]  {$T_{FS1}$};
	\draw (116.07,316.79) node [anchor=north west][inner sep=0.75pt]  [rotate=-0.38]  {$\mathcal{SH}( \Omega _{D})$};
	\draw (203,304.9) node [anchor=north west][inner sep=0.75pt]    {$\overline{\mathcal{D}}$};
	\draw (256,318.9) node [anchor=north west][inner sep=0.75pt]    {$\mathcal{APC}_{(}{}_{1,2) \ }( \Omega _{D})$};
	\draw (633,316.4) node [anchor=north west][inner sep=0.75pt]    {$\mathcal{AM}( \Omega _{D})$};
	\draw (471,303.4) node [anchor=north west][inner sep=0.75pt]    {$\Delta \mathcal{D}$};
	\draw (295,252.9) node [anchor=north west][inner sep=0.75pt]    {$\mathcal{D}$};
	\draw (295,387.9) node [anchor=north west][inner sep=0.75pt]    {$\overline{\mathcal{D}}$};
	\draw (346,267.9) node [anchor=north west][inner sep=0.75pt]    {$\Delta $};
	\draw (424,240.9) node [anchor=north west][inner sep=0.75pt]    {$\mathcal{AP}_{2}( \Omega _{D})$};
	\draw (574,260.9) node [anchor=north west][inner sep=0.75pt]    {$\mathcal{D}$};
	\draw (381,458.9) node [anchor=north west][inner sep=0.75pt]    {$\mathcal{AP}_{3}( \Omega _{D})$};
	\draw (405,143.9) node [anchor=north west][inner sep=0.75pt]    {$\mathcal{ACH}_{1}( \Omega _{D})$};
	\draw (566,206.32) node [anchor=north west][inner sep=0.75pt]    {$\Delta $};
	\draw (534,129.9) node [anchor=north west][inner sep=0.75pt]    {$\mathcal{D}$};
	\draw (510,95.9) node [anchor=north west][inner sep=0.75pt]    {$\overline{\mathcal{D}}$};
	\draw (600,33.9) node [anchor=north west][inner sep=0.75pt]    {$\mathcal{AP}_{2}( \Omega _{D})$};
	\draw (569,143.4) node [anchor=north west][inner sep=0.75pt]    {$\mathcal{AH}( \Omega _{D})$};
	\draw (621,187.9) node [anchor=north west][inner sep=0.75pt]    {$\overline{\mathcal{D}}$};
	\draw (671,138.9) node [anchor=north west][inner sep=0.75pt]    {$\mathcal{D}$};
	\draw (552,458.82) node [anchor=north west][inner sep=0.75pt]    {$\mathcal{AP}_{2}( \Omega _{D})$};
	\draw (491,444.9) node [anchor=north west][inner sep=0.75pt]    {$\mathcal{D}$};
	\draw (643,399.9) node [anchor=north west][inner sep=0.75pt]    {$\mathcal{D}$};
	\draw (495,378.9) node [anchor=north west][inner sep=0.75pt]    {$\mathcal{D}^{2}$};

\end{tikzpicture}
}
\end{figure}

Finally, all the other possible fine structures are given by the diagram:
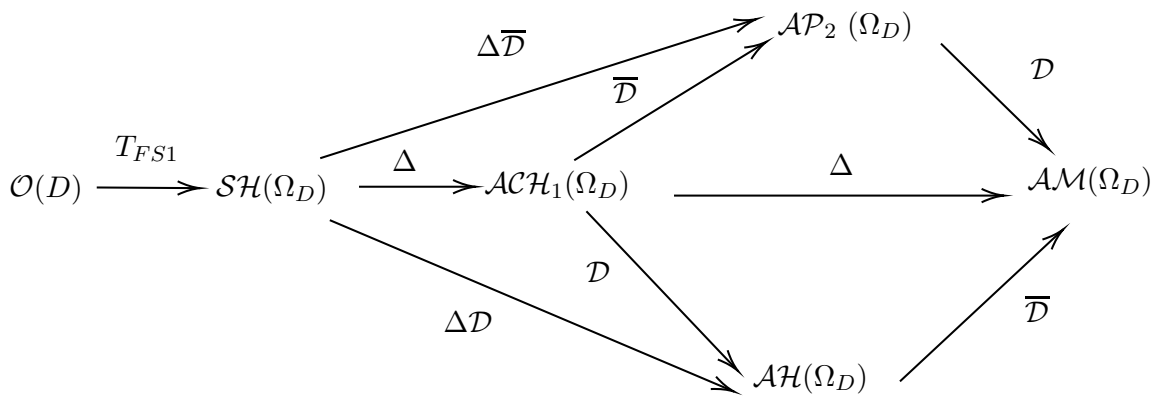
\begin{figure}[H]
\centering
\resizebox{0.95\textwidth}{!}{%
\tikzset{every picture/.style={line width=0.75pt}} 

\begin{tikzpicture}[x=0.75pt,y=0.75pt,yscale=-1,xscale=1]
	
	\draw    (62.22,182.04) -- (110.64,182.35) ;
	\draw [shift={(112.64,182.37)}, rotate = 180.37] [color={rgb, 255:red, 0; green, 0; blue, 0 }  ][line width=0.75]    (10.93,-3.29) .. controls (6.95,-1.4) and (3.31,-0.3) .. (0,0) .. controls (3.31,0.3) and (6.95,1.4) .. (10.93,3.29)   ;
	\draw    (295.3,168.75) -- (387.6,111.8) ;
	\draw [shift={(389.3,110.75)}, rotate = 148.32] [color={rgb, 255:red, 0; green, 0; blue, 0 }  ][line width=0.75]    (10.93,-3.29) .. controls (6.95,-1.4) and (3.31,-0.3) .. (0,0) .. controls (3.31,0.3) and (6.95,1.4) .. (10.93,3.29)   ;
	\draw    (301.24,193.67) -- (373.92,270.3) ;
	\draw [shift={(375.3,271.75)}, rotate = 226.51] [color={rgb, 255:red, 0; green, 0; blue, 0 }  ][line width=0.75]    (10.93,-3.29) .. controls (6.95,-1.4) and (3.31,-0.3) .. (0,0) .. controls (3.31,0.3) and (6.95,1.4) .. (10.93,3.29)   ;
	\draw    (190.24,181.67) -- (245.08,181.67) ;
	\draw [shift={(247.08,181.67)}, rotate = 180] [color={rgb, 255:red, 0; green, 0; blue, 0 }  ][line width=0.75]    (10.93,-3.29) .. controls (6.95,-1.4) and (3.31,-0.3) .. (0,0) .. controls (3.31,0.3) and (6.95,1.4) .. (10.93,3.29)   ;
	\draw    (344,186) -- (503.08,186) ;
	\draw [shift={(505.08,186)}, rotate = 180] [color={rgb, 255:red, 0; green, 0; blue, 0 }  ][line width=0.75]    (10.93,-3.29) .. controls (6.95,-1.4) and (3.31,-0.3) .. (0,0) .. controls (3.31,0.3) and (6.95,1.4) .. (10.93,3.29)   ;
	\draw    (474.3,111.75) -- (525.65,162.02) ;
	\draw [shift={(527.08,163.42)}, rotate = 224.39] [color={rgb, 255:red, 0; green, 0; blue, 0 }  ][line width=0.75]    (10.93,-3.29) .. controls (6.95,-1.4) and (3.31,-0.3) .. (0,0) .. controls (3.31,0.3) and (6.95,1.4) .. (10.93,3.29)   ;
	\draw    (454.3,276.75) -- (532.62,203.78) ;
	\draw [shift={(534.08,202.42)}, rotate = 137.03] [color={rgb, 255:red, 0; green, 0; blue, 0 }  ][line width=0.75]    (10.93,-3.29) .. controls (6.95,-1.4) and (3.31,-0.3) .. (0,0) .. controls (3.31,0.3) and (6.95,1.4) .. (10.93,3.29)   ;
	\draw    (171.3,167.75) -- (382.39,100.36) ;
	\draw [shift={(384.3,99.75)}, rotate = 162.29] [color={rgb, 255:red, 0; green, 0; blue, 0 }  ][line width=0.75]    (10.93,-3.29) .. controls (6.95,-1.4) and (3.31,-0.3) .. (0,0) .. controls (3.31,0.3) and (6.95,1.4) .. (10.93,3.29)   ;
	\draw    (176,198) -- (371.46,280.97) ;
	\draw [shift={(373.3,281.75)}, rotate = 203] [color={rgb, 255:red, 0; green, 0; blue, 0 }  ][line width=0.75]    (10.93,-3.29) .. controls (6.95,-1.4) and (3.31,-0.3) .. (0,0) .. controls (3.31,0.3) and (6.95,1.4) .. (10.93,3.29)   ;
	
	\draw (18.09,173.97) node [anchor=north west][inner sep=0.75pt]  [rotate=-0.38] [align=left] {$\displaystyle \mathcal{O}( D)$};
	\draw (69.72,155.35) node [anchor=north west][inner sep=0.75pt]  [rotate=-0.38]  {$T_{FS1}$};
	\draw (119.07,173.79) node [anchor=north west][inner sep=0.75pt]  [rotate=-0.38]  {$\mathcal{SH}( \Omega _{D})$};
	\draw (299.64,216.53) node [anchor=north west][inner sep=0.75pt]  [rotate=-0.38]  {$\mathcal{D}$};
	\draw (312.89,126.48) node [anchor=north west][inner sep=0.75pt]  [rotate=-0.38]  {$\mathcal{\overline{D}}$};
	\draw (515,170.48) node [anchor=north west][inner sep=0.75pt]    {$\mathcal{AM}( \Omega _{D})$};
	\draw (381,266.4) node [anchor=north west][inner sep=0.75pt]    {$\mathcal{AH}( \Omega _{D})$};
	\draw (205,162.4) node [anchor=north west][inner sep=0.75pt]    {$\Delta $};
	\draw (250,172.4) node [anchor=north west][inner sep=0.75pt]    {$\mathcal{ACH}_{1}( \Omega _{D}) \ $};
	\draw (418,164.4) node [anchor=north west][inner sep=0.75pt]    {$\Delta $};
	\draw (393,94.4) node [anchor=north west][inner sep=0.75pt]    {$\mathcal{AP}_{2} \ ( \Omega _{D})$};
	\draw (516.64,117.53) node [anchor=north west][inner sep=0.75pt]  [rotate=-0.38]  {$\mathcal{D}$};
	\draw (513.89,233.48) node [anchor=north west][inner sep=0.75pt]  [rotate=-0.38]  {$\mathcal{\overline{D}}$};
	\draw (230,241.4) node [anchor=north west][inner sep=0.75pt]    {$\Delta \mathcal{D}$};
	\draw (246,103.4) node [anchor=north west][inner sep=0.75pt]    {$\Delta \overline{\mathcal{D}}$};

\end{tikzpicture}
}
\end{figure}


\begin{thebibliography}{99}


\bibitem{ACGS15} D. Alpay, F. Colombo, J. Gantner, I. Sabadini, \emph{A new resolvent equation for the $S$-functional calculus}, J. Geom. Anal. $25$ $(2015)$, no. $3$, 1939 - 1968.

\bibitem{6SpecThm1}
 D. Alpay, F. Colombo, D.P. Kimsey,
 {\em The spectral theorem for quaternionic unbounded normal operators based on the $S$-spectrum},
 J. Math. Phys., {\bf 57} (2016), no. 2, 023503, 27 pp.







\bibitem{6COFBook}
D. {Alpay}, F. {Colombo},  I. {Sabadini},
 {\em Quaternionic de Branges spaces and characteristic operator function},
 SpringerBriefs in Mathematics, Springer, Cham, 2020/21.

\bibitem{6ACSBOOK} D. {Alpay}, F. {Colombo},  I. {Sabadini},
 {\em Slice Hyperholomorphic Schur Analysis},
 Operator Theory: Advances and Applications, 256. Birkh\"auser/Springer, Cham, 2016. xii+362 pp.

\bibitem{ADS} D. Alpay, K. Diki, I. Sabadini, \emph{On the global operator and Fueter mapping theorem for slice polyanalytic functions},
 Anal. Appl. (Singap.), {\bf 19} (2021), no. 6, 941--964.


\bibitem{A} N. Aronszain, T.M. Creese, L.J. Lipkin, \emph{Polyharmonic functions}, Clarendon Press (1983).

\bibitem{B} H.~Begeher, \emph{Iterated integral opererators in Clifford analysis}, J. Anal. Appl. \textbf{18}, 361-377 (1999).


\bibitem{BF}
 G. Birkhoff, J. von Neumann, {\em The logic of quantum mechanics}, Ann. of
Math., {\bf 37} (1936), 823--843.

\bibitem{B1976} F. Brackx, \emph{On (k)-monogenic functions of a quaternion variable}, Function theoretic methods in differential equations, 22–44. Res. Notes in Math. \textbf{8}, Pitman, London (1976).

\bibitem{DB1978} F. Brackx, R. Delanghe, \emph{Hypercomplex Function Theory and Hilbert Modules with Reproducing Kernel}, proceedings of the London Mathematical Society. s3-37, 545–576. (1978).


\bibitem{6CCKS}
 P. Cerejeiras, F. Colombo, U. K\"ahler, I.  Sabadini, {\em Perturbation of normal quaternionic operators},
  Trans. Amer. Math. Soc., { \bf 372} (2019), no. 5, 3257--3281.




\bibitem{CDPS} F. Colombo, A. De Martino, S. Pinton, I. Sabadini, \emph{Axially harmonic functions and the harmonic functional calculus on the $S$-specturm}, (arXiv: 2205.08162 ) (submitted).




\bibitem{CDQS} F. Colombo, A. De Martino, T. Qian, I. Sabadini, \emph{The Poisson kernel and the Fourier transform of the slice monogenic Cauchy kernels},
     J. Math. Anal. Appl., {\bf 512} (2022), no. 1, Paper No. 126115.

\bibitem{CDS} F. Colombo, A. De Martino, I. Sabadini, \emph{The $ \mathcal{F}$-resolvent equation and Riesz projectors for the $ \mathcal{F}$-functional calculus}, arXiv 2112.04830.



\bibitem{frac4}
F. Colombo, D. Deniz-Gonzales, S. Pinton,
{\em Fractional powers of vector operators with first order boundary conditions},
J. Geom. Phys., {\bf 151} (2020), 103618, 18 pp.

\bibitem{frac5}
F. Colombo, D. Deniz-Gonzales, S. Pinton,
{\em Non commutative fractional Fourier law in bounded and unbounded domains},
 Complex Anal. Oper. Theory, {\bf 15} (2021), no. 7, Paper No. 114, 27 pp.



\bibitem{CG} F. Colombo, J. Gantner, \emph{Formulations of the $ \mathcal{F}$- functional calculus and some consequences}, Proceedings of the Royal Society of Edinburgh, \textbf{146 A}  (2016), 509-545.

\bibitem{frac1}
F. Colombo and J. Gantner, {\em Fractional powers of vector operators and fractional Fourier's law in a Hilbert space},
 J. Phys. A,  {\bf 51} (2018), 305201 (25pp).


\bibitem{FJBOOK} F. Colombo, J. Gantner,
{\em Quaternionic closed operators, fractional powers and fractional diffusion processes},
Operator Theory: Advances and Applications, 274. Birkh\"auser/Springer, Cham, 2019. viii+322.

\bibitem{CGKBOOK} F. Colombo, J. Gantner, D.P. Kimsey, {\em Spectral theory on the $S$-spectrum for quaternionic operators}, Operator Theory: Advances and Applications, 270.
Birkh\"auser/Springer, Cham, 2018. ix+356 pp.


\bibitem{UNIV}
F. Colombo, J. Gantner, D.P. Kimsey, I. Sabadini,
{\em Universality property of the $S$-functional
calculus, noncommuting matrix variables and Clifford operators}, arXiv:2112.05204.

\bibitem{CLIFST}
F. Colombo, D.P. Kimsey,
{\em The spectral theorem for normal operators on a Clifford module},
  Anal. Math. Phys., {\bf 12} (2022), no. 1, Paper No. 25.


\bibitem{ClifFUN}
F. Colombo, D.P. Kimsey, S. Pinton, I. Sabadini,
{ \em  Slice monogenic functions of a Clifford variable},
Proc. Amer. Math. Soc. Ser. B 8 (2021), 281--296.



\bibitem{CS} F. Colombo, I.Sabadini, \emph{The $\mathcal F$-spectrum and the $\mathcal{SC}$-functional calculus.}, Proc. R. Soc. Edinb. A, {\bf 142} (2012), 479-500.


\bibitem{6css} F. Colombo, I. Sabadini, D.C. Struppa, {\em Noncommutative functional calculus. Theory and applications of slice hyperholomorphic functions},
     Progress in Mathematics, 289. Birkh\"auser/Springer Basel AG, Basel, 2011. vi+221 pp.

\bibitem{CSS3} F. Colombo, I.Sabadini, D.C. Struppa, \emph{Michele Sce's Works in Hypercomplex Analysis. A Translation with Commentaries}, Birkhäuser/Springer Basel AG, Basel, 2020.

\bibitem{CSS10} F. Colombo, I. Sabadini, F. Sommen, \emph{The Fueter mapping theorem in integral form and the $F$-functional calculus}, Math. Methods Appl. Sci., \textbf{33} $(2010)$, no. $17$, 2050-2066.

\bibitem{DMD3} A. De Martino, K. Diki, \emph{Generalized Appell polynomials and Fueter-Bargmann transforms in the quaternionic setting}, (arXiv 2112.15116) (submitted)


\bibitem{CDPS1}  A. De Martino, S. Pinton,  \emph{Axially polyanalytic functions of order 2 and their functional calculus on the $S$-specturm}, preprint 2022.


\bibitem{LR} G. Laville, I.Ramadanoff, \emph{Holomorphic Cliffordian functions}, Adv. Appl. Clifford Algebr., \textbf{8}(2)	 (1998), 323-340.





\bibitem{F} R. Fueter, \emph{Die Funktionentheorie der Differentialgleichungen $\Delta u = 0$ und $\Delta\Delta u = 0$ mit vier reellen Variablen},  Comm. Math. Helv., {\bf 7} (1934), 307-330.

    \bibitem{JONAQS}
J. Gantner, {\em On the equivalence of complex and quaternionic quantum mechanics},
 Quantum Stud. Math. Found.,{\bf 5} (2018), no. 2, 357--390.

\bibitem{6JONAME}
 J. Gantner, {\em  Operator Theory on One-Sided Quaternionic Linear Spaces: Intrinsic S-Functional Calculus and Spectral Operators},
 Mem. Amer. Math. Soc., {\bf 267} (2020), no. 1297, iii+101 pp.

\bibitem{GP} R. Ghiloni, A. Perotti, {\em Slice regular functions on real alternative algebras}, Adv. Math., {\bf 226} (2011), 1662--1691.


\bibitem{J} B. Jefferies, {\em Spectral properties of noncommuting operators}, Lecture Notes in Mathematics, 1843, Springer-Verlag, Berlin, 2004.

\bibitem{JM}   B. Jefferies, A. McIntosh, J. Picton-Warlow,  {\em The monogenic functional calculus}, Studia Math., {\bf 136} (1999), 99-119.


\bibitem{Dixan} D. Pena-Pena, \emph{Cauchy Kowalevski extensions, Fueter's theorems and boundary values of special systems in Clifford analysis}, PhD Dissertation, Gent, 2008.

\bibitem{Q1}T. Qian, \emph{Generalization of Fueters result to $\mathbb{R}^{n+1}$}, Rend. Mat. Acc. Lincei, \textbf{9} (1997), 111-117.

\bibitem{S} M. Sce, \emph{Osservazioni sulle serie di potenze nei moduli quadratici}, Atti Accad. Naz. Lincei. Rend. CI. Sci.Fis. Mat. Nat. \textbf{23} (1957), 220-225.

\bibitem{TAOBOOK} T. Qian, P. Li, \emph{Singular integrals and Fourier theory on Lipschitz boundaries}, Science Press Beijing, Beijing; Springer, Singapore, 2019. xv+306 pp.

\bibitem{T} N. Théodoresco, \emph{La dérivée aréolaire et ses applications a' la physique mathématique}, Paris, 1931.
\end{thebibliography}
\end{document}